\documentclass[10pt]{article}
\pdfoutput=1

\usepackage[letterpaper]{geometry}
\usepackage{hicss51}
\usepackage{times}
\usepackage[none]{hyphenat}
\usepackage{url}
\usepackage{latexsym}
\usepackage{indentfirst}
\usepackage{graphicx}
\usepackage{amsfonts,amsmath,amssymb}
\usepackage{amsthm}
\usepackage{mathtools, stackengine}
\usepackage{bbm}
\usepackage{mathrsfs}  
\usepackage{stmaryrd}
\usepackage{xcolor}
\usepackage{textcomp, gensymb}
\usepackage{enumitem}
\usepackage{calc}
\usepackage{bbold}
\usepackage{float}
\usepackage{appendix}
\usepackage{hyperref}
\usepackage{wrapfig}
\usepackage{parskip}
\usepackage{caption}
\usepackage[nameinlink, capitalize, noabbrev]{cleveref}

\usepackage[super,sort&compress]{natbib}

\setcitestyle{authoryear,open={(},close={)}}

\setlist[itemize]{leftmargin=*}

\expandafter\def\expandafter\normalsize\expandafter{%
    \normalsize%
    \setlength\abovedisplayskip{5pt}%
    \setlength\belowdisplayskip{6pt}%
    \setlength\abovedisplayshortskip{-2pt}%
    \setlength\belowdisplayshortskip{4pt}%
}

\title{\textbf{On the Impact of Overparameterization on the Training of a \\ Shallow Neural Network in High Dimensions}}

\usepackage{authblk}

\date{}
\author[1, 2]{\large Simon Martin}
\author[1]{\large Francis Bach}
\author[2]{\large Giulio Biroli}
\affil[1]{\small INRIA - Ecole Normale Supérieure,
PSL Research University}
\affil[2]{\small Laboratoire de Physique de l’Ecole Normale Supérieure, ENS, Université PSL, CNRS, \protect\\ Sorbonne Université, Université de Paris, F-75005 Paris, France}

\date{}

\begin{document}

\renewcommand{\.}{\cdot}
\newcommand{\R}{\mathbb{R}}
\renewcommand{\L}{\mathcal{L}}
\renewcommand{\d}{\mathrm{d}}
\newcommand{\E}{\mathbb{E}}
\newcommand{\N}{\mathcal{N}}
\renewcommand{\P}{\mathbb{P}}
\renewcommand{\sp}{\mathrm{Sp}}
\newcommand{\tr}{\mathrm{Tr}}
\newcommand{\1}{\mathbb{1}}
\newcommand{\md}{\mathcal{M}_{d,m}}
\renewcommand{\H}{\mathcal{H}}
\newcommand\ddfrac[2]{\frac{\displaystyle #1}{\displaystyle #2}}

\newtheorem{theorem}{Theorem}[section]
\newtheorem{proposition}{Proposition}[section]
\newtheorem{assumption}{Assumption}[section]
\newtheorem{lemma}{Lemma}[section]

\maketitle
\begin{abstract}
We study the training dynamics of a shallow neural network with quadratic activation functions and quadratic cost in a teacher-student setup. In line with previous works on the same neural architecture, the optimization is performed following the gradient flow on the population risk, where the average over data points is replaced by the expectation over their distribution, assumed to be Gaussian. 

We first derive convergence properties for the gradient flow and quantify the overparameterization that is necessary to achieve a strong signal recovery. Then, assuming that the teachers and the students at initialization form independent orthonormal families, we derive a high-dimensional limit for the flow and show that the minimal overparameterization is sufficient for strong recovery. We verify by numerical experiments that these results hold for more general initializations. 
\end{abstract}

\section{Introduction}
While neural networks have revolutionized various domains such as image recognition \citep{Krizh}, image generation \citep{Goodfellow}, and natural language processing \citep{chatgpt}, a strong gap remains between their practical achievements and the theoretical understanding of their behaviours. In addition to the formal guarantees that such an understanding can provide, new theoretical insights may lead to an improvement of optimization algorithms as well as the development of more robust and reliable models. 

One main obstacle to a general theoretical study of neural networks is the fact that they are characterized by highly non-convex loss functions. As a consequence, trajectories of gradient-based algorithms and their convergence properties are often hard to understand. Indeed, even for shallow architectures, it is known that loss landscapes of neural networks possess spurious local minima in which parameters can be trapped during optimization \citep{safran,choromanska,christof,baity2018comparing}. However, some recent works show that in some limits, either those local minima tend to disappear from the landscape, or gradient-based algorithms avoid them despite their presence. 
This is for instance the case in highly overparameterized networks \citep{chizat,du2,mei} or when optimizing with a very large amount of data \citep{du3,li,tian} or in high-dimensional inference \citep{sarao2019afraid}. Note that those results are often obtained in purely theoretical limits, where the number of neurons or data points go to infinity. A key open question to determine is how many neurons are needed to achieve global optima; this paper provides an answer for an idealized problem.

\subsection{Contributions} 
We investigate the effect of overparameterization on the training of a shallow neural network in a teacher-student setup. Our model is a one-hidden layer neural network with quadratic activation functions. The optimization is performed on the population loss (i.e., in the infinite data limit) under the assumption of Gaussian data points. While we assume without loss of generality that the number of neurons of the teacher network (denoted $m^*$) is less than the dimension $d$, we do not constraint the number of neurons of the student (denoted $m$). In this setup:

\begin{itemize}
    \item In \cref{sec:dynamics} we derive a general solution of the gradient flow depending on an unknown scalar function which is solution of an implicit equation. 
    \item In \cref{sec:CV} we show that the gradient flow always converges to the global minimizer of the loss. In the case where the student network has more neurons $m$ than the teacher, this global optimum corresponds to a perfect recovery of the teacher network. In the overparameterized case, i.e., $m \geq m^*$, we derive tight convergence rates for the gradient flow. 
    \item In \cref{sec:highdim}, assuming that the teacher vectors and students at initialization form orthonormal families, drawn from an appropriate distribution, we derive a high-dimensional limit for the flow. In this limit, we show that strong recovery is achieved as soon as there are more students than teachers. A key feature of our analysis is to go beyond the case $m=m^\ast=1$ by letting these two quantities to grow with dimension $d$.
\end{itemize}

\subsection{Related Works} 
\paragraph{Shallow neural networks with quadratic activations.} One hidden-layer neural networks with quadratic activation functions have already been studied in the literature. \citet{Du} as well as \citet{Soltano} focus on the empirical loss and derive landscape properties in the overparameterized case. More precisely, \citet{Du} obtain that for $m \geq d$, the landscape does not admit spurious local minimizers in the case where the output weights are fixed (which corresponds to our setup), and \citet{Soltano} derive the same property for $m \geq 2d$, if the output weights can also be learned. This result is also obtained by \citet{venturi} for the case of the population loss. 

Moreover, \citet{Garmanik} and \citet{mannelli} worked in the exact same setup as ours. The main differences are that \citet{Garmanik} focus on the case where students and teachers have more neurons than the dimension ($m, m^* \geq d$), and \citet{mannelli} only assume that $m \geq d$. In this overparameterized setting, \citet{mannelli} show that the gradient flow on the population loss converges towards optimum and derives a convergence rate. 
They also show rigorously that the gradient flow on the empirical loss leads to a minimizer with optimal prediction for a number of sample larger than $2d$ for $m^*=1$ (and heuristically $(m^*+1)d$ for $m^*>1$).
Finally, \citet{Garmanik}, only considering sub-Gaussian observations, proves that the minimal value of both the empirical and population loss over rank deficient matrices is bounded away from zero with high probability. In addition, they obtain a generalization bound for the weights optimized on the empirical loss. 

\paragraph{Phase retrieval.} Another topic related to ours is the phase retrieval problem, which corresponds to the simpler case where $m = m^* = 1$. This problem has been extensively studied in the literature \citep[see][for a review]{fienup}. As a result of interest, we mention \citet{biroli}, who exhibit a phase transition for the number of observations per dimension in order to achieve strong recovery. This criterion is obtained in the mean-field limit, where the dimension goes jointly to infinity with the number of observations, with a constant ratio. 

\paragraph{High-dimensional limits.} More generally, several learning properties of shallow neural networks have been obtained through high-dimensional limits. For instance, in the case of a one hidden layer neural network, \citet{berthier} obtain a set of low dimensional equations for the gradient flow dynamics in the mean-field limit, and \citet{arnaboldi} derive similar equations for the SGD algorithm with different scaling between the dimension, the number of neurons and the stepsize. Other works rely on the use of statistical physics methods to derive high-dimensional equations for the learning dynamics of neural networks \citep{gabrie2023neural,gamarnik2022disordered,biroli,Mignacco}. 

\section{Setting}
\subsection{Notations}
For a matrix $A \in \R^{d \times m}$, we denote $A^T \in \R^{m \times d}$ its transpose and $\|A\|_F = [ \tr(AA^T)]^{1/2}$ its Frobenius norm. We denote $\mathcal{S}_d(\R)$, $\mathcal{S}_d^+(\R)$, $\mathcal{S}_d^{++}(\R)$ the spaces of $d \times d$ symmetric, positive semi-definite and positive definite matrices. If $\big(E, \langle \., \. \rangle \big)$ is a Euclidean space and $\Phi : E \xrightarrow[]{} \R$ is twice continuously differentiable, we denote $\nabla L(x) \in E$ its gradient and $\d^2 L_x$ its Hessian at $x$, which is a linear self-adjoint map on $E$. We say that $x \in E$ is a critical point of~$L$ if $\nabla L(x) = 0$, and a local minimizer if in addition, $\langle \d^2L_x(h), h \rangle \geq 0$ for all $h \in E$. 

\subsection{Model} \label{subsec:model}
We consider a teacher-student setup where the input $x \in \R^d$ is randomly generated and fed to a teacher network $u^*$ whose output is learned by a student $u$ with the same structure, with:
\begin{equation} \label{eq:predictors}
\begin{aligned}
u^*(x) &= \frac{1}{m^*} \sum_{j=1}^{m^*} (w_j^* \. x)^2 = \tr \big(xx^T W^*W^{*T} \big), \\
u(x) &=\frac{1}{m} \sum_{j=1}^m (w_j \. x)^2 = \tr \big( xx^T W W^T \big),
\end{aligned}
\end{equation}
where $m$ and $m^*$ are the number of neurons of the student and the teacher networks, $(w_j)_{1 \leq j \leq m} \in \big(\R^d)^m$ and $(w_j^*)_{1 \leq j \leq m^*} \in \big(\R^d)^{m^*}$ are the corresponding weights, and $W$, $W^*$ the associated renormalized matrices: 
\begin{equation} \label{eq:matrixvector}
\begin{aligned}
    W^* &= \frac{1}{\sqrt{m^*}} \big( w_1^* \ \big| \dots \big| \ w_{m^*}^* \big) \in \R^{d \times m^*}, \\
 W &= \frac{1}{\sqrt{m}} \big( w_1 \ \big| \dots \big| \ w_m \big) \in \R^{d \times m}.
\end{aligned}
\end{equation}
This common architecture corresponds to a one hidden layer neural network where the output weights are fixed, with quadratic activation functions. 

The error made by the student network is evaluated using the quadratic cost function. In line with previous work on this specific model \citep[see][]{mannelli,Garmanik}, our study exclusively focuses on the population risk, which is obtained by taking the expectation over the data points distribution, assumed to be Gaussian with zero mean and identity covariance matrix: 
\begin{equation}  \label{eq:loss}
\begin{aligned}
    \mathcal{L}(W) &= \frac{1}{4} \E_x \Big[ \tr \Big ( xx^T (W^*W^{*T} - WW^T) \Big)^2 \Big]\\
&\equiv  \frac{1}{2} \big\|\Delta_W\big\|_F^2 + \frac{1}{4} \big[ \tr (\Delta_W)\big]^2,
\end{aligned}
\end{equation}
with $\Delta_W = WW^T - W^* W^{*T} \in \R^{d \times d}$. As shown by \citet[Theorem 3.1]{Garmanik}, equation \eqref{eq:loss} is verified whenever $x$ has i.i.d.~coordinates drawn from a distribution matching the standard Gaussian distribution (with mean 0 and variance 1)  up to the fourth moment.\footnote{\citet{Garmanik} show that in the general case, an extra term appears in the loss, proportional to $\kappa_4 \| \mathrm{diag}(\Delta_W)\|_2^2$, where $\kappa_4$ is the fourth cumulant of the distribution from which the coefficients of $x$ are drawn (which is zero in the Gaussian case), and $\mathrm{diag}(\Delta_W)$ is the vector of the diagonal elements of $\Delta_W$. This leads to a different flow which is not covered by our results.}

Although the loss $\L$ is non convex, it can be written as $\L(W) = F(WW^T)$ where $F$ is convex. Some results derived in \cref{sec:CV} will show that $\L$ enjoys some properties shared by convex functions. For instance, we show that all its local minimizers are in fact global. Such functions depending only on $WW^T$ have already been studied with various approaches, see \citet{edelman,journee,massart}. 

Optimization on $W$ is performed using the gradient flow, which has the form: 
\begin{equation} \label{eq:flow}
\begin{aligned}
    \dot{W} &= - \nabla \mathcal{L}(W) = -2 \Delta_W W - \tr(\Delta_W)W.
\end{aligned}
\end{equation}
Note that, due to the form of $\L$, we do not expect to exactly retrieve  through optimization the students vectors $w_1^*, \dots, w_{m^*}^*$, but rather to obtain a student matrix $W$ such that $WW^T = W^*W^{*T}$. As shown in equation \eqref{eq:predictors}, such a matrix will lead to an optimal predictor. If $W^*$ is of rank $m^*$ and since $W \in \R^{d \times m}$, this is only possible for $m \geq m^*$. Only in this case, one can hope for a convergence of the flow towards strong recovery. 

In the following, we will suppose that the initialization of the flow $W^0 = W(t=0) \in \R^{d \times m}$ as well as the teacher matrix $W^* \in \R^{d \times m^*}$ are random. When necessary, we will make assumptions about their distributions. In \cref{sec:CV}, we will assume these matrices to be drawn from a distribution which is absolutely continuous with respect to the Lebesgue measure on $\R^{d \times m}$ and $\R^{d \times m^*}$. In \cref{sec:highdim}, we will require the initial students and the teachers to be orthonormal families. In general, we will always assume the independence of the matrices $W^0 = W(t=0)$ and~$W^*$.

We restrict our study to the case $m^* \leq d$. Indeed, whenever $m^* > d$, the matrix $Z^* = W^* W^{*T} \in \mathcal{S}_d^+(\R)$ can be determined using only $d$ vectors, thus the optimization problem will be equivalent to the case $m^* = d$. We will assume $Z^*$ to be of rank $m^*$ (that is the teacher vectors form a linearly independent family of $\R^d$), and denote $\mu_1 \geq \dots \geq \mu_{m^*} > 0$ its non-zero eigenvalues, and $v_1^*, \dots, v_{m^*}^* \in \R^d$ the associated orthonormal eigenvectors, so that: 
\begin{equation*}
    Z^* = \sum_{k=1}^{m^*} \mu_k v^*_k(v^*_k)^T.
\end{equation*}
\vspace*{-0.1 cm}
\section{Self-Consistent Solution of the Flow} \label{sec:dynamics}
In this section, we present a first result for the gradient flow dynamics defined in equation \eqref{eq:flow}. This flow is non-linear and associated with a non-convex loss, thus we cannot rely on general results to understand its dynamics. To the exception of \citet{mannelli} who gave an interpretation of the solution of the flow using a stochastic differential equation, and derived convergence properties in the over-parameterized case $m \geq d$, the dynamical properties of the flow in the general case have not been studied before.  

However, a similar problem, the Oja's flow, was treated by \citet{oja}. It corresponds to the gradient flow associated with the loss $\mathcal{L}^{\rm oja}(W) = \frac{1}{2} \| \Delta_W \|_F^2$ (where the second term in equation \eqref{eq:loss} is removed), leading to the differential equation $\dot{W} = - 2 \Delta_W W$. They show that this flow admits a closed-form solution and convergence properties can be deduced from their formula. 

In the following proposition, we obtain a solution of the flow very similar to the one of the Oja's flow. As done by \citet{oja}, this solution is expressed in terms of the matrix $Z(t) = W(t)W(t)^T$. However, in our case the solution also depends on a function $\psi$ which is solution of an implicit equation.

\begin{proposition} \label{Prop:Analytic}
Let $W(t)$ be solution of the flow in equation \eqref{eq:flow} with initial condition $W^0 \in \R^{d \times m}$. Let $Z(t) = W(t)W(t)^T$. Then, for all $t \geq 0$: 
    \begin{equation} \label{eq:sol}
   \begin{aligned}
        Z(t) = M^*(t) \left( I_m \!+\! 4 \int_0^t\! M^*(s)^T M^*(s) \d s \right)^{-1} \!\! M^*(t)^T,
    \end{aligned}
    \end{equation}
with: 
\begin{equation*}
	M^*(t) = e^{-\psi(t)} e^{t \tr(Z^*)} e^{2Z^*t} W^0 \in \R^{d \times m},
\end{equation*}
 and $\psi(t) = \int_0^t \tr \big(W(s)W(s)^T \big) \, \d s$, solution of the equation: 
\begin{equation} \label{eq:selfpsi}
        \psi(t) = \frac{1}{4} \tr \log \left( I_m + 4  \int_0^tM^*(s)^T M^*(s) \d s \right),
    \end{equation}
 where the $\log$ is understood as the spectral map on $\mathcal{S}_m^{++}(\R)$. 
\end{proposition}
\begin{proof}
    To derive equation \eqref{eq:sol}, note that if $W(t) \in \R^{d \times m} $ is solution of equation \eqref{eq:flow}, then $Z(t) = W(t)W(t)^T \in \mathcal{S}_d^+(\R)$ is solution of: 
    \begin{equation*}
        \dot{Z} = 2 \tr(Z^*-Z)Z + 2Z^*Z + 2ZZ^* - 4Z^2.
    \end{equation*}
    The result is  obtained by deriving equation \eqref{eq:sol} along with the definition of $\psi$. For equation \eqref{eq:selfpsi}, we have: 
    \begin{equation*}
        \dot{\psi}(t) = \tr[Z(t)] = \frac{1}{4} \tr \big[ \dot{U}(t) U(t)^{-1} \big],
    \end{equation*}
    with $U(t) = I_m + 4 \int_0^t M^*(s) M^*(s)^T \d s$. Since the map $X \mapsto \tr \log X$ has gradient $X^{-1}$ on $\mathcal{S}_m^{++}(\R)$, this leads to the result by integrating the previous equation. 
\end{proof}

This proposition gives an implicit formula for the solution of the flow $W(t)$, but in terms of the matrix $Z(t) = W(t)W(t)^T$. Obtaining an understanding of $W(t)$ in itself is not needed: as mentioned in subsection \ref{subsec:model}, the signal recovery is achieved whenever $WW^T = W^* W^{*T}$. 

The first consequence of formula \eqref{eq:sol} is that if $W^0$ is of full rank, i.e., $\mathrm{rank}(W^0) = \min(m,d)$, then $W(t)$ stays of full rank along the flow. Thus, due to the lower semicontinuity of the rank, $W_\infty = \lim_{t \to \infty}W(t)$ will be of rank $\leq \min(m,d)$, provided that the limit exists. As a consequence, we cannot hope for a full signal recovery (i.e., $Z(t) \xrightarrow[t \to \infty]{} Z^*$) whenever $m < m^*$. 

Equation \eqref{eq:selfpsi} defines an implicit equation on~$\psi$. This scalar function is the only unknown of the solution obtained in equation \eqref{eq:sol}, therefore, gathering information about its behaviour will be useful to understand the properties of the flow. However, equation \eqref{eq:selfpsi} cannot be solved in closed form in general cases, but it can be analyzed in the high-dimensional limit and solved numerically. 
In \cref{sec:CV}, we first obtain general results on the convergence properties at long times by directly studying the local minimizers of the loss function. We then work out the high-dimensional limit in \cref{sec:highdim}.

\section{Convergence of the Gradient Flow} \label{sec:CV}
We now determine the convergence properties of the gradient flow defined in equation \eqref{eq:flow}. In \cref{PropCV}, we derive the limit of the solution $W(t)$ of the flow as $t \to \infty$. We also mention the rate associated with this convergence, a result we detail in \cref{App:CVRate}. 

The first result makes use of the stable manifold theorem \citep[see][]{SMT}, which states that, under the following assumption, the gradient flow almost surely converges towards a local minimizer of the loss function. 

\begin{assumption} \label{AssumptionInit}
The initialization of the flow $W^0$ is drawn with a distribution which is absolutely continuous with respect to the Lebesgue measure on $\R^{d \times m}$. 
\end{assumption}

\begin{proposition} \label{PropCV}
Let $W(t)$ be solution of the flow \eqref{eq:flow} with initial condition $W^0$. Then, under Assumption \ref{AssumptionInit}, with probability one, $W_\infty = \underset{t \to \infty}{\lim} W(t)$ exists and verifies: 
\begin{itemize}
\item if $m \geq m^*$: $W_\infty W_\infty^T = Z^*$,
\item if $m < m^*$ and the non-zero eigenvalues of $Z^*$ are simple (i.e., $\mu_1 > \dots > \mu_{m^*} > 0$):
\begin{equation*}
 W_\infty W_\infty^T = \sum_{k = 1}^m (\mu_k + \tau) v_k^* (v_k^*)^T,
\end{equation*}
where $\tau$ is defined as:
\begin{equation} \label{eq:tau}
\tau = \frac{1}{m+2} \sum_{k=m+1}^{m^*} \mu_k.
\end{equation}
\end{itemize}
\end{proposition}
This result ensures a strong recovery of the signal by the gradient flow when $m \geq m^*$, which is the minimal overparameterization one needs. Indeed, as mentioned earlier, perfect recovery is impossible when $m < m^*$, due to rank constraints. Instead we recover only the largest eigenvectors. 

We prove this proposition in \cref{ProofCV} using that, under Assumption \ref{AssumptionInit}, one only needs to determine the local minimizers associated with the loss, that is solving the equations: 
\begin{equation*}
    \nabla \L(W) = 0, \hspace{1 cm} \tr \big( \d^2 \L_W(K)K^T \big) \geq 0,
\end{equation*}
for all $K \in \R^{d \times m}$, which is much easier than studying the full flow using the solution in \cref{Prop:Analytic}. The proof is done in two steps: we first characterize the critical points of the loss, i.e., the matrices $W \in \R^{d \times m}$ verifying $\nabla \L(W) = 0$. Then, we select those satisfying the second optimality condition. 

In the case $m < m^*$, the assumption on the simplicity of the eigenvalues of $Z^*$ is not mandatory, but it allows to obtain a single local minimizer up to the representation $W \longmapsto WW^T$. For instance, if all of the non-zero eigenvalues of $Z^*$ are equal to some $\mu > 0$, which happens for instance when the teachers are orthogonal, then for each subset $I \subset \llbracket 1, m^* \rrbracket$ of size $m$, the matrix: 
\begin{equation*}
    \sum_{k \in I} (\mu + \tau) v_k^* v_k^{*T},
\end{equation*}
corresponds to a local minimizer of the loss. Overall, each choice of subset $I$ leads to the same value of the loss, which allows to conclude that every local minimizer of $\L$ is global. For $m \geq m^*$, they are optimal and achieve zero error on the loss. 

Another natural question regarding the long time behaviour of the gradient flow is the understanding of the convergence rate. \cref{PropCVRate} in \cref{App:CVRate} derives them in the overparameterized case $m \geq m^*$, improving the results obtained by \citet{mannelli}. Some numerical experiments, also featured in \cref{App:CVRate} prove our bounds to be tight. 

\section{High-Dimensional Limit} \label{sec:highdim}
We now investigate the high dimensional limit of the flow defined in equation \eqref{eq:flow}. Originally used in statistical physics and mean-field theories where the dimension accounts for the number of degrees of freedom (going to infinity in the thermodynamic limit), high dimensional limits have been widely applied to learning and optimization problems  \citep[see][]{gabrie2023neural,gamarnik2022disordered}. They allow to obtain equations where the main objects (called order parameters in physics) are finite-dimensional and satisfy self-consistent equations. In some cases, one also finds that the relevant random quantities concentrate, thus leading to a deterministic description in the $d\rightarrow \infty$ limit, only depending on the distribution of the randomness associated with the initialization or the signal. We show below that this is indeed what happens for the flow by taking the limit $m,m^*,d \to \infty$ with an appropriate scaling.

In this section, the main goals are to (i) identify the limiting deterministic dynamical equations describing the $d\rightarrow \infty$ limit, (ii) determine the timescale $t_d$ at which convergence occurs, and (iii) characterize the behavior of a scalar quantity accounting for the signal retrieval, which is the overlap between the students and teachers: 
    \begin{equation} \label{eq:overlap}
        \chi(Z,Z^*) = \frac{\big| \tr(ZZ^*) \big|}{\big\|Z\big\|_F \big\|Z^*\big\|_F} \in [0,1],
    \end{equation}
with $Z = WW^T$. This quantity depends on time and dimension, and the first challenge is to derive the relevant scaling between those two parameters.

\subsection{Sampling Assumptions} \label{subsec:sampling}
In all the following, we make the assumption that the families $(w_1^0,  \dots, w_m^0)$ and $(w_1^*, \dots, w_{m^*}^*)$ are orthonormal and drawn independently. We also assume that they are uniformly drawn, which can be achieved by considering the uniform measure on the Stiefel manifold \citep[see][]{gotze2023higher}. From now on, the dependency in the dimension is made explicit when relevant. Thus, defining $U_d^0 = \sqrt{m_d} \, W_d^0$ and $U_d^* = \sqrt{m_d^*} \ W_d^*$ (the initialization $W_d^0$ and the teacher matrix $W_d^*$ are linked to their respective vectors in equation \eqref{eq:matrixvector}), we assume that:
\vspace{0.2cm}
\begin{equation*}
\vspace{0.2cm}
    U_d^{0T} U_d^0 = I_{m_d}, \hspace{1 cm} U_d^{*T} U_d^* = I_{m_d^*}.
\end{equation*}
The orthonormality assumption will help simplify the computations, as the relevant quantities that will be studied in the high dimensional limit will only depend on the matrix $Y_d = U_d^{0T} U_d^* U_d^{*T} U_d^0 \in \R^{m_d \times m_d}$. Indeed, we will make use of the implicit solution for the flow obtained in equation \eqref{eq:sol}, which depends on the matrix $W_d^{0T} \exp \big(4 W_d^* W_d^{*T} t \big) W_d^0$. Under our orthonormality assumption, we have the more convenient formula: 
\vspace{0.1cm}
\begin{equation} \label{eq:exponentialorthonormal}
\vspace{0.1cm}
    U_d^{0T} \exp \big( \lambda U_d^* U_d^{*T} \big) U_d^0 = I_{m_d} + \big( e^{\lambda} - 1 \big) Y_d,
\end{equation}
which is specific to our choice of distribution. Moreover, the limit distribution for the eigenvalues of the random matrix $Y_d$ is well understood \citep[see][]{manova,aubrun,hiai}. 

For obvious dimensional reasons, we necessarily have $m_d, m_d^* \leq d$. As $d \to \infty$, we will assume that $m_d$ and $m_d^*$ diverge but keep a fixed ratio with $d$: 
\begin{equation*}
    \alpha = \lim_{d \to \infty} \frac{m_d}{d} \in \ ]0, 1], \hspace{1 cm} \alpha^* = \lim_{d \to \infty} \frac{m_d^*}{d} \in \ ]0,1]. 
\end{equation*}
Thus, we let the number of teachers and students to vary with the dimension, leading to predictors parameterized by a continuum of neurons. 

In this setting, the matrix $Y_d$ has a convergent empirical spectral distribution, namely there exists a probability measure $\mu$ (depending on $\alpha$, $\alpha^*$) supported on $[0,1]$, such that for any continuous $f : \R \xrightarrow[]{} \R$:
\begin{equation*}
    \frac{1}{m_d} \tr \big( f(Y_d) \big) \xrightarrow[d \to \infty]{} \int f(x) \d \mu(x).
\end{equation*}
This convergence result is presented by \citet[Theorem 1.7]{manova}, and the distribution $\mu$ is known as a limiting distribution in the $\beta-$Jacobi ensemble \citep[see][]{jiang,collins}: 
\begin{equation} \label{eq:manova}
\begin{aligned}
    &\hspace*{-.15cm} \d \mu(x) = \frac{\sqrt{(r_+-x)(x-r_-)}}{2 \pi \alpha x(1-x)} \mathbb{1}_{[r_-, r_+]}(x) \d x  \\
    &\hspace*{.5cm} + \Big(1 - \frac{\alpha^*}{\alpha} \Big)^+ \delta_0(x) + \frac{1}{\alpha} \big( \alpha + \alpha^* - 1 \big)^+ \delta_1(x),
\end{aligned}
\end{equation}
where $u^+ = \max(u, 0)$ and: 
\begin{equation*}
    r_\pm = \left( \sqrt{\alpha(1-\alpha^*)} \pm \sqrt{\alpha^*(1-\alpha)} \right)^2.
\end{equation*}
In \cref{fig:PDF}, we compare the probability density function of $\mu$ (when the coefficients associated with the Dirac measures are zero) and its empirical counterpart when drawing the eigenvalues of $Y_d$ in finite dimension.
\begin{figure} [ht]
    \centering
    \includegraphics[width=\linewidth]{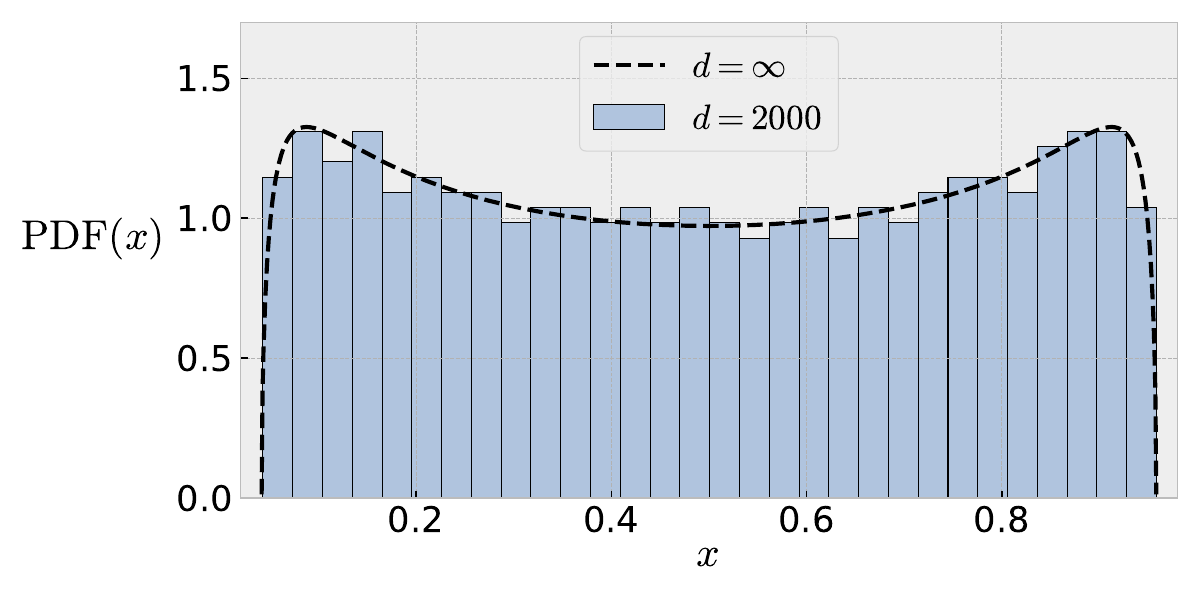}
    \caption{Probability density function of $\mu$ defined in equation \eqref{eq:manova} (dashed line) and histogram of the eigenvalues of $Y_d$ for $d = 2000$ with $30$ bins, with $\alpha = 0.3$, $\alpha^* = 0.5$. }
    \label{fig:PDF}
\end{figure}
This convergence result is a key ingredient to characterize the high-dimensional limit. In fact, as shown in \cref{sec:dynamics}, the solution of the flow derived in equation \eqref{eq:sol} mainly relies on the unknown function $\psi_d$, which is solution of the implicit equation \eqref{eq:selfpsi}. Our first result identifies the high-dimensional limit of $\psi_d$ defined in \cref{Prop:Analytic}.

\subsection{High-Dimensional Limit of the Flow}
In order to investigate signal recovery in the high-dimensional limit, we first determine the limit equations solved by $\psi_d$. The convergence will be shown to happen at timescale $t \sim d$, we therefore allow time to vary with dimension by setting $t_d = \gamma d$, where $\gamma$ is fixed and accounts for a renormalized time variable. We also set:
\begin{equation} \label{eq:phi}
\begin{aligned}
    \phi_d(\gamma) &= \int_0^{\gamma d} \tr \big( W_d(s) W_d(s)^T - W_d^* W_d^{*T} \big) \d s \\ 
    &= \psi_d(\gamma d) - \gamma d,
\end{aligned}
\end{equation}
since under our assumptions, we have $\tr(Z_d^*) = 1$. Note that we cannot hope for a high-dimensional limit for $\psi_d$ in itself: in the case of a signal recovery at finite dimension, one can expect $\psi_d(\gamma d)$ to be of the order~$d$, whereas $\phi_d(\gamma)$ should remain of order 1. That is indeed what we show in the following proposition.  

\begin{proposition} \label{PropPhi}
    Let $W_d(t)$ be solution of the gradient flow in equation \eqref{eq:flow} with initial condition $W_d^0$. Then, with probability one, as $d \to \infty$, $\phi_d$ uniformly converges on any compact of $\R_+$ towards a function $\phi$ which is solution of the equation:
    \begin{equation} \label{eq:selfphi}
    \begin{aligned}
        4\gamma &= \alpha \log F_\phi(\gamma)
        + \big(\alpha + \alpha^* - 1 \big)^+ \log \left( \frac{G_\phi(\gamma)}{F_\phi(\gamma)} \right) \\
        &\hspace{3.2cm}+ \Theta \left( \frac{G_\phi(\gamma)}{F_\phi(\gamma)} - 1 \right),
    \end{aligned}
    \end{equation}
    with:
    \vspace{-0.2cm}
    \begin{equation*} 
    \begin{aligned}
        F_\phi(\gamma) &= 1 + \frac{4}{\alpha} \int_0^\gamma e^{-2 \phi(s)} \d s, \\
        G_\phi(\gamma) &= 1 + \frac{4}{\alpha} \int_0^\gamma e^{-2 \phi(s)} e^{4s / \alpha^*} \d s, \\
        \Theta(u) = \frac{1}{2\pi} \int_{r_-}^{r_+} &\frac{\sqrt{(r_+-x)(x-r_-)}}{x(1-x)} \log(1+ ux) \d x.
    \end{aligned}
    \end{equation*}
\end{proposition}
This result is proven in \cref{ProofPhi}. The idea is to start from equation \eqref{eq:selfpsi} which allows to obtain an implicit equation of $\phi_d$. The key observation is that this equation only depends on the matrix $Y_d$ (this is a consequence of equation \eqref{eq:exponentialorthonormal}).

The equation above, in which the dimension only appears through the parameters $\alpha$ and $\alpha^*$, allows to analyze the high-dimensional limit of the flow. Although it cannot be solved analytically, it is to be compared with the implicit equation \eqref{eq:selfpsi} on $\psi_d$: this new equation is purely deterministic and does not depend on the initialization of the flow. 
\begin{figure}[ht]
    \centering
    \includegraphics[width=\linewidth]{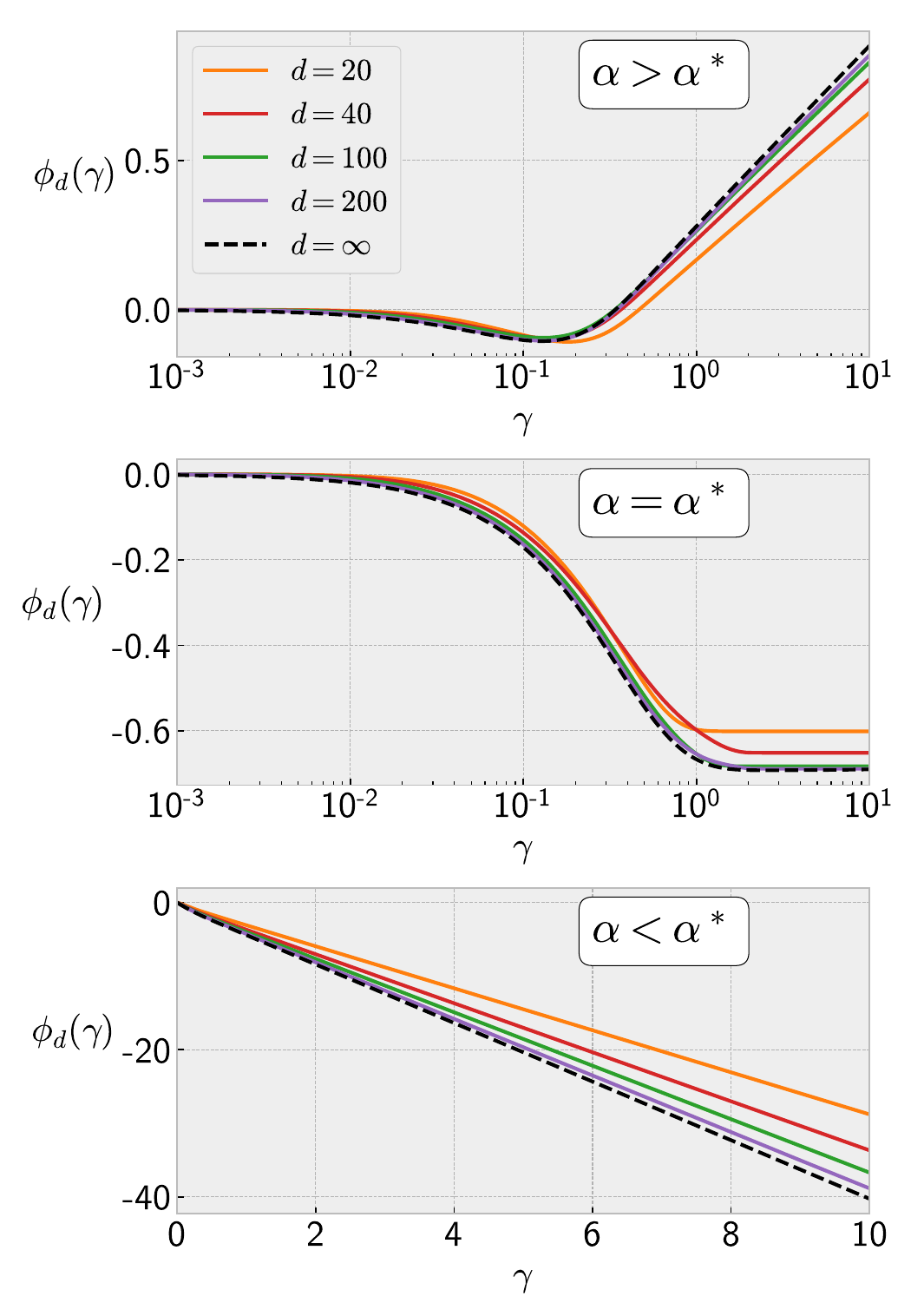}    
    \caption{Simulations of $\phi_d(\gamma)$ (defined in equation \eqref{eq:phi}) with $\gamma = t/d$ for increasing values of $d$. Finite $d$: simulation using gradient descent with stepsize $\eta = 10^{-2}$, orthonormal teachers and students at initialization. $d = \infty$: simulations using discretized version of equation \eqref{eq:selfphi}. Top: $m_d = d/2$, $m_d^* =  d/4$, middle: $m_d = d/2$, $m_d^* =  d/2$, bottom: $m_d = d/4$, $m_d^* = d/2$. Results averaged over 5 simulations. }
    \label{fig1}
\end{figure}
We can obtain an approximate numerical solution of $\phi$ by differentiating this equation with respect to $\gamma$, and obtaining an equation of the form:
\begin{equation*}
    \phi(\gamma) = \Omega \big( F_\phi(\gamma), G_\phi(\gamma), \gamma \big).
\end{equation*}
From the expression of $F_\phi$ and $G_\phi$, this can be numerically handled as a standard differential equation. We have compared this numerical solution to the one obtained by a direct solution of the flow equations in finite dimension. \cref{fig1} displays the evolution of the function $\phi_d$ for increasing values of $d$, as well as the numerical solution of equation \eqref{eq:selfphi}. It shows a quite fast convergence with the dimension (determining the convergence rate as $d \to \infty$ is a challenge that we leave for future studies). \cref{fig1} also showcases very different behaviours for $\phi_d$ depending on the values of $m_d$ and $m_d^*$: logarithmic for $m_d > m_d^*$, converging for $m_d = m_d^*$ and linear for $m_d < m_d^*$. 
Those behaviours of $\phi_d$ in the different regimes can be also understood from finite dimension estimation. In \cref{App:CVRate}, we determine an asymptotic development of the function $\psi(t)$ as $t \to \infty$, which allows to recover the behaviours displayed in \cref{fig1}.

\subsection{High-Dimensional Signal Recovery} \label{subsec:overlap}
We now investigate the question of the signal retrieval in the high-dimensional limit. \cref{PropCV} in \cref{sec:CV} shows that strong recovery as $t \to \infty$ is guaranteed whenever $m \geq m^*$. The question is now whether this criterion still holds as $d \to \infty$. To this end, we introduce the overlap, a scalar quantity describing the alignment between the teachers and students: 
\begin{equation*}
    \chi(Z_d(t), Z_d^*) = \frac{\big| \tr\big(Z_d(t) Z_d^*\big) \big|}{\big\|Z_d(t) \big\|_F \, \big\| Z^*_d \big\|_F},
\end{equation*}
with $Z_d(t) = W_d(t)W_d(t)^T$. A strong signal recovery is obtained whenever $\chi(Z_d, Z_d^*) = 1$, corresponding to a perfect alignment. Moreover, we say that the flow achieves a weak recovery if $\chi(Z_d, Z_d^*)$ is larger than a purely random overlap $\chi_d^0$, i.e., the overlap between two independent projection matrices. In the vector case, where it is commonly used (replacing the trace by the standard inner product on $\R^d$ and $\| \. \|_F$ by the Euclidean norm), this purely random overlap goes to zero as $d \to \infty$. However, in our case, we show that this quantity stays positive in the high dimensional limit. More precisely, $\chi_d^0$ can be expressed as the overlap between the teachers and the students at initialization, which leads to the limit $\chi_d^0 \xrightarrow[]{} \sqrt{\alpha \alpha^*}$, see \cref{App:Overlap}.

In the same spirit as \cref{PropPhi}, we determine the convergence of the overlap between the students and teachers as $d \to \infty$ and at timescale $t_d = \gamma d$. 

\begin{proposition} \label{PropOverlap}
    Let $W_d(t)$ be solution of the gradient flow in equation \eqref{eq:flow} with initial condition $W_d^0$. Then, as $d \to \infty$, almost surely, the function $\chi_d(\gamma) = \chi \big( Z_d(\gamma d), Z_d^* \big)$ uniformly converges on every compact of $\R_+$. Let $\chi(\gamma) = \underset{d \to \infty}{\lim} \chi_d(\gamma)$. We have: 
    \begin{equation*}
        \chi(\gamma) \xrightarrow[\gamma \to \infty]{} \min \left( \sqrt{\frac{\alpha}{\alpha^*}}, 1 \right).
    \end{equation*}
\end{proposition}
We prove this proposition in \cref{ProofOverlap}. The proof is a consequence of the convergence of the empirical measure associated with the eigenvalues of $Y_d$ as well as the result of \cref{PropPhi}. 
\begin{figure} [ht]
    \centering
    \includegraphics[width=\linewidth]{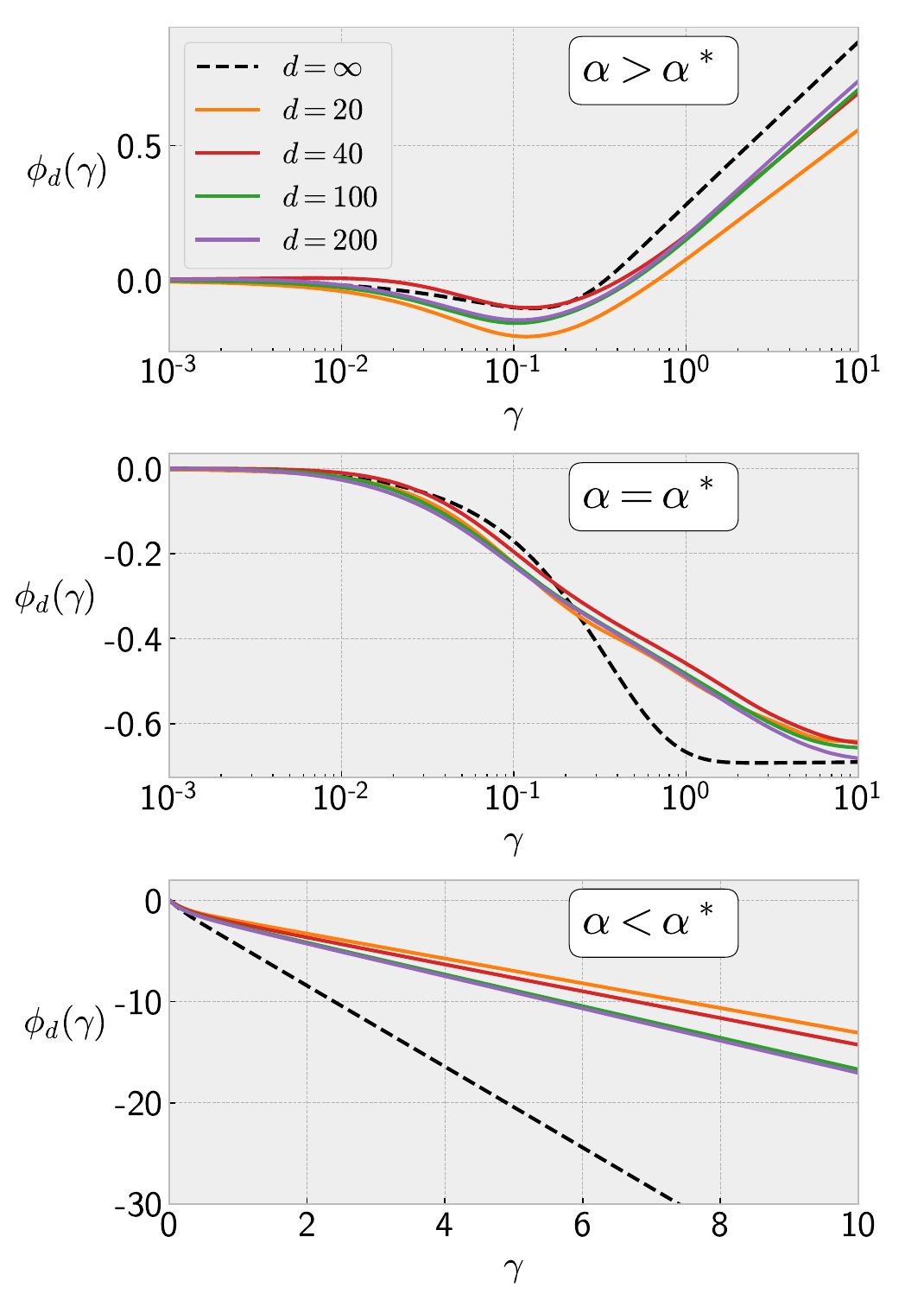}
    \caption{Simulations of $\phi_d(\gamma)$ (defined in equation \eqref{eq:phi}) with $\gamma = t/d$ for increasing values of $d$. Finite~$d$: simulation using gradient descent with stepsize $\eta = 10^{-2}$, and Gaussian $\N(0, I_d/d)$ teachers and students at initialization. $d = \infty$: simulations using discretized version of equation \eqref{eq:selfphi} Top: $m_d =   d/2$, $m_d^* =  d/4$, middle: $m_d = d/2$, $m_d^* = d/2$, bottom: $m_d = d/4$, $m_d^* = d/2$. Results averaged over 5 simulations. }
    \label{fig2}
\end{figure}
This last result gives the behaviour of the overlap in the regime $t \gg d$. On this timescale, we obtain a strong recovery of the signal when $\alpha \geq \alpha^*$, and a weak recovery when $\alpha < \alpha^*$. This corresponds to the same threshold as the one obtained in \cref{PropCV}. More precisely, when the teachers are orthonormal, one can compute the overlap between $Z_d^*$ and the limit $Z_d^\infty = \lim_{t \to \infty} Z_d(t)$ derived in \cref{PropCV} in finite dimension, and obtain the same result as $d \to \infty$, i.e., that the limits $d,t\rightarrow \infty$ commute:
\begin{equation*}
   \lim_{d \to \infty} \lim_{t \to \infty} \chi \big( Z_d(t), Z_d^* \big) = \lim_{\gamma \to \infty} \lim_{d \to \infty} \chi \big(Z_d(\gamma d), Z_d^* \big).
\end{equation*}
Moreover, still in the orthonormal case, one can show that the quantity obtained in the previous proposition realizes the maximum overlap possible for a given number of students. We detail this result in \cref{App:Overlap}.

The previous results are valid for teachers and initial students which are drawn orthonormally. However, we believe it still holds for a larger class of distributions. Indeed, the determination of the convergence rates at finite dimension in \cref{App:CVRate} allows to obtain the asymptotic behaviour of the flow as $t \to \infty$, without any distributional assumption. Interestingly, those behaviours coincide with the ones obtained throughout the section. To support the conjectured broader generality of our results, we show in \cref{fig2} a simulation of the function $\phi_d$ (directly computed from the flow), where both teacher and students are independently drawn from a Gaussian distribution $\N(0, I_d/d)$. This is compared to the approximate numerical solution of equation \eqref{eq:selfphi}, where they are drawn orthonormally. We find again a quite fast convergence for $d \rightarrow \infty$, and a similar qualitative behaviour. However, the infinite dimensional limit of $\phi_d$ appears to be different from the orthonormal case. Generalizing the self-consistent analysis of \cref{sec:highdim} to the Gaussian case is left for future works.

\section{Conclusion and Further Work}
In this paper we presented new theoretical results on the optimization of one-hidden layer neural networks with quadratic activation functions. Focusing on the population loss gradient flow, we derived convergence properties and showed that a slight overparameterization is enough to achieve signal recovery. Then, we derived a high-dimensional limit for the flow and showed that our criterion still holds whenever the initialized students and teachers are orthonormal families. 

\textbf{Further work.} The assumptions we made along this paper leaves several challenges of interest:
\begin{itemize}
    \item As mentioned previously, we believe that the orthonormality assumption we made throughout \cref{sec:highdim} is not mandatory and that the results of this section can be generalized to a larger class of distributions, as suggested by \cref{fig2}.
    \item Our study essentially focused on the optimization of the population loss. The next step is to understand the gradient flow associated with the empirical loss on a finite dataset. This extension has already been studied in several papers \citep[see][]{mannelli, Garmanik, Du}, but very few results were obtained regarding the dynamics of the flow. One promising strategy relies on the use of statistical physics methods, and more precisely the dynamical mean-field theory, which allows to obtain a low-dimensional set of equations describing the dynamics in the limit where the dimension and the number of samples jointly go to infinity \citep[see][]{DMFTGBA, Mignacco}. 
    \item More generally, it is of a high interest to obtain similar results for general activation functions, although several methods we used throughout this work are specific to the quadratic activation. 
\end{itemize}

\medskip

\section*{Acknowledgements}

The authors acknowledge support from the French government under the management of the Agence Nationale de la Recherche as part of the “Investissements d’avenir” program, reference ANR-19-P3IA0001 (PRAIRIE 3IA Institute). SM also thanks Louis-Pierre Chaintron for fruitful mathematical discussions. 

\medskip

\begingroup
    \setlength{\bibsep}{5pt}
    \bibliography{Refs}

\begin{thebibliography}{40}
\providecommand{\natexlab}[1]{#1}
\providecommand{\url}[1]{\texttt{#1}}
\expandafter\ifx\csname urlstyle\endcsname\relax
  \providecommand{\doi}[1]{doi: #1}\else
  \providecommand{\doi}{doi: \begingroup \urlstyle{rm}\Url}\fi

\bibitem[Arnaboldi et~al.(2023)Arnaboldi, Stephan, Krzakala, and Loureiro]{arnaboldi}
L.~Arnaboldi, L.~Stephan, F.~Krzakala, and B.~Loureiro.
\newblock From high-dimensional \& mean-field dynamics to dimensionless {ODEs}: A unifying approach to {SGD} in two-layers networks.
\newblock Technical Report 2302.05882, arXiv, 2023.

\bibitem[Aubrun(2021)]{aubrun}
G.~Aubrun.
\newblock Principal angles between random subspaces and polynomials in two free projections.
\newblock \emph{Confluentes Mathematici}, 13\penalty0 (2):\penalty0 3--10, 2021.

\bibitem[Baity-Jesi et~al.(2018)Baity-Jesi, Sagun, Geiger, Spigler, Ben~Arous, Cammarota, LeCun, Wyart, and Biroli]{baity2018comparing}
M.~Baity-Jesi, L.~Sagun, M.~Geiger, S.~Spigler, G.~Ben~Arous, C.~Cammarota, Y.~LeCun, M.~Wyart, and G.~Biroli.
\newblock Comparing dynamics: Deep neural networks versus glassy systems.
\newblock In \emph{International Conference on Machine Learning}, pages 314--323, 2018.

\bibitem[Ben~Arous et~al.(2006)Ben~Arous, Dembo, and Guionnet]{DMFTGBA}
G.~Ben~Arous, A.~Dembo, and A.~Guionnet.
\newblock Cugliandolo-{Kurchan} equations for dynamics of spin-glasses.
\newblock \emph{Probab. Theory Relat. Fields}, 136\penalty0 (4):\penalty0 619--660, Dec. 2006.

\bibitem[Berthier et~al.(2023)Berthier, Montanari, and Zhou]{berthier}
R.~Berthier, A.~Montanari, and K.~Zhou.
\newblock Learning time-scales in two-layers neural networks.
\newblock Technical Report 2303.00055, arXiv, 2023.

\bibitem[Chizat and Bach(2018)]{chizat}
L.~Chizat and F.~Bach.
\newblock On the global convergence of gradient descent for over-parameterized models using optimal transport.
\newblock \emph{Advances in Neural Information Processing Systems}, 31, 2018.

\bibitem[Choromanska et~al.(2015)Choromanska, Henaff, Mathieu, Ben~Arous, and LeCun]{choromanska}
A.~Choromanska, M.~Henaff, M.~Mathieu, G.~Ben~Arous, and Y.~LeCun.
\newblock The loss surfaces of multilayer networks.
\newblock In \emph{Artificial intelligence and statistics}, pages 192--204, 2015.

\bibitem[Christof and Kowalczyk(2023)]{christof}
C.~Christof and J.~Kowalczyk.
\newblock On the omnipresence of spurious local minima in certain neural network training problems.
\newblock \emph{Constructive Approximation}, June 2023.

\bibitem[Collins(2005)]{collins}
B.~Collins.
\newblock Product of random projections, {Jacobi} ensembles and universality problems arising from free probability.
\newblock \emph{Probability Theory and Related Fields}, 133\penalty0 (3):\penalty0 315--344, Nov. 2005.

\bibitem[Dasgupta and Gupta(2003)]{dasgupta}
S.~Dasgupta and A.~Gupta.
\newblock An elementary proof of a theorem of {J}ohnson and {L}indenstrauss.
\newblock \emph{Random Structures \& Algorithms}, 22\penalty0 (1):\penalty0 60--65, 2003.

\bibitem[Du and Lee(2018)]{Du}
S.~S. Du and J.~D. Lee.
\newblock On the power of over-parametrization in neural networks with quadratic activation.
\newblock In \emph{Proceedings of the {International} {Conference} on {Machine} {Learning}}, pages 1329--1338, 2018.

\bibitem[Du et~al.(2018)Du, Lee, Tian, Singh, and Poczos]{du3}
S.~S. Du, J.~D. Lee, Y.~Tian, A.~Singh, and B.~Poczos.
\newblock Gradient descent learns one-hidden-layer {CNN}: Don’t be afraid of spurious local minima.
\newblock In \emph{Proceedings of the International Conference on Machine Learning}, pages 1339--1348, 2018.

\bibitem[Du et~al.(2019)Du, Zhai, Poczos, and Singh]{du2}
S.~S. Du, X.~Zhai, B.~Poczos, and A.~Singh.
\newblock Gradient descent provably optimizes over-parameterized neural networks.
\newblock In \emph{International Conference on Learning Representations}, 2019.

\bibitem[Edelman et~al.(1998)Edelman, Arias, and Smith]{edelman}
A.~Edelman, T.~A. Arias, and S.~T. Smith.
\newblock The geometry of algorithms with orthogonality constraints.
\newblock \emph{SIAM J. Matrix Anal. \& Appl.}, 20\penalty0 (2):\penalty0 303--353, Jan. 1998.

\bibitem[Fienup(1982)]{fienup}
J.~R. Fienup.
\newblock Phase retrieval algorithms: a comparison.
\newblock \emph{Appl. Opt.}, 21\penalty0 (15):\penalty0 2758, Aug. 1982.

\bibitem[Gabri{\'e} et~al.(2023)Gabri{\'e}, Ganguli, Lucibello, and Zecchina]{gabrie2023neural}
M.~Gabri{\'e}, S.~Ganguli, C.~Lucibello, and R.~Zecchina.
\newblock Neural networks: from the perceptron to deep nets.
\newblock Technical Report 2304.06636, arXiv, 2023.

\bibitem[Gamarnik et~al.(2019)Gamarnik, K{\i}z{\i}lda{\u{g}}, and Zadik]{Garmanik}
D.~Gamarnik, E.~C. K{\i}z{\i}lda{\u{g}}, and I.~Zadik.
\newblock Stationary points of shallow neural networks with quadratic activation function.
\newblock Technical Report 1912.01599, arXiv, 2019.

\bibitem[Gamarnik et~al.(2022)Gamarnik, Moore, and Zdeborov{\'a}]{gamarnik2022disordered}
D.~Gamarnik, C.~Moore, and L.~Zdeborov{\'a}.
\newblock Disordered systems insights on computational hardness.
\newblock \emph{Journal of Statistical Mechanics: Theory and Experiment}, 2022\penalty0 (11):\penalty0 114015, 2022.

\bibitem[Goodfellow et~al.(2014)Goodfellow, Pouget-Abadie, Mirza, Xu, Warde-Farley, Ozair, Courville, and Bengio]{Goodfellow}
I.~J. Goodfellow, J.~Pouget-Abadie, M.~Mirza, B.~Xu, D.~Warde-Farley, S.~Ozair, A.~Courville, and Y.~Bengio.
\newblock Generative adversarial nets.
\newblock In \emph{Advances in Neural Information Processing Systems}, volume~27, 2014.

\bibitem[G{\"o}tze and Sambale(2023)]{gotze2023higher}
F.~G{\"o}tze and H.~Sambale.
\newblock Higher order concentration on stiefel and grassmann manifolds.
\newblock \emph{Electronic Journal of Probability}, 28:\penalty0 1--30, 2023.

\bibitem[Gozalo-Brizuela and Garrido-Merchan(2023)]{chatgpt}
R.~Gozalo-Brizuela and E.~C. Garrido-Merchan.
\newblock Chatgpt is not all you need. a state of the art review of large generative {AI} models.
\newblock Technical Report 2301.04655, arXiv, 2023.

\bibitem[Hiai and Petz(2005)]{hiai}
F.~Hiai and D.~Petz.
\newblock Large deviations for functions of two random projection matrices.
\newblock Technical Report math/0504435, arXiv, 2005.

\bibitem[Jiang(2013)]{jiang}
T.~Jiang.
\newblock Limit theorems for {B}eta-{J}acobi ensembles.
\newblock \emph{Bernoulli}, 19\penalty0 (3):\penalty0 1028--1046, 2013.

\bibitem[Journée et~al.(2010)Journée, Bach, Absil, and Sepulchre]{journee}
M.~Journée, F.~Bach, P.-A. Absil, and R.~Sepulchre.
\newblock Low-rank optimization on the cone of positive semidefinite matrices.
\newblock \emph{SIAM J. Optim.}, 20\penalty0 (5):\penalty0 2327--2351, Jan. 2010.

\bibitem[Krizhevsky et~al.(2017)Krizhevsky, Sutskever, and Hinton]{Krizh}
A.~Krizhevsky, I.~Sutskever, and G.~E. Hinton.
\newblock Imagenet classification with deep convolutional neural networks.
\newblock \emph{Commun. ACM}, 60\penalty0 (6):\penalty0 84–90, may 2017.

\bibitem[Kunisky(2023)]{manova}
D.~Kunisky.
\newblock Generic {MANOVA} limit theorems for products of projections.
\newblock Technical Report 2301.09543, arXiv, 2023.

\bibitem[Li and Yuan(2017)]{li}
Y.~Li and Y.~Yuan.
\newblock Convergence analysis of two-layer neural networks with relu activation.
\newblock \emph{Advances in Neural Information Processing Systems}, 30, 2017.

\bibitem[Mannelli et~al.(2019)Mannelli, Biroli, Cammarota, Krzakala, and Zdeborov{\'a}]{sarao2019afraid}
S.~S. Mannelli, G.~Biroli, C.~Cammarota, F.~Krzakala, and L.~Zdeborov{\'a}.
\newblock Who is afraid of big bad minima? analysis of gradient-flow in spiked matrix-tensor models.
\newblock \emph{Advances in Neural Information Processing Systems}, 32, 2019.

\bibitem[Mannelli et~al.(2020{\natexlab{a}})Mannelli, Biroli, Cammarota, Krzakala, Urbani, and Zdeborov\'{a}]{biroli}
S.~S. Mannelli, G.~Biroli, C.~Cammarota, F.~Krzakala, P.~Urbani, and L.~Zdeborov\'{a}.
\newblock Complex dynamics in simple neural networks: Understanding gradient flow in phase retrieval.
\newblock In \emph{Advances in Neural Information Processing Systems}, 2020{\natexlab{a}}.

\bibitem[Mannelli et~al.(2020{\natexlab{b}})Mannelli, Vanden-Eijnden, and Zdeborov\'{a}]{mannelli}
S.~S. Mannelli, E.~Vanden-Eijnden, and L.~Zdeborov\'{a}.
\newblock Optimization and generalization of shallow neural networks with quadratic activation functions.
\newblock In \emph{Advances in Neural Information Processing Systems}, 2020{\natexlab{b}}.

\bibitem[Massart and Absil(2020)]{massart}
E.~Massart and P.-A. Absil.
\newblock Quotient geometry with simple geodesics for the manifold of fixed-rank positive-semidefinite matrices.
\newblock \emph{SIAM J. Matrix Anal. Appl.}, 41\penalty0 (1):\penalty0 171--198, Jan. 2020.

\bibitem[Mei et~al.(2019)Mei, Misiakiewicz, and Montanari]{mei}
S.~Mei, T.~Misiakiewicz, and A.~Montanari.
\newblock Mean-field theory of two-layers neural networks: dimension-free bounds and kernel limit.
\newblock In \emph{Proceedings of the Conference on Learning Theory}, pages 2388--2464, 2019.

\bibitem[Mignacco et~al.(2021)Mignacco, Krzakala, Urbani, and Zdeborová]{Mignacco}
F.~Mignacco, F.~Krzakala, P.~Urbani, and L.~Zdeborová.
\newblock Dynamical mean-field theory for stochastic gradient descent in {Gaussian} mixture classification.
\newblock \emph{J. Stat. Mech.}, 2021\penalty0 (12):\penalty0 124008, Dec. 2021.

\bibitem[Paszke et~al.(2019)Paszke, Gross, Massa, Lerer, Bradbury, Chanan, Killeen, Lin, Gimelshein, Antiga, Desmaison, Kopf, Yang, DeVito, Raison, Tejani, Chilamkurthy, Steiner, Fang, Bai, and Chintala]{pytorch}
A.~Paszke, S.~Gross, F.~Massa, A.~Lerer, J.~Bradbury, G.~Chanan, T.~Killeen, Z.~Lin, N.~Gimelshein, L.~Antiga, A.~Desmaison, A.~Kopf, E.~Yang, Z.~DeVito, M.~Raison, A.~Tejani, S.~Chilamkurthy, B.~Steiner, L.~Fang, J.~Bai, and S.~Chintala.
\newblock Pytorch: An imperative style, high-performance deep learning library.
\newblock In H.~Wallach, H.~Larochelle, A.~Beygelzimer, F.~d\textquotesingle Alch\'{e}-Buc, E.~Fox, and R.~Garnett, editors, \emph{Advances in Neural Information Processing Systems}, volume~32. Curran Associates, Inc., 2019.

\bibitem[Safran and Shamir(2018)]{safran}
I.~Safran and O.~Shamir.
\newblock Spurious local minima are common in two-layer relu neural networks.
\newblock In \emph{International conference on machine learning}, pages 4433--4441, 2018.

\bibitem[Smale(2011)]{SMT}
S.~Smale.
\newblock Stable manifolds for differential equations and diffeomorphisms.
\newblock In \emph{Topologia differenziale}, pages 93--126. Springer Berlin Heidelberg, 2011.

\bibitem[Soltanolkotabi et~al.(2018)Soltanolkotabi, Javanmard, and Lee]{Soltano}
M.~Soltanolkotabi, A.~Javanmard, and J.~D. Lee.
\newblock Theoretical insights into the optimization landscape of over-parameterized shallow neural networks.
\newblock \emph{IEEE Transactions on Information Theory}, 65\penalty0 (2):\penalty0 742--769, 2018.

\bibitem[Tian(2017)]{tian}
Y.~Tian.
\newblock An analytical formula of population gradient for two-layered {R}e{LU} network and its applications in convergence and critical point analysis.
\newblock In \emph{Proceedings of the International Conference on Machine Learning}, pages 3404--3413, 2017.

\bibitem[Venturi et~al.(2019)Venturi, Bandeira, and Bruna]{venturi}
L.~Venturi, A.~S. Bandeira, and J.~Bruna.
\newblock Spurious valleys in one-hidden-layer neural network optimization landscapes.
\newblock \emph{Journal of Machine Learning Research}, 20\penalty0 (133):\penalty0 1--34, 2019.

\bibitem[Yan et~al.(1994)Yan, Helmke, and Moore]{oja}
W.-Y. Yan, U.~Helmke, and J.~B. Moore.
\newblock Global analysis of {Oja}'s flow for neural networks.
\newblock \emph{IEEE Trans. Neural Netw.}, 5\penalty0 (5):\penalty0 674--683, Sept. 1994.

\end{thebibliography}
\endgroup

\newpage
\onecolumn
\appendix

\newgeometry{textwidth=15.7cm, textheight=23.3cm}
\raggedbottom

\section{Additional Results}

In this section we provide some additional results and insights. As mentioned in \cref{sec:CV}, we derive convergence rates in the overparameterized case $m \geq \min(m^*, d)$ in \cref{App:CVRate}. In \cref{App:Manova}, we give a more detailed description of the results obtained for the limit distributions of product of projections that we introduce in \cref{subsec:sampling}. Finally, we discuss in \cref{App:Overlap} the notion of overlap introduced in equation \eqref{eq:overlap} and prove some results mentioned throughout \cref{sec:highdim}. 

\subsection{Convergence Rates} \label{App:CVRate}

Following the convergence result derived in \cref{PropCV}, the natural question is to understand how fast this convergence is.  In the following, we investigate the convergence rates associated with the gradient flow in the case $m \geq \min(m^*, d)$. The method we use in the proof can also be applied to the underparameterized setting $m < m^*$, but we choose to focus on the relevant case $m \geq \min(m^*, d)$, where the gradient flow converges towards the teacher matrix. 

In the following proposition, the convergence rates are derived in terms of the loss, i.e., we obtain a bound on $\L(W(t))$ where $W(t)$ is solution of the gradient flow in equation \eqref{eq:flow}. Due to the expression of the loss, this leads to a bound on the distance $\|W(t)W(t)^T - Z^*\|$.
\begin{proposition} \label{PropCVRate}
Let $W(t)$ be solution of the flow \eqref{eq:flow} with initial condition $W(t=0) = W^0$. Suppose that $m \geq \min(m^*, d)$. Denote $\mu$ the smallest non-zero eigenvalues of $Z^*$. Then, under Assumption \ref{AssumptionInit}, with probability one:
\begin{itemize}
    \item If $m, m^* \geq d$:
    \begin{equation*}
        \L(W(t)) \underset{t \to \infty}{=}  O \big( e^{-8 \mu t} \big).
    \end{equation*}
    \item If $m=m^* < d$: 
    \begin{equation*}
\L(W(t)) \underset{t \to \infty}{=}  O \big( e^{-4 \mu t} \big).
\end{equation*}
    \item If $m, d > m^*$: 
    \begin{equation*}
 \L(W(t)) \underset{t \to \infty}{=} O \left( \frac{1}{t^2} \right).
\end{equation*}
\end{itemize}
\end{proposition}
This proposition is proven in \cref{ProofCVRate}. The proof uses the fact that if $W(t)$ is solution of the flow in equation \eqref{eq:flow}, then with probability one, $W(t)W(t)^T \xrightarrow[]{} Z^*$ as $t \to \infty$. As explained in \cref{sec:dynamics}, the main challenge to understand the long time dynamics of the flow is to obtain the behaviour of the function $\psi(t)$ introduced in \cref{Prop:Analytic}. 
\vspace{0.1 cm}
\begin{figure} [H]
    \centering
    \includegraphics[width=0.75\linewidth]{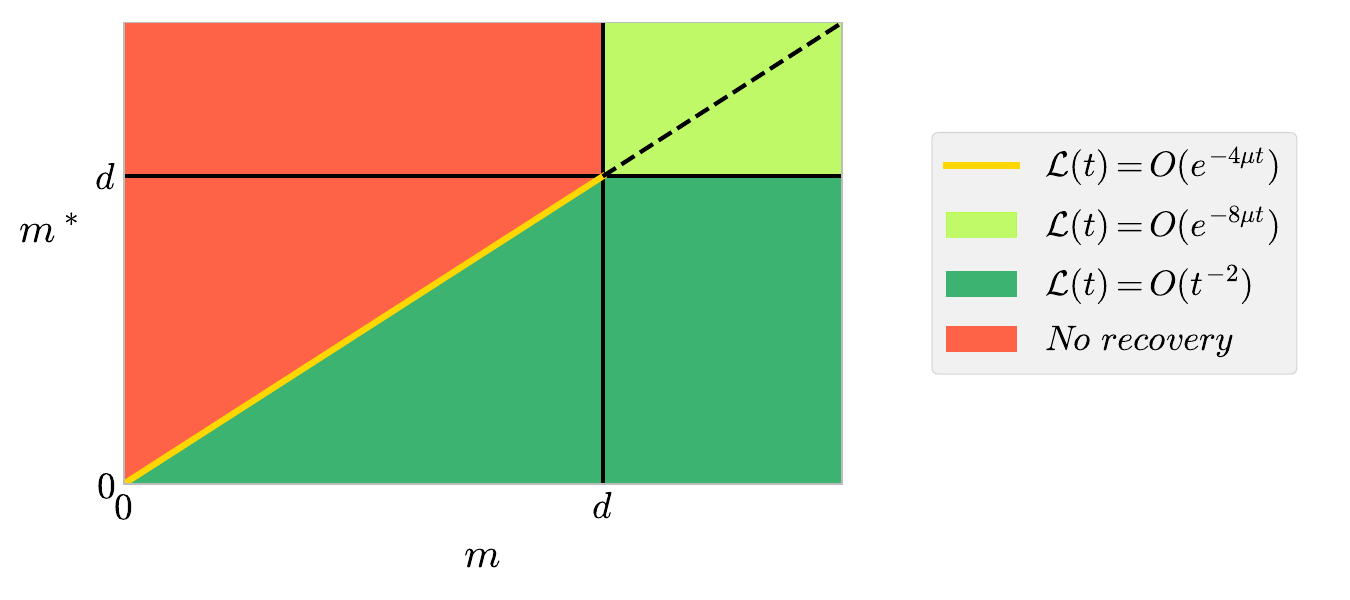}
    \caption{Diagram for the convergence rates of $\L(t) = \L(W(t))$ depending on the value of $m, m^*$. Valid under the initialization Assumption \ref{AssumptionInit}.  }
    \label{fig:diagrammCV}
\end{figure}
\cref{fig:diagrammCV} displays the diagram of convergence rates in the overparamterized case. While the regions $m, m^* \geq d$ and $m=m^* \leq d$ exhibit exponentially fast convergence, the third one ($m, d > m^*$) reveals a slower convergence. This discrepancy can be understood in terms of rank: in the two first case, we have that $\mathrm{rank}(W(t)W(t)^T) = \mathrm{rank}(Z^*)$ all along the flow. On the contrary, when $m, d > m^*$, then $W(t)W(t)^T$ still converges towards $Z^*$, which is now of lower rank. In this case the proof shows that the convergence is slowed down by the positive eigenvalues of $W(t)W(t)^T$ that go to zero as $t \to \infty$. 

Whenever $m^* < d$, a high overparameterization is not ideal in terms of convergence rate. In this case, this is the smallest overparameterization (exactly $m=m^*$) which gives the most efficient convergence. Obviously, the number of neurons of the teacher $m^*$ is not known in advance, and it is not clear how it can be inferred from the observations of the output of the predictors $u^*(x) = \tr \big(W^*W^{*T} xx^T \big)$ (see equation \eqref{eq:predictors}). 

Finally, we believe that the bounds derived in \cref{PropCVRate} are tight. As depicted in \cref{fig:CVrate}, in the case $m^* \leq m \leq d$ (left panel), the function $t^2 \L(W(t))$ stays bounded with time. Likewise, for the case $m=m^* \leq d$ (right panel), the loss $\L(W(t))$ follows the line drawn by $e^{-4 \mu t}$, where $\mu$ is the smallest non-zero eigenvalue of $Z^*$ (in log-scale on the $y$-axis). 
\begin{figure}[h]
    \centering
    \includegraphics[width=\linewidth]{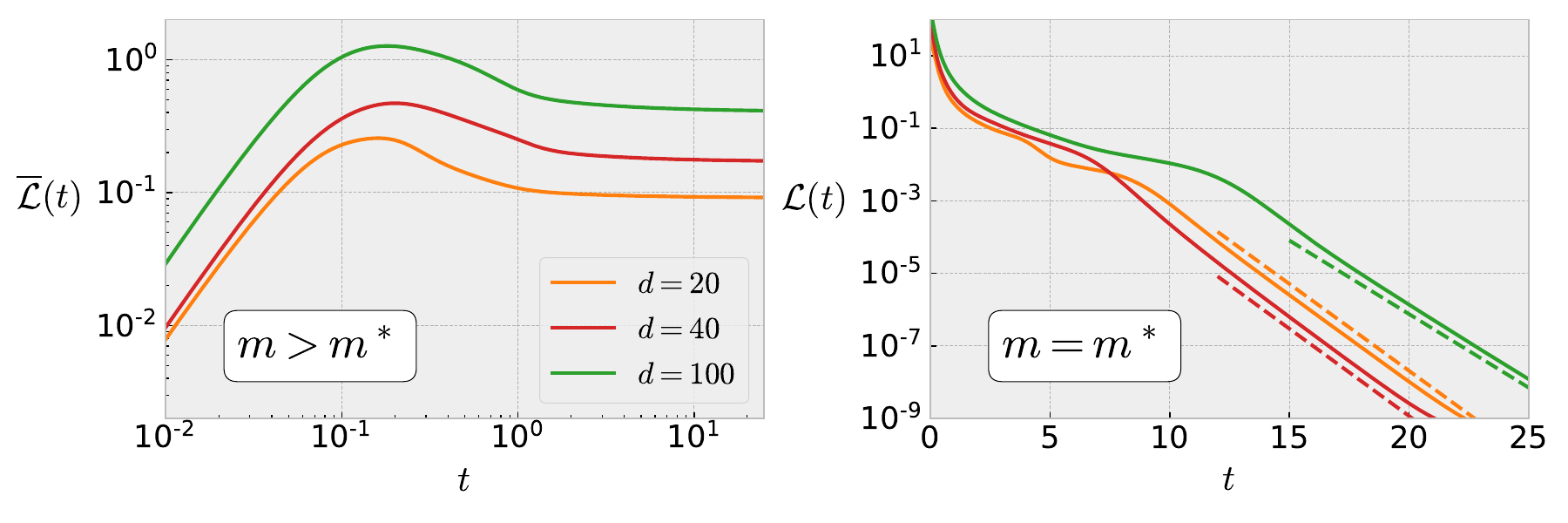}
    \vspace*{-0.5cm}  
    \caption{Evolution of the loss as a function of time. Left: $m = d/2$ and $m^* = d/4$, renormalized loss $\overline{\L}(t) = t^2 \L(W(t))$. Right: $m = m^* = d/2$, $\L(t) = \L(W(t))$ and dotted lines correspond to $\mathrm{cst} \times e^{-4 \mu t}$ where $\mu$ is the smallest non-zero eigenvalue of $Z^*$. The solution $W(t)$ was simulated using gradient descent with stepsize $\eta = 10^{-2}$. Results averaged over 5 simulations. }
    \label{fig:CVrate}
\end{figure}

\subsection{Limit Distribution for Product of Projections} \label{App:Manova}

We now give more detailed results on the limit distribution for products of projections. In \cref{sec:highdim}, we present a convergence result for the product of matrices $Y_d = U_d^{0T} U_d^* U_d^{*T} U_d^0\in \R^{m_d \times m_d}$, where $U_d^{0T}U_d^0 = I_{m_d}$ and $U_d^{*T} U_d^* = I_{m_d^*}$. We also assume that $U_d^0$ and $U_d^*$ are uniformly drawn under this constraint. This can be achieved by two different but equivalent ways. For $k \leq d$, the manifold $\mathcal{M}_{d, k} = \big\{ U \in \R^{d \times k}, \, U^TU = I_k \big\}$ can be endowed with a uniform measure, and $U_d^0, U_d^*$ can be respectively drawn from this measure on $\mathcal{M}_{d, m_d}$ and $\mathcal{M}_{d, m_d^*}$. Another equivalent way of drawing uniformly on $\mathcal{M}_{d, k}$ is by drawing $V \in \R^{d \times k}$ with i.i.d.~Gaussian coefficients $\N(0,1)$ and to let $U = V(V^TV)^{-1/2}$, so that, conditionally on the event $\{ V^TV \text{ is invertible}\}$ (which has probability one), $U$ is uniform on $\mathcal{M}_{d,k}$. Thus, in all the following, we make the assumption:
\begin{assumption} \label{AssumptionHighdim}
    The initial condition of the flow $W_d^0 \in \R^{d \times m_d}$ and the teacher matrix $W_d^* \in \R^{d \times m_d^*}$ verify:
    \begin{equation*}
        W_d^0 = \frac{1}{\sqrt{m_d}} U_d^0 ,\hspace{1.5 cm} W_d^* = \frac{1}{\sqrt{m_d^*}} U_d^*,
    \end{equation*}
    where the matrices $U_d^0$ and $U_d^*$ are uniformly drawn on $\mathcal{M}_{d, m_d}$ and $\mathcal{M}_{d, m_d^*}$ respectively, with:
    \begin{equation*}
        \frac{m_d}{d} \xrightarrow[d \to \infty]{} \alpha ,\hspace{1.5 cm} \frac{m_d^*}{d} \xrightarrow[d \to \infty]{} \alpha^*.
    \end{equation*}
\end{assumption}
This is the assumption under which the results of \cref{sec:highdim} are valid. As mentioned earlier, with probability one, the empirical spectral distribution of $Y_d$ weakly converges towards some probability measure as $d \to \infty$. \citet{manova}, \citet{aubrun}, and \citet{hiai} have shown the convergence of the empirical spectral distribution of the matrix $X_d = U_d^0 Y_d U_d^{0T} \in \mathcal{S}_d^+(\R)$. Informally, they obtain that for all $f : [0,1] \xrightarrow[]{} \R$ continuous: 
\begin{equation*}
    \frac{1}{d} \tr \big( f(X_d) \big) \xrightarrow[d \to \infty]{} \int f(x) \d \nu(x),
\end{equation*}
where: 
\begin{equation*}
    \d \nu(x) = \big( 1 - \min(\alpha, \alpha^*)) \delta_0(x) + \max(\alpha + \alpha^* - 1, 0) \delta_1(x) + \frac{\sqrt{(r_+-x)(x-r_-)}}{2 \pi x(1-x)} \1_{[r_-, r_+]}(x) \d x.
\end{equation*}
More precisely, \citet{manova} obtained a convergence in moments and in probability under general assumption on the matrices distribution, and \citet{hiai} derived a large deviation result whenever the matrices $U_d^*$, $U_d^0$ are drawn uniformly on the Stiefel manifold. 

To recover the measure defined in equation \eqref{eq:manova}, note that $X_d$ and $Y_d$ have the same non-zero eigenvalues (this is true only when $U_d^{0T} U_d^0 = I_{m_d}$). Thus, we have for $f : [0,1] \xrightarrow[]{} \R$ continuous: 
\begin{equation*}
\begin{aligned}
    \frac{1}{m_d} \tr \big( f(Y_d) \big) &= \frac{d}{m_d} \frac{1}{d} \tr \big( f(X_d) \big) - \frac{d-m_d}{m_d} f(0) \\
    & \xrightarrow[d \to \infty]{} \frac{1}{\alpha} \int f(x) \d \nu(x) - \frac{1 - \alpha}{\alpha} f(0).
\end{aligned}
\end{equation*}
Thus, the empirical spectral distribution of $Y_d$ converges as $d \to \infty$ towards a probability measure $\mu$ which satisfies $\alpha \mu = \nu - (1- \alpha) \delta_0$, which indeed corresponds to equation \eqref{eq:manova}. 

We now state a lemma which will be used in the proof of Propositions \ref{PropPhi} and \ref{PropOverlap}. It generalizes the previous convergence of the empirical spectral distribution of $Y_d$ for a compact family of functions. 
\begin{lemma} \label{lemma:largedev}
    Suppose that Assumption \ref{AssumptionHighdim} is verified and denote $Y_d = U_d^{0T} U_d^* U_d^{*T} U_d^0\in \R^{m_d \times m_d}$. Let $\mathcal{F} \subset \mathcal{C}([0,1], \R)$ be a compact set for the norm $\| \. \|_\infty$. Then, with probability one:
    \begin{equation*}
        \sup_{f \in \mathcal{F}} \left| \frac{1}{m_d} \tr \big( f(Y_d) \big) - \int f(x) \d \mu(x) \right| \xrightarrow[d \to \infty]{} 0.
    \end{equation*}
\end{lemma}
\begin{proof}
    We let $\eta > 0$. Since $\mathcal{F}$ is compact, we can find $f_1, \dots, f_n \in \mathcal{F}$ such that any element of $\mathcal{F}$ is at distance $\leq \eta / 4$ of one of the $f_i$'s. Thus: 
    \begin{equation*}
        \sup_{f \in \mathcal{F}} \left| \frac{1}{m_d} \tr \big( f(Y_d) \big) - \int f(x) \d \mu(x) \right| \leq \frac{\eta}{2} + \sup_{1 \leq i \leq n} \left| \frac{1}{m_d} \tr \big( f_i(Y_d) \big) - \int f_i(x) \d \mu(x) \right|,
    \end{equation*}
    and: 
    \begin{equation} \label{BoundProbaManova}
         \P \left(  \sup_{f \in \mathcal{F}} \left| \frac{1}{m_d} \tr \big( f(Y_d) \big) - \int f(x) \d \mu(x) \right| \geq \eta \right) \leq \P \left(  \sup_{1 \leq i \leq n} \left| \frac{1}{m_d} \tr \big( f_i(Y_d) \big) - \int f_i(x) \d \mu(x) \right| \geq \frac{\eta}{2} \right).
    \end{equation}
    We now use the large deviation result obtained by \citet{hiai}. We denote $\mu_d$ the empirical spectral distribution of $Y_d$ and for $\delta > 0$, we define: 
    \begin{equation*}
        E_\eta = \left\{ \nu \in \mathcal{M} \ \Big| \ \sup_{1 \leq i \leq n} \left| \int f_i(x) \d \nu(x) - \int f_i(x) \d \mu(x) \right| \geq \frac{\eta}{2} \right\},
    \end{equation*}
    where $\mu$ is defined in equation \eqref{eq:manova} and $\mathcal{M}$ denotes the space of probability measures over $[0,1]$. First, since the supremum is over a finite number of functions, $E_\eta$ is closed for the weak topology on $\mathcal{M}$. 
    
    Now \citet{hiai} state that: 
    \begin{equation*}
        \limsup_{d \to \infty} \frac{1}{d^2} \log \P(\mu_d \in E_\eta) \leq - \inf_{\nu \in E_\eta} I(\nu) \equiv - \beta_\eta,
    \end{equation*}
    where $I$ is positive, lower semi-continuous (with respect to the weak topology on $\mathcal{M}$), such that $I(\mu) = 0$ and $I > 0$ elsewhere. Remains to show that $\beta_\eta > 0$. By contradiction, suppose that there is $(\nu_p)_{p \in \mathbb{N}}$ a sequence of $E_\eta$ such that $I(\nu_p) \xrightarrow[p \to \infty]{} 0$. Since $\mathcal{M}$ is compact, one can extract a converging subsequence: $\nu_{\phi(p)} \xrightarrow[p \to \infty]{} \nu$. Thus, by lower semi-continuity of $I$: 
    \begin{equation*}
        I(\nu) \leq \liminf_{p \to \infty} I(\nu_{\phi(p)}) = 0,
    \end{equation*}
    and $\nu = \mu$. Since $E_\eta$ is closed for the weak topology, this implies that $\mu \in E_\eta$ which leads to a contradiction. Thus for $d$ large enough:
    \begin{equation*}
        \P \left(  \sup_{1 \leq i \leq n} \left| \frac{1}{m_d} \tr \big( f_i(Y_d) \big) - \int f_i(x) \d \mu(x) \right| \geq \frac{\eta}{2} \right)\leq \exp \left(- \frac{\beta_\eta d^2}{ 2} \right),
    \end{equation*}
    and using equation \eqref{BoundProbaManova}, we get that, for all $\eta > 0$: 
    \begin{equation*}
        \sum_{d \geq 0} \P \left(  \sup_{f \in \mathcal{F}} \left| \frac{1}{m_d} \tr \big( f(Y_d) \big) - \int f(x) \d \mu(x) \right| \geq \eta \right) < \infty,
    \end{equation*}
    and the conclusion using Borel-Cantelli lemma. 
\end{proof}

\vspace{0.3 cm}

\subsection{Overlap} \label{App:Overlap}

We study the overlap function in \cref{subsec:overlap}. In a general Euclidean space $E$, the overlap between two vectors $x,y \in E$ is given by:
\begin{equation*}
    \chi(x,y) = \frac{\big| \langle x,y \rangle \big|}{\big\|x\big\| \big\|y\big\|},
\end{equation*}
where $\| \. \|$ is the norm associated with the inner product on $E$. The overlap is a powerful scalar quantity measuring the alignment between two vectors. Indeed, Cauchy-Schwartz inequality ensures that $\chi(x,y) \leq 1$, and $\chi(x,y) = 1$ if and only if $x$ and $y$ are aligned. 

It is well known that if $E = \R^d$ and $x,y$ are independent and uniformly drawn on the sphere $\mathbb{S}^{d-1}$, then:
\begin{equation*}
    \chi(x,y) \underset{d \to \infty}{=} O \left( \frac{1}{\sqrt{d}} \right).
\end{equation*}
In this case, what we called the purely random overlap goes to zero in the high-dimensional limit. Throughout \cref{sec:highdim}, we studied the overlap between two matrices in $\mathcal{S}_d^+(\R)$:
\begin{equation*}
    \chi^{\mathrm{PSD}}(X,Y) = \frac{\tr(XY)}{\big\| X \big\|_F \big\| Y \big\|_F}.
\end{equation*}
For $X,Y \in \mathcal{S}_d^+(\R)$, we always have that $\tr(XY) \geq 0$. Note that we recover the $\R^d$ overlap with $\chi^{\mathrm{PSD}}(xx^T, yy^T) = \chi(x,y)^2$ for $x, y \in \R^d$. The first challenge is to determine the purely random overlap mentioned in the beginning of \cref{subsec:overlap}. In our special case, the goal is to study the overlap between the teacher and student along the flow: therefore, our purely random overlap will be defined as the one between the teacher and student at initialization. Those are independent random projection matrices of rank $m_d^*$ and $m_d$ respectively. We now show that this overlap admits a limit as $d \to \infty$, whenever $m_d, m_d^*$ grow linearly with the dimension. 
\begin{lemma}
    The overlap between the teacher matrix and the student at initialization admits the limit, with probability one: 
    \begin{equation*}
        \chi^\mathrm{PSD}(Z_d^0, Z_d^*) \xrightarrow[d \to \infty]{} \sqrt{\alpha \alpha^*}.
    \end{equation*}
\end{lemma} \begin{proof} Due to the definition of $Z_d^*$ and $Z_d^0$, we have: 
    \begin{equation*}
        \chi^\mathrm{PSD}(Z_d^0, Z_d^*) =  \frac{\tr \big( U_d^{0T} U_d^* U_d^{*T} U_d^0 \big)}{\big\|U_d^0 U_d^{0T} \big\|_F \big\| U_d^* U_d^{*T} \big\|_F}. 
    \end{equation*}
    For the denominator, we obtain: 
    \begin{equation*}
        \big\|U_d^0 U_d^{0T} \big\|_F = \sqrt{m_d} , \hspace{2 cm} \big\|U_d^* U_d^{*T} \big\|_F = \sqrt{m_d^*}.
    \end{equation*}
    Due to the rotational invariance of the distribution of $U_d^* U_d^{*T}$, one can suppose without loss of generality that $U_d^* U_d^{*T}$ is the orthogonal projection onto the $m_d^*$ first vectors of the standard basis in $\R^d$, denoted $\Pi_d$. Thus, one can rewrite: 
    \begin{equation*}
        \chi^\mathrm{PSD}(Z_d^0, Z_d^*) = \frac{1}{\sqrt{m_d m_d^*}} \sum_{k=1}^{m_d} \big\| \Pi_d(v_k) \big\|^2,
    \end{equation*}
    with $(v_1, \dots, v_{m_d})$ orthonormal such that $U_d^0 U_d^{0T} = \sum_{k=1}^{m_d} v_kv_k^T$. Moreover, each $v_j$
 is uniform on the sphere. Thus, for $\epsilon > 0$, applying an union bound:
 \begin{equation*}
     \P \left( \left| \chi^\mathrm{PSD}(Z_d^0, Z_d^*) - \frac{\sqrt{m_d m_d^*}}{d} \right| \geq \epsilon \right) \leq m_d \, \P \left( \left| \big\| \Pi_d(v_k) \big\|^2 - \frac{m_d^*}{d} \right| \geq \epsilon \sqrt{\frac{m_d^*}{m_d}} \right).
 \end{equation*}
Following \citet{dasgupta}, we have, if $v$ is uniform on the sphere and for $\eta \in [0,1]$: 
\begin{equation*}
    \P \left( \left| \big\| \Pi_d(v) \big\|^2 - \frac{m_d^*}{d} \right| \geq \frac{m_d^*}{d} \eta \right) \leq 2 \exp \left( - \frac{m_d^* \eta^2}{4} \right),
\end{equation*}
which leads to, with probability one:
\begin{equation*}
    \left| \chi^\mathrm{PSD}(Z_d^0, Z_d^*) - \frac{\sqrt{m_d m_d^*}}{d} \right| \xrightarrow[d \to \infty]{} 0,
\end{equation*}
and the result. 
 \end{proof}

Note that the behaviour of this overlap is very different from the $\R^d$ case. The overlap between two independent vectors in $\R^d$ vanishes as $d \to \infty$. Here, due to the fact we consider a number of vectors growing with the dimension, we instead obtain a positive limit. 

Finally, we prove a claim done at the end of \cref{subsec:overlap}: the overlap reached by the gradient flow at timescale $t \gg d$ (see \cref{PropOverlap}) is the best we can get whenever the teachers are orthonormal (for a given number of students).

\begin{proposition} Suppose that $Z^* \in \mathcal{S}_d^+(\R)$ is an orthogonal projection matrix of rank $m^* \leq d$. Then, for $m \leq d$:
        \begin{equation*}
        \sup_{\substack{Z \in \mathcal{S}_d^+(\R) \\ \mathrm{rank}(Z) \leq m}} \frac{\tr(ZZ^*)}{\big\|Z\big\|_F \big\|Z^*\big\|_F} =  \min \left( \sqrt{\frac{m}{m^*}}, 1 \right).
    \end{equation*}
\end{proposition}

\begin{proof}
    If $m \geq m^*$, taking $Z = Z^*$ leads to the result since the overlap is always smaller than one. Suppose that $m < m^*$ and take $Z \in \mathcal{S}_d^+(\R)$ of rank $\leq m$. We write: 
    \begin{equation*}
        Z = \sum_{k=1}^m \lambda_k v_kv_k^T \hspace{1.5 cm} Z^* = \sum_{k=1}^{m^*} v_k^* (v_k^*)^T,
    \end{equation*}
    so that: 
    \begin{equation*}
        \tr(ZZ^*) = \sum_{j=1}^m \lambda_j \sum_{k=1}^{m^*} \left( v_j^T v_k^* \right)^2.
    \end{equation*}
    The sum over $k$ corresponds to the norm of the projection of $v_j$ onto $\mathrm{Vect}(v_1^*, \dots, v_{m^*}^*)$, thus it is smaller than the norm of $v_j$, equal to one. Thus, since $\big\|Z^* \big\|_F = \sqrt{m^*}$
    \begin{equation*}
        \frac{\big| \tr(ZZ^*) \big|}{\big\|Z\big\|_F \big\|Z^*\big\|_F} \leq \frac{1}{\sqrt{m^*}} \frac{\sum_{j=1}^m \lambda_j}{\left( \sum_{j=1}^m \lambda_j^2 \right)^{1/2}},
    \end{equation*}
    and the upper bound using Cauchy-Schwarz inequality. The case of equality is reached if and only if $Z$ is of the form $Z = \lambda \sum_{k=1}^m v_kv_k^T$, with $(v_1, \dots, v_m)$ an orthonormal family of $\mathrm{Vect}(v_1^*, \dots, v_{m^*}^*)$. 
\end{proof}

\vspace{0.5 cm}

\section{Useful Lemmas} \label{App:Lemmas}

In this section we mention several lemmas that will be used throughout the proofs of the results in \cref{App:Proofs}.  
\begin{lemma} \label{lemma:traces}
    Let $A \in \mathcal{S}_d^+(\R)$. Then:
    \begin{equation*}
        \frac{1}{d} \big[\tr (A) \big]^2 \leq \tr(A^2) \leq \big[\tr (A) \big]^2.
    \end{equation*}
    Moreover, if $B \in \mathcal{S}_d(\R)$:
    \begin{equation*}
        \lambda_{min}(B) \tr(A) \leq \tr(AB) \leq \lambda_{max}(B) \tr(A).
    \end{equation*}
\end{lemma}
\begin{lemma}[Courant-Fischer min-max formula] \label{lemma:courantfischer}
    Let $A \in \mathcal{S}_d(\R)$ and denote $\lambda_1 \geq \dots \geq \lambda_d$ its eigenvalues. Then:
    \begin{equation*}
        \lambda_j = \max_{\substack{V \subset \R^d \\ \mathrm{dim}V = j}} \min_{\substack{x \in V \\ \|x\|=1}} x^T A x  \  = \min_{\substack{V \subset \R^d \\ \mathrm{dim}V = d-j+1}} \max_{\substack{x \in V \\ \|x\|=1}} x^T A x.
    \end{equation*}
\end{lemma}

\begin{lemma} \label{lemma:asymptoticsIntegral}
    Let $\lambda > 0$ and $v : \R^+ \xrightarrow[]{} \R$ continuously differentiable such that $\dot{v}(t) \xrightarrow[t \to \infty]{} 0$. Then : 
    \begin{equation*}
        \int_0^t e^{\lambda s} e^{v(s)} \d s \underset{t \to \infty}{\sim} \frac{1}{\lambda} e^{\lambda t} e^{v(t)}.
    \end{equation*}
    Moreover, for all $\mu > 0$ :
    \begin{equation*}
        e^{- \mu t} \int_0^t e^{v(s)} \d s \xrightarrow[t \to \infty]{} 0.
    \end{equation*}
\end{lemma}

\begin{proof}
    The condition on $v$ implies that the first integral diverges as $t \to \infty$. Integrating by parts:
    \begin{equation*}
    \int_0^t e^{\lambda s} e^{v(s)} \d s  = \frac{1}{\lambda} \Big[ e^{\lambda s} e^{v(s)}\Big]^t_0 - \frac{1}{\lambda} \int_0^t \dot{v}(s) e^{\lambda s} e^{v(s)} \d s.
    \end{equation*}
    Due to the assumption on $v$, the second integral is negligible with respect to the first, which gives the first result. For the second claim, let $\epsilon \in ]0, \mu[$ and $t_0$ such that $v(t) \leq \epsilon t$ for $t \geq t_0$. Then, splitting the integral:
    \begin{equation*}
        e^{- \mu t} \int_0^t e^{v(s)} \d s \leq e^{- \mu t} \int_0^{t_0} e^{v(s)} \d s + \frac{1}{\epsilon} e^{-(\mu - \epsilon)(t-t_0)} \xrightarrow[t \to \infty]{} 0.
    \end{equation*}
\end{proof}

\newpage
\section{Proofs of the Main Results} \label{App:Proofs}

\subsection{Proof of Proposition \ref{PropCV}} \label{ProofCV}

We now prove the convergence result stated in \cref{PropCV}. As mentioned earlier, we make use of the stable manifold theorem \citep[see][]{SMT}: under the assumption that the initialization of the flow $W^0$ is drawn from a distribution which is absolutely continuous with respect to the Lebesgue measure on $\R^{d \times m}$, and since the loss function $\mathcal{L}$ defined in equation \eqref{eq:loss} is analytic in the coefficients of $W$, then with probability one (with respect to the initialization) the solution of the gradient flow \eqref{eq:flow} converges towards a local minimizer of $\mathcal{L}$, i.e., a point $W \in \R^{d \times m}$ satisfying :
\begin{equation*}
    \nabla \L(W) = 0 \hspace{1 cm} \text{and} \hspace{1 cm} \tr \big( \d^2L_W(K)K^T) \geq 0,
\end{equation*}
for all $K \in \R^{d \times m}$. \cref{PropCV} is proven as follows:
\begin{itemize}
    \item In \cref{lemma:criticalpoints}, we determine the critical points of the loss, i.e., the $W \in \R^{d \times m}$ such that $\nabla \L(W) = 0$.
    \item In \cref{lemma:rankdeficient}, we make use of the structure of the loss and show that a rank deficient local minimizer necessarily leads to an optimal predictor (i.e.,  $WW^T = Z^*$). 
    \item Finally, we determine the local minimizer of the loss.
\end{itemize}

\begin{lemma} \label{lemma:criticalpoints}
    Let $W \in \R^{d \times m}$ be a critical point of $\L$. Then there exists $I \subset \llbracket 1 , d \rrbracket$ of size $\leq \min(m,d)$ such that: 
    \begin{equation*}
        WW^T = \sum_{k \in I} (\mu_k + \tau) v_k^* (v_k^*)^T,
    \end{equation*}
    with:
    \begin{equation*}
        \tau = \frac{1}{2 + |I|}  \sum_{\substack{ 1 \leq k \leq m^* \\ k \notin I}} \mu_k.
    \end{equation*}    
\end{lemma}

\begin{proof}
    Let $W \in \R^{d \times m}$ such that $\nabla \L(W) = 0$, that is:
    \begin{equation*}
        2 \big(Z^* - WW^T \big) W + \tr \big(Z^* - WW^T \big) W = 0.
    \end{equation*}
    Computing $\nabla \L(W) W^T - W \nabla \L(W)^T = 0$, we obtain that $Z^*$ and $WW^T$ commute, therefore we can write a singular value decomposition of the form:
    \begin{equation*}
        W = \sum_{k=1}^d \sqrt{\lambda_k} v_k^* w_k^T,
    \end{equation*}
    where $\lambda_1, \dots, \lambda_d \geq 0$ and if $m \leq d$, at most $m$ of them are non-zero. Setting $\tau = \frac{1}{2} \tr \big(Z^* - WW^T \big)$, the constraint $\nabla \L(W) = 0$ writes:
    \begin{equation*}
        \sum_{k=1}^d \sqrt{\lambda_k} \big( \lambda_k - \mu_k - \tau) v_k^* w_k^T = 0.
    \end{equation*}
    Thus, if $\lambda_k > 0$, necessarily $\lambda_k = \mu_k + \tau$. Defining $I = \{ k \in \llbracket 1, d \rrbracket, \lambda_k > 0 \}$ (which is of size $\leq \min(m,d)$) leads to the first claim. For the expression of $\tau$, simply remark that:
    \begin{equation*}
        \tau = \frac{1}{2} \tr \big(Z^* - WW^T \big) = \frac{1}{2} \left( \tr(Z^*) - \sum_{k \in I} (\mu_k + \tau) \right) = \frac{1}{2} \sum_{\substack{ 1 \leq k \leq m^* \\ k \notin I}} \mu_k - \frac{1}{2} |I| \tau,
    \end{equation*}
    which gives the desired result.
\end{proof}

We now state a more general lemma which bears on the functions of the form $L(W) = F(WW^T)$ where $F$ is convex. This naturally includes our setup since for the loss defined in equation \eqref{eq:loss}, the function $F$ is quadratic with a positive-semidefinite Hessian, hence convex. 
\begin{lemma} \label{lemma:rankdeficient}
    Let $L : \R^{d \times m} \xrightarrow[]{} \R$ such that $L(W) = F(WW^T)$ for all $W \in \R^{d \times m}$ with $F : \mathcal{S}_d(\R) \xrightarrow[]{} \R$ convex and twice continuously differentiable. Suppose that $W \in \R^{d \times m}$ is a local minimizer of $L$. Then, if $\mathrm{rank}(W) < m$, $WW^T$ is a global minimizer of $F$ over the space $\mathcal{S}_d^+(\R)$.
\end{lemma}

\begin{proof}
    The gradient and Hessian of $L$ can be deduced from the ones of $F$:
    \begin{equation*}
    \begin{aligned}
        \nabla L(W) &= 2 \nabla F(WW^T)W \\
       \tr \big( \d^2 L_W(K)K^T \big) &= 2 \tr \big( \nabla F(WW^T)KK^T \big) + \tr \left( \d^2 F_{WW^T}(WK^T + KW^T) \big(WK^T + KW^T) \right).
    \end{aligned}
    \end{equation*}
    If $W \in \R^{d \times m}$ is a local minimizer of $L$, then $\nabla F(WW^T)W = 0$ and $\tr \big( \d^2 L_W(K)K^T \big) \geq 0$ for all $K \in \R^{d \times m}$. If $\mathrm{rank}(W) < m$, then there is a non-zero $v \in \R^m$ such that $Wv = 0$. Let $u \in \R^d$ and evaluate the second equation at $K = uv^T$, so that $WK^T$ and $KW^T$ are zero. Thus, we get that $\tr \big( \nabla F(WW^T) uu^T) \geq 0$ so that $\nabla F(WW^T) \in \mathcal{S}_d^+(\R)$ (it is symmetric since $F$ is a function on $\mathcal{S}_d(\R)$). Since $F$ is convex, we have, for any $S \in \mathcal{S}_d^+(\R)$:
    \begin{equation*}
    \begin{aligned}
        F(S) &\geq F(WW^T) + \tr \big( \nabla F(WW^T) (S - WW^T) \big)\\
        &= F(WW^T) + \underbrace{ \tr \big( \nabla F(WW^T) S \big)}_{\textstyle \geq 0},
    \end{aligned}
    \end{equation*}
    where we used that $\nabla F(WW^T)W = 0$. Thus $WW^T$ is a global minimizer of $F$ over $\mathcal{S}_d^+(\R)$.  
\end{proof}

We can apply this result to our case, with $F(Z) = \frac{1}{2} \big\| Z-Z^* \big\|_F^2 + \frac{1}{4} \tr \big(Z-Z^*)^2$. Note that the only global minimizer of $F$ on $\mathcal{S}_d^+(\R)$ is $Z = Z^*$. 

We are now ready to prove \cref{PropCV}. Let $W \in \R^{d \times m}$ be a local minimizer of the loss $\L$. Already if $m > d$, then $\mathrm{rank}(W) < m$ and necessarily $WW^T = Z^*$ by the previous result, so we can assume that $m \leq d$. Likewise; if $m \geq m^*$ and $\mathrm{rank}(W) < m$, we also get that $WW^T = Z^*$. Finally, for $m < m^*$, such an equality is not possible since $\mathrm{rank}(Z^*) = m^*$ so we necessarily have $\mathrm{rank}(W) = m$. 

We now write the local minimizer condition $\tr \big( \d^2 \L_W(K)K^T \big) \geq 0$:
\begin{equation*}
    \tr \big( (WW^T - Z^*) KK^T \big) - \tau \tr(KK^T) + \tr(KW^TWK^T) + \tr(KW^TKW^T) + \tr(WK^T)^2 \geq 0,
\end{equation*}
where $\tau$ is defined in \cref{lemma:criticalpoints}. Reusing the notations of this lemma, we take $j \in I$ (i.e $\lambda_j = \mu_j + \tau > 0$) and $k \in \llbracket 1, d \rrbracket$. Evaluating the previous equation with $K = v_k^*w_j^T$, we get:
\begin{equation*}
    \lambda_k + \lambda_j + 2 \lambda_j \delta_{j,k} \geq \mu_k + \tau.
\end{equation*}
If $j \neq k$, then replacing $\lambda_j = \mu_j + \tau$, we have that $\lambda_k \geq \mu_k - \mu_j$. By the assumption on the eigenvalues of $Z^*$ ($\mu_1 > \dots > \mu_{m^*} > 0$), taking $k < j$ (if possible), we obtain that $\lambda_k > 0$ thus $k \in I$ by definition of $I$. Therefore, whenever $j \in I$, we have that $\llbracket 1, j \rrbracket \subset I$. Since $|I| = \mathrm{rank}(W) = m$ by assumption, necessarily $I = \llbracket 1, m \rrbracket$. For $m \leq m^*$, this implies the desired by the expression of $\tau$ (obviously $\tau = 0$ for $m=m^*$). For $m > m^*$, we also obtain that $\tau = 0$ and we cannot have $I = \llbracket 1, m \rrbracket$ since $Z^*$ have only $m^*$ non-zero eigenvalues. In this case, any local minimizer of the loss must satisfy $WW^T = Z^*$.

\subsection{Proof of Proposition \ref{PropCVRate}} \label{ProofCVRate}

Following the previous result, we now determine the convergence rates associated with the convergence of the flow in the case $m \geq m^*$. As explained in \cref{sec:dynamics}, the understanding of the function $\psi$ defined in \cref{Prop:Analytic} is essential to the determination of the dynamics. Under the assumption that $W(t)W(t)^T \xrightarrow[t \to \infty]{} Z^*$, which happens with probability one when the flow is correctly initialized, we have:
\begin{equation} \label{eq:C2behavepsi}
    \psi(t) = \int_0^t \tr \big(W(s)W(s)^T \big) \d s \underset{t \to \infty}{\sim} \tr(Z^*) t.
\end{equation}
The first challenge will be to gather more information about $\psi$, using equation \eqref{eq:selfpsi}:
\begin{equation} \label{eq:selfpsi2}
    \psi(t) = \frac{1}{4} \tr \log \left( I_m + 4  W^{0T} \int_0^t e^{-2 \psi(s)} e^{2s \tr(Z^*)} e^{4Z^*s} \d s \  W^0 \right).
\end{equation}
We will denote $E^*(t) = \int_0^t e^{-2 \psi(s)} e^{2s \tr(Z^*)} e^{4Z^* s} \d s \in \R^{d \times d}$. The proof is done as follows:
\begin{itemize}
    \item In \cref{lemma:eigenvalues}, we give an asymptotic behaviour of the eigenvalues of the matrix $W^{0T} E^*(t) W^0 \in \R^{m \times m}$ as $t \to \infty$.
    \item In \cref{lemma:psibehaviour}, we obtain an asymptotic development of $\psi$ up to the precision $o(1)$.
    \item We finally prove \cref{PropCVRate} by using the differential equation on $Z(t) = W(t)W(t)^T$:
    \begin{equation*}
        \dot{Z} = Z \tr(Z^*-Z)Z + 2Z^*Z + 2ZZ^* - 4Z^2.
    \end{equation*}
\end{itemize}
Our results are verified under Assumption \ref{AssumptionInit}. As mentioned before, for $m \geq m^*$, we have $W(t)W(t)^T \longrightarrow Z^*$ as $t \to \infty$ with probability one. Moreover, almost surely, any subfamily of $(W ^{0T} v_1^*, \dots, W^{0T} v_d^*)$ with size $\leq \min(m,d)$ is linearly independent in $\R^m$ (we say that this family is in general position, and this occurs almost surely as soon as assumption \ref{AssumptionInit} is verified and $W^0$ is drawn independently from $Z^*$). We restrict ourselves to this event from now on, to avoid sets of zero probability.

In order to determine the behaviour of $\psi$, it is important to understand how the eigenvalues of $W^{0T} E^*(t) W^0$ evolve as $t \to \infty$. This is done in the following lemma: the core idea is that the eigenvalues of $E^*$ (directly related to those of $Z^*$) are well separated at long timescales, even if we have little information on $\psi$. Indeed, the fact that $\psi(t) \underset{t \to \infty}{\sim} \tr(Z^*)t$ is enough to prove the result. 

In \cref{lemma:eigenvalues} and \cref{lemma:psibehaviour}, we will assume that $m^* \leq d$. Indeed, for the case $m^* > d$, the proof of convergence rates will be straightforward.

\begin{lemma} \label{lemma:eigenvalues}
    Denote $r = \min(m,d)$ and $\lambda_1(t), \dots, \lambda_r(t) > 0$ the non-zero eigenvalues of $W^{0T} E^*(t) W^0$. Then, for $j \in \llbracket 1, r \rrbracket$:
    \begin{equation*}
        \lambda_j(t) \underset{t \to \infty}{=} \Theta \left( \int_0^t e^{-2 \psi(s)} e^{2s \tr(Z^*)} e^{4 \mu_j s} \d s \right).
    \end{equation*}
\end{lemma}

The notation $\Theta$ means that $\lambda_j(t)$ is bounded from below and above by constants times the integral. This lemma is proven in \cref{App:proofeigenvalues}. The goal is now to use this result along with equation \eqref{eq:selfpsi2} to obtain a more precise idea of how $\psi(t)$ evolve as $t \to \infty$. This is done in the following lemma.

\begin{lemma} \label{lemma:psibehaviour}
    For $m \geq m^*$:
    \begin{equation*}
        \psi(t) \underset{t \to \infty}{=} \tr(Z^*)t + \frac{1}{2} \frac{\min(m,d)-m^*}{\min(m,d)+2} \log(t) + O(1).
    \end{equation*}
\end{lemma}

\begin{proof}
    Set $\xi(t) = \psi(t) - \tr(Z^*)t$ which is $o(t)$ by equation \eqref{eq:C2behavepsi}. From the previous lemma, let $C_j, D_j$ such that:
    \begin{equation*}
        C_j \leq \lambda_j(t)  \left( \int_0^t e^{-2 \psi(s)} e^{2s \tr(Z^*)} e^{4 \mu_j s} \d s \right)^{-1} \leq D_j.
    \end{equation*}  
    In the following, we only use the upper bound, the lower bound will be done similarly. By equation \eqref{eq:selfpsi2} and using the notations of the previous lemma, as well as $r = \min(m,d)$:
    \begin{equation*}
    \begin{aligned}
        \psi(t) &= \frac{1}{4} \sum_{j=1}^{r} \log \big( 1 + 4 \lambda_j(t) \big) \\
    &\leq \frac{1}{4} \sum_{j=1}^{m^*} \log \left(1 +  4 D_j \int_0^t e^{-2 \xi(s)} e^{4 \mu_j s} \d s \right) + \frac{1}{4} \sum_{j=m^*+1}^r \log \left( 1 + 4 D_j \int_0^t e^{-2 \xi(s)} \d s \right).
    \end{aligned}
    \end{equation*}
    From \cref{lemma:asymptoticsIntegral}, since $\dot{\xi}(t)$ goes to zero as $t \to \infty$, we can obtain the behaviour of the first term:
    \begin{equation*}
        \frac{1}{4} \sum_{j=1}^{m^*} \log \left(1 + 4 D_j \int_0^t e^{-2 \xi(s)} e^{4 \mu_j s} \d s \right)  \underset{t \to \infty}{=} \tr(Z^*)t - \frac{m^*}{2} \xi(t) + \frac{1}{4} \sum_{j=1}^{m^*} \log \left( \frac{D_j}{\mu_j} \right) + o(1).
    \end{equation*}
    For the second we need to understand how $\int_0^t e^{-2 \xi(s)} \d s$ behave. As $t \to \infty$, this integral either converges or goes to infinity. Suppose by contradiction that it converges towards some finite value. Using the two previous equations, along with the lower bound obtained from the previous lemma, this implies that $\xi(t)$ stays bounded, which leads to a contradiction since $\int_0^t e^{-2 \xi(s)} \d s$ should diverge. Thus, we obtain:
    \begin{equation*}
        \left( 1 + \frac{m^*}{2} \right) \xi(t) \underset{t \to \infty}{\lesssim} \frac{r-m^*}{4} \log \left( \int_0^t e^{-2\xi(s)} \d s \right) + \frac{1}{4} \sum_{j=1}^{m^*} \log \left( \frac{D_j}{\mu_j} \right) + \frac{1}{4} \sum_{j=m^*+1}^r \log(4 D_j) + o(1).
    \end{equation*}
    Thus, using the lower bound of the previous lemma, we obtain $h_1(t), h_2(t)$ two bounded functions such that:
    \begin{equation*}
        h_1(t) \leq \xi(t) - \frac{r-m^*}{2(m^*+2)} \log \left( \int_0^t e^{-2 \xi(s)} \d s \right) \leq h_2(t).
    \end{equation*}
    Setting $a(t) = \int_0^t e^{-2 \xi(s) \d s}$, we obtain:
    \begin{equation*}
        e^{-2 h_2(t)} \leq \dot{a}(t) a(t)^{\beta} \leq e^{-2 h_1(t)},
    \end{equation*}
    with $\beta = \dfrac{r-m^*}{m^*+2}$. Integrating, and using that $\xi(t) = - \dfrac{1}{2} \log \dot{a}(t)$:
    \begin{equation*}
        \zeta + h_1(t) + \frac{\beta}{2(\beta+1)} \log \left( \int_0^t e^{-2 h_2(s)} \d s \right) \leq \xi(t) \leq \zeta + h_2(t) + \frac{\beta}{2(\beta+1)} \log \left( \int_0^t e^{-2 h_1(s)} \d s \right),
    \end{equation*}
    where $\zeta$ is a constant depending only on $\beta$. We now use that $h_1, h_2$ are $O(1)$. There exists constants $K_1, K_2$ such that:
    \begin{equation*}
        K_1 \leq \xi(t) - \frac{\beta}{2(\beta+1)} \log(t) \leq K_2,
    \end{equation*}
    which gives the result by the definition of $\beta$ and $r = \min(m,d)$.     
\end{proof}

Now that the behaviour of $\psi$ is known, we are ready to determine the convergence rates in the case $m \geq m^*$. We showed that we can assume $Z^*$ to be diagonal without loss of generality. Let $Z^* = Q^T D^* Q$ with $Q \in O_d(\R)$. If $W(t)$ is solution of the flow of equation \eqref{eq:flow} (associated with $Z^*$), then $V(t) = QW(t)$ is solution of the same equation but with teacher $D^*$. Moreover, since $VV^T - D^* = Q(WW^T - Z^*) Q^T$, this implies that the loss of $V$ (with teacher $D^*$) is equal to the loss of $W$ (with teacher $Z^*$). Thus, as long as the convergence properties we derive are invariant under conjugation for $Z^*$ (for instance if they only depend on its eigenvalues, which is the case in \cref{PropCVRate}), we may assume that:
\begin{equation} \label{eq:C2diagZ*}
    Z^* = \begin{pmatrix}
        Z_0 & 0 \\ 0 &0
    \end{pmatrix},
\end{equation}
where $Z_0 = \mathrm{diag}(\mu_1, \dots, \mu_{m^*}) \in \R^{m^* \times m^*}$. We split the proof into two parts: first, we study the case $m,m^* \geq d$. In the second part, we jointly cover the two other cases.

\subsubsection{Highly overparameterized case}

We start with the first case, i.e $m, m^* \geq d$. If $W(t)$ if solution of the flow then, $\dot{W}(t) = -M(t)W(t)$ with $M(t) = 2 \big(W(t)W(t)^T - Z^*) + \tr \big(W(t)W(t)^T - Z^* \big) I_d$. Thus:
\begin{equation*}
    \frac{\d}{\d t} \L(W(t)) = - \big\| \dot{W}(t) \big\|^2 = - \tr \big( M(t)^2 W(t) W(t)^T\big) \leq - \lambda_{min} \big(W(t)W(t)^T \big) \tr \big( M(t)^2 \big).
\end{equation*}
We now use the inequality $\tr\big(M(t)^2\big) \geq 8 \, \L(W(t))$, and the fact that $\lambda_{min} \big(W(s)W(s)^T \big) \xrightarrow[t\to \infty]{} \mu$. Therefore:
\begin{equation*}
    \L(W(t)) \leq \L(W^0) \exp \left( - 8 \int_0^t \lambda_{min} \big(W(s)W(s)^T \big) \d s \right) = \L(W^0) \exp \big( - 8 \mu t + o(t) \big).
\end{equation*}
We now show that the $o(t)$ is in fact a $O(1)$:
\begin{equation*}
    \left| \int_0^t \lambda_{min}\big(W(s)W(s)^T \big) \d s - \mu t \right| \leq \int_0^t  \big\|W(s)W(s)^T - Z^* \big\| \d s \leq \sqrt{2} \int_0^t \sqrt{\L(W(s))} \d s,
\end{equation*}
which remains bounded as $t \to \infty$. This proves the first claim of \cref{PropCVRate}. 

\subsubsection{General case}

In the following, we suppose that $m^* < d$. To derive the convergence rates in the case $m \geq m^*$, the first step uses the implicit solution obtained in \cref{Prop:Analytic}:
\begin{equation*}
    Z(t) = e^{-2 \xi(t)} e^{2Z^*t} W^0 \underbrace{\left( I_m + 4 W^{0T} \int_0^t e^{-2 \xi(s)} e^{4Z^*s} \d s \ W^0 \right)^{-1}}_{\textstyle \equiv N(t)^{-1}} W^{0T} e^{2Z^* t},
\end{equation*}
with $\xi(t) = \psi(t) - \tr(Z^*)t$ whose behaviour is known as $t \to \infty$. We now decompose $W^0 = \begin{pmatrix} U_0 \\ V_0 \end{pmatrix}$ where $U_0 \in \R^{m^* \times m}$ and $V_0 \in \R^{(d-m^*) \times m}$. Note that under our assumption on the initialization, we have with probability one $\mathrm{rank}(W_0) = \min(m,d)$ and $\mathrm{rank}(U_0) = m^*$ (since $m \geq m^*$). In order to avoid sets of probability one, we restrict ourselves to this event. Then, from the previous equation:
\begin{equation} \label{eq:C2decomposition}
    Z(t) = e^{-2 \xi(t)} \begin{pmatrix}
        e^{2Z_0t} U_0 N(t)^{-1} U_0^T e^{2Z_0t} & e^{2Z_0t} U_0 N(t)^{-1} V_0^T \\ V_0 N(t)^{-1} U_0^T e^{2Z_0t} & V_0 N(t)^{-1} V_0^T \end{pmatrix}  \equiv \begin{pmatrix}
            A(t) & B(t) \\ B(t)^T & C(t) \end{pmatrix},
\end{equation}
with $A(t) \in \mathcal{S}_{m^*}^+(\R), B(t) \in \R^{m^* \times (d-m^*)}$ and $C(t) \in \mathcal{S}_{d-m^*}^+(\R)$. To start, we are going to bound the bottom right term. From \cref{lemma:traces}:
\begin{equation*}
    \tr\big(C(t) \big) = e^{-2 \xi(t)} \tr \big(V_0^T N(t)^{-1} V_0 \big) \leq e^{-2 \xi(t)} \tr(V_0V_0^T) \lambda_{min}(N(t))^{-1}.
\end{equation*}
Since $N(t) = I_m + 4 W^{0T} E^*(t) W^0$, the behaviour of the smallest eigenvalue of $N(t)$ can be deduced from \cref{lemma:eigenvalues} along with the behaviour of $\xi$. We get that:
\begin{equation}\label{eq:C2C(t)}
    \tr \big( C(t) \big) \underset{t \to \infty}{=} \left\{\begin{array}{ll}
        O \big(e^{- 4 \mu t} \big) & \mathrm{for} \  m =m^* \\
         O \big( t^{-1} \big)&  \mathrm{for} \ m > m^*
    \end{array} \right. ,
\end{equation}
where $\mu$ is the smallest non-zero eigenvalue of $Z^*$. Now, in order to obtain the behaviour of $A(t)$ and $B(t)$, we use the differential equation on $Z(t)$:
\begin{equation*}
    \dot{Z}(t) = - 2 \dot{\xi}(t) Z(t) + 2Z(t)(Z^*-Z(t)) + 2(Z^* - Z(t))Z(t).
\end{equation*}
Due to the form of $Z^*$ in equation \eqref{eq:C2diagZ*}, this induces the evolution equations for $A(t)$ and $B(t)$:
\begin{equation} \label{eq:C2AB}
\begin{aligned}
    \dot{A} &= -2 \dot{\xi} A + 2A(Z_0-A) + 2(Z_0-A)A - 4 BB^T \\
    \dot{B} &= -2 \dot{\xi} B + 2(Z_0-A)B - 2AB - 4BC.
\end{aligned}
\end{equation}
We begin by controlling $\big\|B(t)\big\|_F$:
\begin{equation*}
\begin{aligned}
    \frac{\d}{\d t} \big\| B \big\|_F^2 &= -2 \dot{\xi} \big\| B \big\|_F^2 + 4 \tr \big( (Z_0-A)BB^T \big) -4 \tr \big(ABB^T\big) - 8 \tr \big(BCB^T \big) \\
    &\leq \big( -2 \dot{\xi} + 4 \lambda_{max}(Z_0-A) -4 \lambda_{min}(A) \big) \big\|B \big\|_F^2.
\end{aligned}
\end{equation*}
Therefore:
\begin{equation} \label{eq:C2B(t)}
    \big\|B(t)\big\|_F^2 \leq \big\|B_0\big\|_F^2 e^{- 4 \mu t } \exp \left( - \int_0^t \alpha(s) \d s \right),
\end{equation}
with:
\begin{equation} \label{eq:C2alpha}
    \alpha(t) = 2 \dot{\xi}(t) - 4 \lambda_{max}(Z_0-A(t)) + 4 \lambda_{min}(A(t)) - 4\mu \xrightarrow[t \to \infty]{} 0,
\end{equation}
since $A(t) \xrightarrow[t \to \infty]{} Z_0$ and $\dot{\xi}(t) = \tr(Z(t) - Z^*) \xrightarrow[t \to \infty]{} 0$ (note that $\lambda_{min}(Z_0) = \mu$). We now take care of $A(t)$ by introducing:
\begin{equation*}
    \phi(t) = \big\| A(t) - Z_0 \big\|_F^2 + \frac{1}{2} \big[ \tr(A(t)-Z_0) \big]^2.
\end{equation*}
Differentiating and using the equation \eqref{eq:C2AB} on $A(t)$ :
\begin{equation*}
\begin{aligned}
    \dot{\phi} &= - 2\tr (AM^2) - 2\tr(C) \tr(AM)- 4 \tr (BB^TM) \\  
    &\leq -2\lambda_{min}(A) \tr(M^2) + 2\tr(C) \tr(A^2)^{1/2} \tr(M^2)^{1/2} + 4\tr(BB^TBB^T)^{1/2} \tr(M^2)^{1/2},
\end{aligned}
\end{equation*}
where we defined $M = 2(A-Z_0) + \tr(A-Z_0)I_{m^*}$ and used Cauchy-Schwartz inequality as well as \cref{lemma:traces}. Using the inequality $4 \phi(t) \leq \tr(M^2) \leq 2(4+m^*) \phi(t)$, we obtain that:
\begin{equation*}
    \dot{\phi}(t) \leq - 8 \lambda_{min}(A(t)) \phi(t) + 2 \underbrace{\sqrt{2(4+m^*)} \big( \tr(C(t)) \tr(A(t)) + 2 \tr(B(t)B(t)^T) \big)}_{\textstyle \equiv \delta(t)} \sqrt{\phi(t)}.
\end{equation*}
We set $\Lambda(t) = \int_0^t \lambda_{min}(A(s)) \d s \underset{t \to \infty}{\sim} \mu t$, and obtain:
\begin{equation*}
    \phi(t) \leq e^{-8\Lambda(t)} \left( \sqrt{\phi(0)} + \int_0^t e^{4 \Lambda(s)} \delta(s) \d s \right)^2.
\end{equation*}
Using the definition of $\delta$, we have two terms to bound. From equation \eqref{eq:C2B(t)}:
\begin{equation*}
    \int_0^t e^{4 \Lambda(s)} \tr \big( B(s) B(s)^T \big) \d s \leq \big\|B_0 \big\|_F^2 \int_0^t e^{4 \Lambda(s)} e^{-4 \mu s} \exp \left( - \int_0^s \alpha(u) \d u \right) \d s.
\end{equation*}
Since $\alpha(t) \xrightarrow[t \to \infty]{} 0$ and due to the behaviour of $\Lambda(t)$, this term is $o(e^{ \beta t})$ for all $\beta > 0$ thanks to \cref{lemma:asymptoticsIntegral}. For the other term, we use equation \eqref{eq:C2C(t)} and split the cases. For $m=m^*$:
\begin{equation*}
    \int_0^t e^{4 \Lambda(s)} \tr(C(s)) \tr(A(s)) \d s \underset{t \to \infty}{\lesssim} \int_0^t e^{4 \Lambda(s)} e^{-4 \mu s} \tr(A(s)) \d s,
\end{equation*}
with again is $o(e^{\beta t})$ for all $\beta > 0$. Thus, in this case, we have the bound:
\begin{equation*}
    \phi(t) \lesssim e^{-8 \mu t} e^{v(t)},
\end{equation*}
where $v(t) = o(t)$ as $t \to \infty$. In the case $m > m^*$, again with equation \eqref{eq:C2C(t)} and \cref{lemma:asymptoticsIntegral}: 
\begin{equation*}
    \int_0^t e^{4 \Lambda(s)} \tr(C(s)) \tr(A(s)) \d s \underset{t \to \infty}{\lesssim} \frac{\tr(Z_0)}{4 \mu} \frac{e^{4 \Lambda(t)}}{t},
\end{equation*}
which predominates over the first term. Therefore, for $m > m^*$, we get that $\phi(t) = O \big(t^{-2} \big)$. Finally, putting everything together:
\begin{equation*}
\begin{aligned}
    \big\|Z(t) - Z^*\big\|_F^2 &\leq \big\|A(t) - Z_0\big\|_F^2 + 2 \big\|B(t)\big\|_F^2 + \tr \big(C(t)^2 \big) \\
    &\leq \phi(t) + 2 \big\|B(t)\big\|_F^2 + \big[\tr(C(t))\big]^2.
\end{aligned}
\end{equation*}
 For $m > m^*$, the second term is negligible and we get that $\|Z(t)-Z^*\|_F^2 = O(t^{-2})$. For $m=m^*$, the first and third term are of the order $e^{-8 \mu t}$ (up to some corrections), and the leading term is given by $\big\|B(t) \big\|_F^2$. In this case:
 \begin{equation*}
     \big\| Z(t) - Z^*\big\|_F^2 \underset{t \to \infty}{=} O \left( e^{-4 \mu t} \exp \left( - \int_0^t \alpha(s) \d s \right) \right),
 \end{equation*}
 where $\alpha$ is defined in equation \eqref{eq:C2alpha}. We already have that $\int_0^t \alpha(s) \d s = o(t)$ as $t \to \infty$ so that we know the main behaviour of $\big\|Z(t) - Z^* \big\|_F$. We now show that $\int_0^t \alpha(s) \d s$ is bounded. From the expression of $\xi$:
 \begin{equation*}
\begin{aligned}
    \left| \int_0^t  \alpha(s)  \d s \right| &\leq 2 \int_0^t \big| \tr(Z(s) - Z^*) \big| \d s + 4 \int_0^t \big| \lambda_{max}(Z_0 - A(s)) \big| \d s + 4 \int_0^t \big| \lambda_{min}(A(t)) - \mu \big| \d s \\
    &\leq 2 \sqrt{d} \int_0^t \big\| Z(s) - Z^*\big\|_F \d s + 8 \int_0^t \big\| A(s) - Z_0 \big\|_F \d s.
\end{aligned}
 \end{equation*}
Due to the bounds obtained in the case $m=m^*$, the integrals converge which leads to $\big\|Z(t) - Z^*\big\|_F^2 = O \big( e^{-4 \mu t} \big)$. The result of the proposition follows from the bound:
\begin{equation*}
    \L(W(t)) = \frac{1}{2} \big\| Z(t) - Z^*\big\|_F^2 + \frac{1}{4} \big[ \tr(Z(t)-Z^*)]^2 \leq \left( \frac{1}{2} + \frac{d}{4} \right) \big\|Z(t) - Z^*\big\|_F^2.
\end{equation*}

\subsection{Proof of Proposition \ref{PropPhi}} \label{ProofPhi}

The main idea behind the proof is that $\psi_d$ is solution of the implicit equation \eqref{eq:selfpsi}:
\begin{equation} \label{eq:C3selfpsi}
    \psi_d(t) = \frac{1}{4} \tr \log \left( I_{m_d} + 4 W_d^{0T} \int_0^t e^{-2 \psi_d(s)} e^{2s \tr(Z_d^*)} e^{4Z_d^*s} \d s \ W_d^0 \right).
\end{equation}
In order to understand the high-dimensional limit of $\psi_d$ one should be able to derive the limiting distribution of the matrix inside the log, which is not easily solved since this matrix depends on $\psi_d$. However, this becomes easier under Assumption \ref{AssumptionHighdim} which states that the matrices $U_d^0U_d^{0T}$ and $U_d^* U_d^{*T}$ are uniformly drawn orthonormal projections (see \cref{App:Manova} for a more precise definition). Indeed, equation \eqref{eq:exponentialorthonormal} shows that:
\begin{equation} \label{eq:C3exponential}
    W_d^{0T} e^{4Z_d^*s} W_d^0 = \frac{1}{m_d} \left( I_{m_d} + \big( e^{4s / m_d^*} - 1 \big) U_d^{0T} U_d^* U_d^{*T} U_d^0 \right).
\end{equation}
Thus, the integrals involving $\psi_d$ can be detached from the matrices and the analysis will be simpler. The proof is split in different steps. We start by deriving properties of $\phi_d$ using the implicit equation \eqref{eq:C3selfpsi}. Through those steps, the goal is to obtain sufficient information on $\phi_d$ so that we can extract a converging subsequence (using Arzelà-Ascoli theorem). Finally, the goal will be to identify an equation which is verified in the limit (\cref{lemma:phi4}) and show that its solution is unique (\cref{lemma:phi5}). 

In the following, we set $T > 0$ and denote $\mathcal{C}_T$ the space of continuous functions on $[0,T]$ equipped with the norm:
\begin{equation*}
    \| \chi \|_T = \sup_{\gamma \in [0,T]} |\chi(t)|.
\end{equation*}

\subsubsection{Useful bounds}

As mentioned earlier, the two following lemmas allow to gather sufficient information on $\phi_d$ in order to extract converging subsequences in $\mathcal{C}_T$ for $T > 0$. The first naive bound we obtain uses the fact that the loss function is always decreasing along a gradient flow. 

\begin{lemma} \label{lemma:phi1}
    There exists $\kappa > 0$, such that:
    \begin{equation*}
        \sup_{\gamma \geq 0} \big| \dot{\phi}_d(\gamma) \big| \leq \kappa \sqrt{d}.
    \end{equation*}
\end{lemma}

\begin{proof}
    This property is a consequence of the gradient flow structure. Indeed, following equation \eqref{eq:phi}:
    \begin{equation*}
        \phi_d(\gamma) = \int_0^\gamma \tr \big(W_d(s)W_d(s)^T - W_d^* W_d^{*T} \big) \d s,
    \end{equation*}
    where $W_d(t)$ is solution of the gradient flow \eqref{eq:flow} with initial condition $W_d^0$. Now, deriving the loss with respect to time:
    \begin{equation*}
        \frac{\d}{\d t} \L(W_d(t)) = - \big\| \nabla \L(W_d(t)) \big\|^2 \leq 0.
    \end{equation*}
    Thus, $\L(W_d(t))$ in non-increasing in time, and from the expression of the loss in equation \eqref{eq:loss}:
    \begin{equation*}
        \big[ \tr \big( W_d(t)W_d(t)^T - W_d^* W_d^{*T} \big)]^2 \leq 4 \L (W_d(t)) \leq 4 \L(W_d^0).
    \end{equation*}
    We now use the orthonormality assumption:
    \begin{equation*}
    \begin{aligned}
        \L(W_d^0) &= \frac{1}{2} \left\| \frac{1}{m_d} U_d^0 U_d^{0T} -\frac{1}{m_d^*} U_d^* U_d^{*T} \right\|_F^2 + \underbrace{\frac{1}{4} \left( \frac{1}{m_d} \tr \big( U_d^0 U_d^{0T}\big) - \frac{1}{m_d^*} \tr \big( U_d^* U_d^{*T} \big) \right)^2 }_{\textstyle = 0} \\
        &= \frac{1}{2} \left( \frac{1}{m_d^2} \tr \big(U_d^0 U_d^{0T} U_d^0 U_d^{0T} \big) - \frac{2}{m_d m_d^*} \tr \big( U_d^0 U_d^{0T} U_d^* U_d^{*T} \big) + \frac{1}{m_d^{*2}} \tr \big(  U_d^* U_d^{*T} U_d^* U_d^{*T} \big) \right) \\
        &\leq \frac{1}{2} \left( \frac{1}{m_d} + \frac{1}{m_d^*} \right).
    \end{aligned}
    \end{equation*}
    Since $m_d / d \xrightarrow[d \to \infty]{} \alpha$ and $m^*_d / d \xrightarrow[d \to \infty]{} \alpha^*$, we obtain that:
    \begin{equation*}
        \big| \tr \big(W_d(t)W_d(t)^T - W_d^* W_d^* \big) \big| \leq \frac{\kappa}{\sqrt{d}},
    \end{equation*}
    for some $\kappa$ independent from $t$ and $d$. Since $\dot{\phi}_d(\gamma) = d \, \tr \big(W_d(\gamma d)W_d(\gamma d)^T - W_d^* W_d^* \big)$, we obtain the desired result. 
\end{proof}

\begin{lemma} \label{lemma:phi2}
    For all $T > 0$:
    \begin{equation*}
        \sup_{d \in \mathbb{N}} \  \sup_{\gamma \in [0,T]} \big| \phi_d(\gamma) \big| < \infty. 
    \end{equation*}
    Moreover, the family $(\phi_d)_{d \in \mathbb{N}}$ is equicontinuous on $[0,T]$, that is:
    \begin{equation}
        \sup_{d \in \mathbb{N}} \  \sup_{\substack{\gamma, \gamma' \in [0,T] \\ |\gamma-  \gamma'| \leq \eta}} \big| \phi_d(\gamma) - \phi_d(\gamma') \big| \xrightarrow[ \eta \to 0]{} 0.
    \end{equation}
\end{lemma}

This lemma will be proven in \cref{App:proofLemma1}. Indeed, the proof is long and do not carry many elements of interest. The main idea is to start from the result of \cref{lemma:phi1} and use it in the self-consistent equation solved by $\phi_d$, which we prove to be:
\begin{equation} \label{eq:C3selfphi}
    \phi_d(\gamma) + \gamma d = \frac{1}{4} \tr \log \left( \left(1 + \frac{4d}{m_d} \int_0^\gamma e^{-2 \phi_d(s)} \right) I_{m_d} + \frac{4d}{m_d} \int_0^\gamma e^{-2 \chi(s)} \big( e^{4s d/m_d^*} - 1 \big) \d s \ U_d^{0T} U_d^* U_d^{*T} U_d^0 \right).
\end{equation}

\vspace*{-0.2 cm}
\subsubsection{Identifying the limit}

The previous lemma shows that the family $(\phi_d)_{d \in \mathbb{N}}$ is compact in $\mathcal{C}_T$, i.e.,  it allows us to extract converging subsequences. The following lemmas will show that those subsequences always converge to the same function. 
\begin{lemma} \label{lemma:phi4}
    Let $\mu$ be defined as in equation \eqref{eq:manova} and $\phi : [0,T] \rightarrow \R$ a subsequential limit of $(\phi_d)_{d \in \mathbb{N}}$ in $\mathcal{C}_T$. Then, with probability one, for all $\gamma \in [0,T]$:
    \begin{equation} \label{eq:C3unique}
        \int \log \left( 1 + \frac{4}{\alpha} \int_0^\gamma e^{-2 \phi(s)} \d s + \frac{4x}{\alpha} \int_0^\gamma e^{-2 \phi(s)} \big( e^{4s / \alpha^*} - 1 \big) \d s \right) \d \mu(x) = \frac{4 \gamma}{\alpha}.
    \end{equation}
\end{lemma}
\begin{proof}
For $\xi \in \mathcal{C}_T$, we set:
\begin{equation*}
    \H_d(\xi)(\gamma) = \frac{1}{4d} \tr \log \left( \left( 1 + \frac{4d}{m_d} \int_0^\gamma e^{-2 \xi(s)} \d s \right) I_{m_d} + \frac{4d}{m_d} \int_0^\gamma e^{-2 \xi(s)} \big( e^{4s d/m_d^*} - 1 \big) \d s \, Y_d \right).
\end{equation*}
From equation \eqref{eq:C3selfphi}, we have that, for all $\gamma \geq 0$: 
\begin{equation} \label{eq:C3phiu}
    \H_d(\phi_d)(\gamma) = \gamma + \frac{\phi_d(\gamma)}{d} \equiv u_d(\gamma).
\end{equation}
From the bound obtained in \cref{lemma:phi1}, we have that $\sup_{\gamma \in [0,T]} | \phi_d(\gamma) | \leq \kappa T \sqrt{d}$, therefore $u_d(\gamma) \xrightarrow[d \to \infty]{} \gamma \equiv u(\gamma)$ uniformly on $[0,T]$. We know that the empirical spectral distribution of $Y_d$ converges towards $\mu$ defined in equation \eqref{eq:manova}. Therefore, we also define:
\begin{equation} \label{eq:C3defH}
    \H(\xi)(\gamma) = \frac{\alpha}{4} \int \log  \left( 1 + \frac{4}{\alpha} \int_0^\gamma e^{-2 \xi(s)} \d s  + \frac{4x}{\alpha} \int_0^\gamma e^{-2 \xi(s)} \big( e^{4s / \alpha^*} - 1 \big) \d s \right) \d \mu(x).
\end{equation}
It is reasonable to hope that $\H_d(\xi) \longrightarrow \H(\xi)$ from the convergence of $Y_d$ and the fact that $m_d/d \longrightarrow \alpha$ and $m_d^*/d \longrightarrow \alpha^*$ as $d \to \infty$. To do so, we introduce the following lemma that we prove in \cref{App:proofphi3}:

\begin{lemma} \label{lemma:phi3}
    For $c \in \R$, let $B_c = \big\{ \xi \in \mathcal{C}_T, \, \xi \geq c \big\}$. Then, with probability one:
    \begin{equation} \label{eq:C3uniformCV}
    \sup_{\xi \in B_c} \big\| \H_d(\xi) - \H(\xi) \big\|_T \xrightarrow[d \to \infty]{} 0.
    \end{equation}
\end{lemma}
The main ingredient of this result is the uniform convergence obtained in \cref{lemma:largedev} for the empirical spectral distribution of the matrix $Y_d$. However, we need to be careful since there are other quantities depending on the dimension. 

Now, to prove \cref{lemma:phi4}, we take an extraction $d(n)$ and suppose that $\phi_{d(n)}$ converges towards some $\phi$ is $\mathcal{C}_T$. From equation \eqref{eq:C3phiu}, we have that $\H_{d(n)}( \phi_{d(n)}) = u_{d(n)} \xrightarrow[n \to \infty]{} u$. We are going to show that with probability one, $\H(\phi) = u$. This will give the desired result from the definition of $\H$ in equation \eqref{eq:C3defH}. We have:
\begin{equation*}
\big\| \H(\phi) - u \big\|_T \leq \big\| \H(\phi) - \H(\phi_{d(n)}) \big\|_T + \big\| \H(\phi_{d(n)}) - \H_{d(n)}(\phi_{d(n)})\big\|_T + \big\| \H_{d(n)}(\phi_{d(n)}) - u \big\|_T.
\end{equation*}
We already know that the last term goes to zero. For the first one, we need to show that $\H$ is continuous. Indeed, for $\xi, \varsigma \in \mathcal{C}_T$ both lower bounded by some $c \in \R$:
\begin{equation*}
\begin{aligned}
    \big| \H(\xi)(\gamma) - \H(\varsigma)(\gamma) \big| &\leq \frac{\alpha}{4} \sup_{x \in [0,1]} \Bigg| \log \left(  1 + \frac{4}{\alpha} \int_0^\gamma e^{-2 \xi(s)} \d s  + \frac{4x}{\alpha} \int_0^\gamma e^{-2 \xi(s)} \big( e^{4s / \alpha^*} - 1 \big) \d s \right) \\
    &- \log \left(  1 + \frac{4}{\alpha} \int_0^\gamma e^{-2 \varsigma(s)} \d s  + \frac{4x}{\alpha} \int_0^\gamma e^{-2 \varsigma(s)} \big( e^{4s / \alpha^*} - 1 \big) \d s \right) \Bigg| \\
    &\leq \sup_{x \in [0,1]} \int_0^\gamma \big| e^{-2 \xi(s)} - e^{-2 \varsigma(s)} \big| \d s + x \int_0^\gamma  e^{4s / \alpha^*} \big| e^{-2 \xi(s)} - e^{-2 \varsigma(s)} \big| \d s \\
    &\leq 2e^{-2c} T (1 + e^{4T / \alpha^*}) \big\| \xi - \varsigma \|_T.
\end{aligned}
\end{equation*}
Since by \cref{lemma:phi2}, $(\phi_d)_{d \in \mathbb{N}}$ is uniformly bounded on $[0,T]$, it is in $B_c$ for some $c \in \R$. This implies that the first term goes to zero. From equation \eqref{eq:C3uniformCV}, this also implies that the second term goes to zero with probability one. As a conclusion, we necessarily have that $\H(\phi) = u$ which proves the lemma. 
\end{proof}

\begin{lemma} \label{lemma:phi5}
    Equation \eqref{eq:C3unique} has a unique continuous solution on $\R^+$. 
\end{lemma}

This lemma is technical and uses the Picard-Lindelöf theorem which gives the existence and uniqueness for the solutions of differential equations. First, equation \eqref{eq:C3unique} has to be interpreted as a differential equation, which can be done using equation \eqref{eq:selfphi}. We prove this lemma in \cref{App:proofphi5}. 

The proof of \cref{PropPhi} is now elementary. From \cref{lemma:phi2} and \cref{lemma:phi3}, the family $(\phi_d)_{d \in \mathbb{N}}$ is compact in $\mathcal{C}_T$. Thus, it admits at least one subsequential limit by Arzelà-Ascoli theorem. From \cref{lemma:phi4}, such a limit verifies $\H(\phi) = u$ with probability one. Since this equation admits a unique solution from \cref{lemma:phi5}, $(\phi_d)_{d \in \mathbb{N}}$ has a single subsequential limit with probability one. Thus, outside the event of zero probability:
\begin{equation*}
\left\{ \sup_{\xi \in B_c} \big\| \H_d(\xi) - \H(\xi) \big\|_T \underset{d \to \infty}{\not\longrightarrow} 0 \right\},
\end{equation*}
$(\phi_d)_{d \in \mathbb{N}}$ uniformly converges on $[0,T]$. Here $c \in \R$ is such that $(\phi_d)_{d \in \mathbb{N}}$ is contained in $B_c$ (exists and independent of the chosen extraction from \cref{lemma:phi2}). To conclude, taking a sequence $T_n \xrightarrow[n \to \infty]{} \infty$, this proves that almost surely, $(\phi_d)_{d \in \mathbb{N}}$ uniformly converges on each $[0,T_n]$, thus on every compact of $\R^+$. From \cref{lemma:phi4}, the limit function $\phi$ is solution of equation \eqref{eq:C3unique}. Replacing $\mu$ from its definition in equation \eqref{eq:manova} leads to equation \eqref{eq:selfphi} which concludes the proof of \cref{PropPhi}. 

\subsection{Proof of Proposition \ref{PropOverlap}} \label{ProofOverlap}

\cref{PropOverlap} is decomposed into two parts: first we show the uniform convergence of the overlap between the teachers and students at timescale $t = \gamma d$. Then we derive the limit of the overlap as $\gamma \to \infty$. Note that our proof method will also give access to the speed at which the convergence occurs as $\gamma \to \infty$. This will be to be compared with the convergence rates we obtained at finite dimension (see \cref{PropCVRate}).

\subsubsection{High-dimensional limit for the overlap}

We now prove the first part of the proposition, namely the overlap uniformly converges as $d \to \infty$. This result is a consequence of \cref{PropPhi} as well as the convergence of the empirical spectral distribution of the matrix $Y_d$ as $d \to \infty$. We first define some quantities that will be useful in the following:
\begin{equation*}
    F_d(\gamma) = 1 + \frac{4d}{m_d} \int_0^\gamma e^{-2 \phi_d(s)} \d s, \hspace{1 cm}  G_d(\gamma) = 1 + \frac{4d}{m_d} \int_0^\gamma e^{-2 \phi_d(s)} e^{4sd / m_d^*} \d s, \hspace{1 cm} J_d(\gamma) = \frac{G_d(\gamma)}{F_d(\gamma)}.
\end{equation*}
Since it has been shown that $\phi_d$ converges as $d \to \infty$, it is easily shown that these quantities also converge as $d \to \infty$. We denote $F, G$ and $J$ their infinite dimensional counterparts. 

\begin{lemma} \label{lemma:overlap1}
Let $\chi_d(\gamma)$ be defined in \cref{PropOverlap}. Then:
\begin{equation} \label{eq:C4overlap}
    \chi_d(\gamma) =  \sqrt{\frac{m_d}{m_d^*}} e^{4 \gamma d /m_d^* } \frac{m_d^{-1} \tr \Big( Y_d \big(I_{m_d} + (J_d(\gamma)-1) Y_d\big)^{-1} \Big)}{\sqrt{m_d^{-1} \tr \Big( \big(I_{m_d} + (e^{4\gamma d / m_d^*} - 1) Y_d \big)^2 \big(I_{m_d} + (J_d(\gamma)-1) Y_d \big)^{-2} \Big)}},
\end{equation}
with $Y_d = U_d^{0T} U_d^* U_d^{*T} U_d^0$. 
\end{lemma}

Before proving this lemma, we observe that at finite dimension, the overlap mainly depends on two quantities: the matrix $Y_d$ (which is reasonable due to our orthonormality assumption) and the function $J_d(\gamma)$. This function also converges as $d \to \infty$, therefore this previous equation suggests that the overlap will also converge. 

\begin{proof}
Using the implicit solution obtain in \cref{Prop:Analytic}:
\begin{equation*}
    Z_d(t) = e^{-2 \psi_d(t)} e^{2 \tr(Z_d^*)t} e^{2Z_d^*t} W^0_d \left( I_{m_d} + 4 \int_0^t e^{-2 \psi_d(t)} e^{2s \tr(Z_d^*)} W_d^{0T} e^{4 Z_d^* s} W_d^0 \d s \right)^{-1} W_d^{0T} e^{2Z_d^*t}.
\end{equation*}
Introducing $\gamma = t/d$ and using the definitions of $\phi_d(\gamma)$, $U_d^0$, $U_d^*$:
\begin{equation*}
    Z_d(\gamma d) = \frac{1}{m_d} e^{-2 \phi_d(\gamma)} \exp \left( \frac{2d \gamma}{m_d^*} U_d^* U_d^{*T} \right) U_d^0 N_d(\gamma)^{-1} U_d^{0T} \exp \left( \frac{2d \gamma}{m_d^*} U_d^* U_d^{*T} \right) ,
\end{equation*}
with:
\begin{equation*}
\begin{aligned}
    N_d(\gamma) &= \left( 1 + \frac{4d}{m_d} \int_0^\gamma e^{-2 \phi_d(s)} \d s \right) I_{m_d} + \frac{4d}{m_d} \int_0^\gamma e^{-2 \phi_d(s)} \big( e^{4sd/ m_d^*} - 1 \big) \d s \ Y_d \\
    &\equiv F_d(\gamma) + \big( G_d(\gamma) - F_d(\gamma) \big) Y_d.
\end{aligned}
\end{equation*}
Now using the fact that $U_d^* U_d^{*T}$ is a projection:
\begin{equation*}
\exp \left( \frac{2d \gamma}{m_d^*} U_d^* U_d^{*T} \right) = I_d - U_d^* U_d^{*T} + \exp \left( \frac{2d \gamma}{m_d^*} \right) U_d^*U_d^{*T}.
\end{equation*}
And denoting $\lambda_d^* = d/m_d^*$, we obtain:
\begin{equation} \label{eq:C4formula1}
\begin{aligned}
    \tr \big( Z_d^* Z_d(\gamma d) \big) &= \frac{1}{m_d m_d^*} e^{-2 \phi_d(\gamma)} \tr \Big( U_d^* U_d^{*T} \big( I_d + (e^{2 \lambda_d^* \gamma} - 1) U_d^* U_d^{*T} \big) U_d^0 N_d(\gamma)^{-1} U_d^{0T} \big( I_d + (e^{2 \lambda_d^* \gamma} - 1) U_d^* U_d^{*T} \big) \Big) \\
    &= \frac{1}{m_d m_d^*} e^{-2 \phi_d(\gamma)} e^{4 \lambda_d^* \gamma} \tr \Big( Y_d \big( F_d(\gamma) I_{m_d} + \big( G_d(\gamma) - F_d(\gamma) \big) Y_d \big)^{-1} \Big).
\end{aligned}
\end{equation}
Now, to compute the overlap, one has to compute the norms $\big\|Z_d^*\big\|_F$ and $\big\| Z_d(\gamma d) \big\|_F$. The first one is easily done:
\begin{equation}  \label{eq:C4formula2}
    \big\|Z_d^*\big\|_F^2 = \frac{1}{m_d^{*2}} \tr \Big( U_d^*U_d^{*T} U_d^*U_d^{*T} \Big) = \frac{1}{m_d^*}.
\end{equation}
For the second one, one can do a similar calculation as for the trace:
\begin{equation}  \label{eq:C4formula3}
    \big\|Z_d(\gamma d) \big\|_F^2 = \frac{1}{m_d^2} e^{-4 \phi_d(\gamma)} \tr \Big( \big(I_{m_d} + (e^{4 \lambda_d^* \gamma} - 1 ) Y_d \big)^2 \big( F_d(\gamma) I_{m_d} + \big( G_d(\gamma) - F_d(\gamma) \big) Y_d \big)^{-2} \Big).
\end{equation}
Assembling equation \eqref{eq:C4formula1}, \eqref{eq:C4formula2} and \eqref{eq:C4formula3} gives the desired result. 
\end{proof}

The goal is now to prove the uniform convergence of the overlap in the high-dimensional limit. This is achieved by the following lemma:
\begin{lemma} \label{lemma:overlap2}
    Let $T > 0$. Then, uniformly on $[0,T]$:
    \begin{equation*}
    \begin{aligned}
    \frac{1}{m_d} \tr \Big( Y_d \big(I_{m_d} + (J_d(\gamma)-1) Y_d\big)^{-1} \Big) &\xrightarrow[d \to \infty]{} \int \frac{x}{1 + (J(\gamma) - 1)x} \d \mu(x) \equiv A(\gamma) \\
    \frac{1}{m_d} \tr \Big( \big(I_{m_d} + (e^{4 \lambda_d^* \gamma} - 1) Y_d \big)^2 \big(I_{m_d} + (J_d(\gamma)-1) Y_d \big)^{-2} \Big) &\xrightarrow[d \to \infty]{} \int \left( \frac{1 + \big( e^{4 \gamma / \alpha^*} - 1 \big)x}{1 + (J(\gamma)-1)x} \right)^2 \d \mu(x) \\
    &\hspace{3 cm}\equiv B(\gamma).
    \end{aligned}
    \end{equation*}
    As a consequence, uniformly on $[0,T]$:
    \begin{equation} \label{eq:C4limitxi}
        \chi_d(\gamma) \xrightarrow[d \to \infty]{} \sqrt{\frac{\alpha}{\alpha^*}} e^{4 \gamma / \alpha^*} \frac{A(\gamma)}{\sqrt{B(\gamma)}} \equiv \chi(\gamma).
    \end{equation}
\end{lemma}

The proof of this technical lemma is deferred to \cref{App:proofoverlap2}. The main ingredients are the fact that $J_d$ uniformly converges towards $J$, as well as the uniform convergence result for the empirical spectral distribution of the matrix $Y_d$ presented in \cref{lemma:largedev}.

This last formula give the overlap in the high-dimensional limit. This depends on $\alpha, \alpha^*$ (also through the measure $\mu$ defined in equation \eqref{eq:manova}) and on the function $J$. In the following, we obtain an understanding of this function as $\gamma \to \infty$, which leads to the determination of the asymptotic behaviour of the overlap.

\subsubsection{Limit and convergence rate for the overlap}

We now prove the second part of \cref{PropOverlap} and determine in addition at which rate the convergence $\underset{\gamma \to \infty}{\lim} \chi(\gamma)$ occurs. From the result of \cref{lemma:overlap2}, the main challenge is to determine the behaviour of $J(\gamma)$ as $\gamma \to \infty$. This is done in the following lemma. We remind the definitions of $F,G,J$, depending on the function $\phi = \underset{d \to \infty}{\lim} \phi_d$:
\begin{equation} \label{eq:C4defFGJ}
    F(\gamma) = 1 + \frac{4}{\alpha} \int_0^\gamma e^{-2 \phi(s)} \d s \hspace{1 cm} G(\gamma) = 1 + \frac{4}{\alpha} \int_0^\gamma e^{-2 \phi(s)}  e^{4s / \alpha^*} \d s \hspace{1 cm} J(\gamma) = \frac{G(\gamma)}{F(\gamma)}.
\end{equation}

\begin{lemma} \label{lemma:overlap3}
    There exists a constant $\kappa(\alpha, \alpha^*) > 0$ such that:
\begin{equation*}
J(\gamma) \underset{\gamma \to \infty}{\sim}    \left\{ \begin{array}{cc}
\vspace{0.1 cm}
\kappa(\alpha, \alpha^*)  e^{4 \gamma / \alpha^*} & \mathrm{for} \ \alpha < \alpha^* \\
\kappa(\alpha, \alpha^*) \gamma^{-1} e^{4 \gamma / \alpha^*} & \mathrm{for} \ \alpha \geq \alpha^*.
\end{array} \right.
\end{equation*}
\end{lemma}

This technical lemma is proven in \cref{App:proofoverlap3}. It mainly uses the implicit equation on $\phi$ from \cref{PropPhi}. 

The understanding of the asymptotics of $J$, along with the limit we identified in \cref{lemma:overlap2} allows to determine the behaviour of $\chi(\gamma)$ as $\gamma \to \infty$. This is a stronger version of \cref{PropOverlap} as we also determine the rate at which convergence occurs. Note that, unlike the convergence rates we determine in the finite dimensional case (see \cref{PropCVRate}), the asymptotics we obtain in the following are exact. 

Before obtaining the asymptotic behaviour of $\chi(\gamma)$ as $\gamma \to \infty$, we need a last technical lemma in order to obtain the asymptotics of $A(\gamma), B(\gamma)$ as $\gamma \to \infty$ (defined in \cref{lemma:overlap2}).

\begin{lemma} \label{lemma:overlap4}
    Let $A(\gamma)$, $B(\gamma)$ be defined in \cref{lemma:overlap3}. Then:
    \begin{equation*}
        A(\gamma) \underset{\gamma \to \infty}{=} \left\{ \begin{array}{cc}
        \vspace{0.2 cm}
            \dfrac{\min(\alpha, \alpha^*)}{\alpha J(\gamma)} + O \left( \dfrac{1}{J(\gamma)^2} \right) & \mathrm{for} \  \alpha \neq \alpha^* \\
            \dfrac{1}{J(\gamma)} +  O \left( \dfrac{1}{J(\gamma)^{3/2}} \right)& \mathrm{for} \  \alpha = \alpha^* 
        \end{array} \right. ,
    \end{equation*}
    and:
        \begin{equation*}
        B(\gamma) \underset{\gamma \to \infty}{=} \left\{ \begin{array}{cc}
        \vspace{0.2 cm}
            \left( 1 - \dfrac{\alpha^*}{\alpha} \right)^+ + \dfrac{e^{8 \gamma / \alpha^*}}{J(\gamma)^2} \left(\dfrac{\min(\alpha, \alpha^*)}{\alpha} + O \left( \dfrac{1}{J(\gamma)} \right) \right) & \mathrm{for} \  \alpha \neq \alpha^* \\
            \dfrac{e^{8 \gamma / \alpha^*}}{J(\gamma)^2} \left( 1  +  O \left( \dfrac{1}{J(\gamma)^{1/2}} \right) \right)& \mathrm{for} \  \alpha = \alpha^* .
        \end{array} \right.
    \end{equation*}
    
\end{lemma}
This lemma is proven in \cref{App:proofoverlap4}. Using this result, we are now ready to establish $\underset{\gamma \to \infty}{\lim} \chi(\gamma)$ as well as the convergence rate:

\begin{lemma} \label{lemma:overlap5}
    The asymptotic behaviour of $\chi(\gamma)$ is given by:
    \begin{equation*}
        \chi(\gamma) \underset{\gamma \to \infty}{=} \left\{ \begin{array}{cc}
        \vspace{0.1 cm}
           \sqrt{\dfrac{\alpha}{\alpha^*}} + O \big( e^{-4 \gamma / \alpha^*} \big)  & \mathrm{for} \ \alpha < \alpha^*  \\
           \vspace{0.1 cm}
           1 + O \big( \sqrt{\gamma} e^{-2 \gamma / \alpha^*} \big)  & \mathrm{for} \ \alpha = \alpha^* \\
           1 + O \big( \gamma^{-2} \big) & \mathrm{for} \ \alpha > \alpha^*.
        \end{array} \right.
    \end{equation*}
\end{lemma}

\begin{proof}
    The proof is a simple consequence of the two previous lemmas and equation \eqref{eq:C4limitxi} giving the expression of $\chi(\gamma)$. Starting with the case $\alpha < \alpha^*$, we obtain:
    \begin{equation*}
        \chi(\gamma) \underset{\gamma \to \infty}{=} \sqrt{\frac{\alpha}{\alpha^*}} \frac{1 + O \big(J(\gamma)^{-1} \big)}{\sqrt{1 + O \big( J(\gamma)^{-1}  \big)}} = \sqrt{\frac{\alpha}{\alpha^*}} + O \big( J(\gamma)^{-1} \big),
    \end{equation*}
    and the result from \cref{lemma:overlap3}. Likewise, for $\alpha = \alpha^*$:
        \begin{equation*}
        \chi(\gamma) \underset{\gamma \to \infty}{=} \frac{1 + O \big(J(\gamma)^{-1/2} \big)}{\sqrt{1 + O \big( J(\gamma)^{-1/2}  \big)}} = 1 + O \big( J(\gamma)^{-1/2} \big).
    \end{equation*}
    Finally, if $\alpha > \alpha^*$:
    \begin{equation*}
        \chi(\gamma) \underset{\gamma \to \infty}{=} \sqrt{\frac{\alpha^*}{\alpha}} \frac{1 + O \big( J(\gamma)^{-1} \big)}{ \sqrt{\left(1 - \dfrac{\alpha^*}{\alpha} \right) \dfrac{J(\gamma)^2}{e^{8 \gamma / \alpha^*}} + \dfrac{\alpha^*}{\alpha} + O \big( J(\gamma)^{-1}\big)}} =  \frac{1 + O \big(J(\gamma)^{-1} \big)}{\sqrt{1 + O \big( \gamma^{-2}  \big)}} = 1 + O \big( \gamma^{-2} \big),
    \end{equation*}
    which concludes the proof. 
\end{proof}

 \begin{figure} [ht]
    \begin{minipage}{0.58 \textwidth}
        \includegraphics[width=\linewidth]{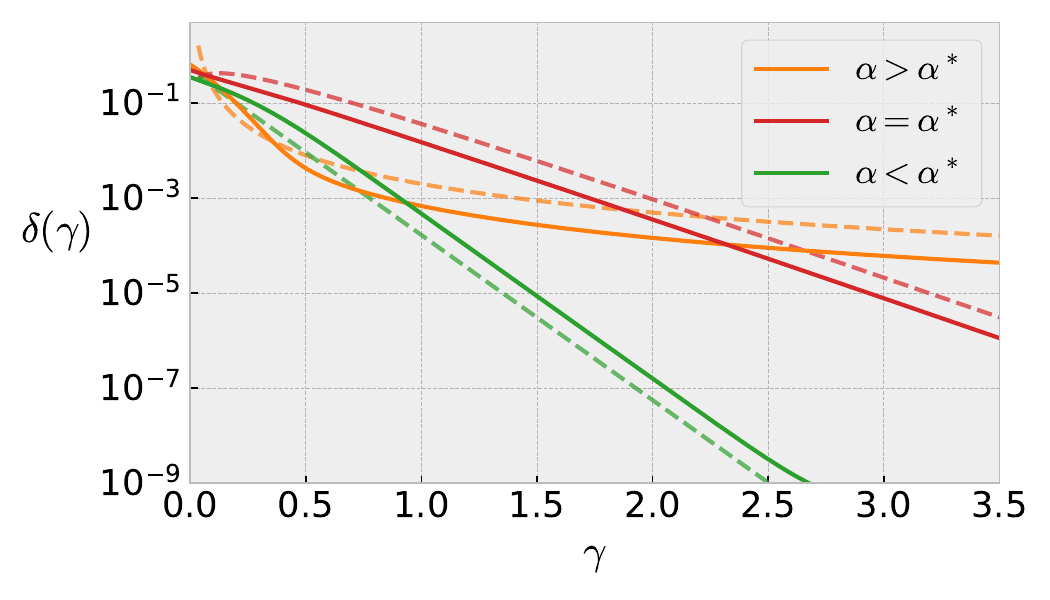}
    \end{minipage}
    \hfill
    \raisebox{0.2ex}{
    \begin{minipage}{0.38 \textwidth}
        \centering
    \caption{Evolution of $\delta(\gamma) = \chi_\infty - \chi(\gamma)$ (where $\chi_\infty = \lim_{\gamma \to \infty} \chi(\gamma)$) as a function of $\gamma$ for different values of $\alpha, \alpha^*$. Dashed lines: convergence rates obtained in \cref{lemma:overlap5}. Simulated using equation \eqref{eq:C4limitxi}, with standard discretization for the integral computation ($\alpha \neq \alpha^*$) and exact solution ($\alpha = \alpha^*$). 
     }
         \label{fig:Overlap}
    \end{minipage} }
 \end{figure}

\cref{fig:Overlap} compares a simulated version of equation \eqref{eq:C4limitxi} with the bounds obtained in the previous lemma. As expected, those bounds are tight since they were obtained using an exact asymptotic development.

This last result is to be compared with the convergence rates obtained in finite dimension (see \cref{PropCVRate}). In the overparameterized case $m \geq m^*$, we obtained $Z(t) \xrightarrow[t \to \infty]{} Z^*$. Regarding the overlap:
\begin{equation*}
\begin{aligned}
    \frac{\tr(Z(t)Z^*)}{\|Z(t)\|_F \|Z^*\|_F} &= \frac{1}{2} \left( \frac{\|Z(t)\|_F}{\|Z^*\|_F} +  \frac{\|Z^*\|_F}{\|Z(t)\|_F} - \frac{\|Z(t) - Z^*\|_F^2}{\|Z^*\|_F \|Z(t)\|_F}\right) \\
    &= 1 + \Theta \big( \|Z(t)-Z^*\|_F \big)
\end{aligned}
\end{equation*}
From the rates obtained in \cref{PropCVRate}, we obtain that the overlap is $1 + O(t^{-1})$ whenever $m > m^*$. In the high-dimensional limit, the convergence is faster for $\alpha > \alpha^*$ from \cref{lemma:overlap4}. As for the case $m=m^*$, the convergence at finite dimension was exponential and the speed was proportional to the smallest non-zero eigenvalue of $Z^*$. In our orthonormal setup, this smallest eigenvalue is given by $(m^*)^{-1}$. This leads to:
\begin{equation*}
    \frac{\tr(Z(t)Z^*)}{\|Z(t)\|_F \|Z^*\|_F} = 1 + O \left( e^{-2t / m^*} \right).
\end{equation*}
In the high-dimensional limit where $t = \gamma d$ and $m^*/ d \xrightarrow[]{} \alpha^*$, we obtain a rate $O(e^{-2\gamma / \alpha^*})$ which is the one obtained in \cref{lemma:overlap4} up to a correction proportional to $\sqrt{\gamma}$.

\vspace{1 cm}

\section{Proofs of the Technical Lemmas}
\subsection{Proof of Lemma \ref{lemma:eigenvalues}} \label{App:proofeigenvalues}

For $k = 1, \dots, d$, we set $z_k = W^{0T} v_k^* \in \R^m$, so that:
\begin{equation*}
    W^{0T} E^*(t) W^0 = \sum_{k=1}^d \gamma_k(t) z_kz_k^T,
\end{equation*}
with:
\begin{equation} \label{eq:D1gamma}
    \gamma_k(t) = \int_0^t e^{-2\psi(s)} e^{2s \tr(Z^*)} e^{4 \mu_k s} \d s \underset{t \to \infty}{\sim} \frac{1}{4 \mu_k} e^{-2 \psi(t)} e^{2t \tr(Z^*)} e^{4 \mu_k t},
\end{equation}
whenever $\mu_k > 0$. This is a consequence of \cref{lemma:asymptoticsIntegral} along with the fact that $\dot{\psi}(t) = \tr ( W(t)W(t)^T) \longrightarrow \tr(Z^*)$ as $t \to \infty$. When $\mu_k = 0$, then $\gamma_k(t) = o \big( e^{\beta t} \big)$ for all $\beta > 0$. Thus, is $\mu_k > \mu_j$, we have that $\gamma_k(t) \gamma_j(t)^{-1} \longrightarrow \infty$ as $t \to \infty$. Thus, we have some $t_0 > 0$ such that $\gamma_1(t) \geq \dots \geq \gamma_d(t) \geq 0$ as soon as $t \geq t_0$. Using Courant-Fischer formula (\cref{lemma:courantfischer}), we have for $j\in \llbracket 1, \min(m,d) \rrbracket$:
\begin{equation} \label{eq:C2courant}
    \lambda_j(t) =\max_{\substack{V \subset \R^m \\ \mathrm{dim} V = j}} \min_{\substack{x \in V \\\|x\|=1}} \left( \sum_{k=1}^d \gamma_k(t)  (x \. z_k)^2 \right) =  \min_{\substack{V \subset \R^m \\ \mathrm{dim} V = m-j+1}} \max_{\substack{x \in V \\\|x\|=1}} \left( \sum_{k=1}^d \gamma_k(t)  (x \. z_k)^2 \right).
\end{equation}
We start by using the second equality and choose $V_1 = \mathrm{Span}(z_1, \dots, z_{j-1})^\perp$ which is of dimension $m-j+1$ by assumption. Then, for $t \geq t_0$:
\begin{equation*}
    \lambda_j(t) \leq \max_{\substack{x \in V_1 \\\|x\|=1}} \left( \sum_{k=j}^d \gamma_k(t)  (x \. z_k)^2 \right) \leq \gamma_j(t) \underbrace{\max_{\substack{x \in V_1 \\ \|x\|=1}} \sum_{k=j}^d (x \. z_k)^2}_{\textstyle \equiv C_j}.
\end{equation*}
We now use the first equality of equation \eqref{eq:C2courant}. We choose $V_2 = \mathrm{Span}(z_1, \dots, z_j)$ which is of dimension $j$ by assumption. Then:
\begin{equation*}
    \lambda_j(t) \geq \min_{\substack{x \in V_2 \\\|x\|=1}} \left( \sum_{k=1}^d \gamma_k(t)  (x \. z_k)^2 \right) \geq \min_{\substack{x \in V_2 \\\|x\|=1}} \left( \sum_{k=1}^j \gamma_k(t)  (x \. z_k)^2 \right) \geq \gamma_j(t) \underbrace{\min_{\substack{x \in V_2 \\\|x\|=1}} \left( \sum_{k=1}^j  (x \. z_k)^2 \right)}_{\textstyle \equiv D_j},
\end{equation*}
for $t \geq t_0$. By definition of $V_2$, $D_j$ cannot be zero. Equation \eqref{eq:D1gamma} allows to conclude the proof. 

\vspace{0.5 cm}
\subsection{Proof of Lemma \ref{lemma:phi2}} \label{App:proofLemma1}

We divide the proofs into two parts. We first prove that the family $(\phi_d)_{d \in \mathbb{N}}$ is uniformly bounded, then we show it is equicontinuous.

\subsubsection{Uniform bound}

To prove the first point, we start by determining the implicit equation solved by $\phi_d$. Using equation \eqref{eq:C3selfpsi} making the substitution $s \mapsto s/d$:
    \begin{equation*}
    \begin{aligned}
        \psi_d(t) &= \frac{1}{4} \tr \log\left( I_{m_d} + \frac{4d}{m_d} U_d^{0T} \int_0^{t/d} e^{-2 \phi_d(s)} \exp \left( \frac{4sd}{m_d^*} U_d^* U_d^{*T} \right) \d s \, U_d^0 \right).
    \end{aligned}
    \end{equation*}
Therefore, introducing $\gamma = t/d$, $\lambda_d = d/m_d$ and $\lambda_d^* = d/m_d^*$:
\begin{equation*}
    \phi_d(\gamma) + \gamma d = \frac{1}{4} \tr \log \left( I_{m_d} + 4 \lambda_d \int_0^\gamma e^{-2 \phi_d(s)} \left( I_{m_d} + \big( e^{4 \lambda_d^* s} - 1 \big) U_d^{0T} U_d^* U_d^{*T} U_d^0 \right) \d s \right),
\end{equation*}
where we used that $\phi_d(\gamma) = \psi_d(\gamma d) - \gamma d$ (in our setup $\tr(Z_d^*) = 1$). This is precisely equation \eqref{eq:C3selfphi}. We set $Y_d = U_d^{0T} U_d^* U_d^{*T} U_d^0$ and:
\begin{equation*}
    F_d(\gamma) = 1 + 4 \lambda_d \int_0^\gamma e^{-2 \phi_d(s)} \d s \hspace{1 cm} G_d(\gamma) = 1 + 4 \lambda_d \int_0^\gamma e^{-2 \phi_d(s)} e^{4 \lambda_d^* s} \d s,
\end{equation*}
so that:
\begin{equation*}
    \phi_d(\gamma) + \gamma d = \frac{1}{4} \tr \log \Big( F_d(\gamma)(1-Y_d) + G_d(\gamma)Y_d \Big).
\end{equation*}
Deriving with respect to $\gamma$:
\begin{equation} \label{eq:C3dotphi}
    \dot{\phi}_d(\gamma) + d = \frac{1}{4} \tr \big[ U_d(\gamma, Y_d) \big],
\end{equation}
with:
\begin{equation*}
    U_d(\gamma, x) = \frac{\dot{F}_d(\gamma)(1-x) + \dot{G}_d(\gamma)x}{F_d(\gamma)(1-x) + G_d(\gamma)x}.
\end{equation*}
Using that $\dot{G}_d(\gamma) = e^{4 \lambda_d^* \gamma} \dot{F}_d(\gamma)$ and integrating by parts:
\begin{equation*}
    G_d(\gamma) = e^{4 \lambda_d^* \gamma} F_d(\gamma) - 4 \lambda_d^* \int_0^\gamma e^{4 \lambda_d^* s} F_d(s) \d s \leq e^{4 \lambda_d^* \gamma} F_d(\gamma).
\end{equation*}
Thus:
\begin{equation*}
    \dot{F}_d(\gamma) G_d(\gamma) - F_d(\gamma) \dot{G}_d(\gamma) = 4 \lambda_d e^{-2 \phi_d(\gamma)} \big( G_d(\gamma) - F_d(\gamma) e^{4 \lambda_d^* \gamma} \big) \leq 0.
\end{equation*}
As a consequence, for every $\gamma \geq 0$, the map $x \mapsto U_d(\gamma, x)$ is non-decreasing. Using equation \eqref{eq:C3dotphi} and the fact that the eigenvalues of $Y_d$ are contained in $[0,1]$:
\begin{equation*}
    \frac{m_d}{4}\frac{\dot{F}_d(\gamma)}{F_d(\gamma)} \leq \dot{\phi}_d(\gamma) + d \leq  \frac{m_d}{4} \frac{\dot{G}_d(\gamma)}{G_d(\gamma)}.
\end{equation*}
Using the bound on $\dot{\phi}_d$ obtained in \cref{lemma:phi1}, we obtain:
\begin{equation*}
    \frac{\dot{F}_d(\gamma)}{F_d(\gamma)} \leq 4 \kappa \frac{\sqrt{d}}{m_d} + 4 \lambda_d  \hspace{1 cm } \frac{\dot{G}_d(\gamma)}{G_d(\gamma)} \geq -4 \kappa \frac{\sqrt{d}}{m_d} + 4 \lambda_d.
\end{equation*}
This can be integrated to obtain a bound on $\phi_d$:
\begin{equation*}
\left(1 -\frac{\kappa}{\sqrt{d}} \right) e^{- 4 \lambda_d^* \gamma} \exp \left(  4\lambda_d \left(1 - \frac{\kappa}{\sqrt{d}} \right) \gamma \right) \leq  e^{-2 \phi_d(\gamma)} \leq \left(1 + \frac{\kappa}{\sqrt{d}} \right) \exp \left( 4\lambda_d \left(1 + \frac{\kappa}{\sqrt{d}} \right) \gamma \right).
    \end{equation*}
Since $\lambda_d$ and $\lambda_d^*$ converge as $d \to \infty$, this proves the first point. 

\subsubsection{Equicontinuity}

We now show the equicontinuity of the family $(\phi_d)_{d \in \mathbb{N}}$ on $[0,T]$. Back to equation \eqref{eq:C3dotphi}, using the definition of $F_d(\gamma)$ and $G_d(\gamma)$ :
\begin{equation*}
    U_d(\gamma, x) = 4 \lambda_d e^{-2 \phi_d(\gamma)} \underbrace{ \frac{1-x + xe^{4 \lambda_d^* \gamma}}{1 + 4 \lambda_d \int_0^\gamma e^{-2 \phi_d(s)} (1-x + xe^{4 \lambda_d^* s}) \d s }}_{ \textstyle \equiv v_d(\gamma, x)}.
\end{equation*}
Thus, using equation \eqref{eq:C3dotphi} and rearranging:
\begin{equation*}
    \phi_d(\gamma) = - \frac{1}{2} \log \biggl( 1 + \underbrace{\frac{\dot{\phi}_d(\gamma)}{d}}_{\textstyle \equiv \epsilon_d(\gamma)} \biggr) + \frac{1}{2} \log\biggl( \underbrace{\frac{1}{m_d} \tr \big[ v_d(\gamma, Y_d) \big] }_{\textstyle \equiv T_d(\gamma)} \biggr).
\end{equation*}
Now, it can be shown that:
\begin{equation*}
    \left| \dot{T}_d(\gamma) \right| \leq \sup_{x \in [0,1]} \left| \frac{\partial v_d}{\partial \gamma} (\gamma, x) \right| \leq  4 \lambda_d^* e^{4 \lambda_d^* T} + 4 \lambda_d e^{-2 \phi_d(\gamma)} e^{8 \lambda_d^* T}. 
\end{equation*}
Since $\phi_d$ is uniformly bounded in $d$ and on $[0,T]$ from the previous step, then this quantity is uniformly bounded in $d$ and on $[0,T]$, say by $L > 0$. We let $\gamma, \gamma' \in [0,T]$ and look at:
\begin{equation} \label{eq:C3boundequi}
    \left| \phi_d(\gamma) - \phi_d(\gamma') \right| \leq \frac{1}{2} \left| \log \left( 1 + \frac{\epsilon_d(\gamma') - \epsilon_d(\gamma)}{1 + \epsilon_d(\gamma)} \right) \right| + \frac{1}{2} \left| \log \left( 1 + \frac{T_d(\gamma') - T_d(\gamma)}{T_d(\gamma)} \right) \right|.
\end{equation}
Moreover:
\begin{equation*}
    T_d(\gamma) \geq \inf_{x \in [0,1]} v_d(\gamma, x) \geq \left( 1 + \frac{\lambda_d}{\lambda_d^*} e^{-2c} e^{4 \lambda_d T} \right)^{-1},
\end{equation*}
where $c \in \R$ is such that $\phi_d(\gamma) \geq c$ for all $\gamma \in [0,T]$ and $d \in \mathbb{N}$ (such a constant exists from the previous step). This proves that $T_d(\gamma)$ is lower bounded by some constant $\rho > 0$, independent from $\gamma \in [0,T]$ and $d$. As a consequence:
\begin{equation*}
    \left| \frac{T_d(\gamma) - T_d(\gamma')}{T_d(\gamma)} \right| \leq \frac{L}{\rho} \big| \gamma - \gamma' \big|,
\end{equation*}
for all $\gamma, \gamma' \in [0,T]$ and $d \in \mathbb{N}$. Take $\eta \leq \rho / (2L)$ and suppose that $|\gamma- \gamma'| \leq \eta$. Then using the inequality $\big| \log(1+x) \big| \leq 2 |x|$ for $x \in [-1/2, 1/2]$, the second term of equation \eqref{eq:C3boundequi}:
\begin{equation*}
\left| \log \left( 1 + \frac{T_d(\gamma') - T_d(\gamma)}{T_d(\gamma)} \right) \right| \leq \frac{2L}{\rho}\eta .
\end{equation*}
For the first term, we use the bound obtained in \cref{lemma:phi1}, which implies that:
\begin{equation*}
    \sup_{\gamma \in [0,T]} \big| \epsilon_d(\gamma) \big| \leq \frac{\kappa}{\sqrt{d}}.
\end{equation*}
Thus, for $d \geq \lfloor \kappa^2 \max(25, (1+2/\eta)^2 ) \rfloor \equiv d_0$, we obtain:
\begin{equation*}
     \left| \log \left( 1 + \frac{\epsilon_d(\gamma') - \epsilon_d(\gamma)}{1 + \epsilon_d(\gamma)} \right) \right| \leq 2 \eta,
\end{equation*}
and finally:
\begin{equation*}
    \sup_{d \geq d_0} \sup_{\substack{\gamma, \gamma' \in [0,T] \\ |\gamma-  \gamma'| \leq \eta}} \big| \phi_d(\gamma) - \phi_d(\gamma') \big| \leq \eta \left( 1 + \frac{L}{\rho} \right).
\end{equation*}
Obviously this quantity goes to zero as $\eta \to 0$. For $d = 0, \dots, d_0-1$, each $\phi_d$ is continuous on $[0,T]$, thus uniformly continuous, and:
\begin{equation*}
    \sup_{d \in \mathbb{N}} \sup_{\substack{\gamma, \gamma' \in [0,T] \\ |\gamma-  \gamma'| \leq \eta}} \big| \phi_d(\gamma) - \phi_d(\gamma') \big| \leq \max \left[ \max_{d \in \llbracket 0, d_0-1 \rrbracket} \sup_{\substack{\gamma, \gamma' \in [0,T] \\ |\gamma-  \gamma'| \leq \eta}} \big| \phi_d(\gamma) - \phi_d(\gamma') \big|, \ \eta \left( 1 + \frac{L}{\rho} \right) \right] \xrightarrow[\eta \to 0]{} 0,
\end{equation*}
which is the claim of \cref{lemma:phi2}.

\subsection{Proof of Lemma \ref{lemma:phi3}} \label{App:proofphi3}

We set $\sigma_d = (d/m_d, d/m_d^*)$ which converges towards $\sigma = (\alpha^{-1}, \alpha^{*-1})$ as $d \to \infty$. Observe that:
\begin{equation*}
    \H_d(\xi)(\gamma) = \frac{1}{m_d} \tr \big( S(\sigma_d, \xi, \gamma, Y_d) \big) \hspace{1 cm} \H(\xi)(\gamma) = \int S(\sigma, \xi, \gamma, x) \d \mu(x),
\end{equation*}
with:
\begin{equation*}
    S(\sigma, \xi, \gamma, x) = \frac{1}{4 \sigma_1} \log \left( 1 + 4 \sigma_1 \int_0^\gamma e^{-2 \xi(s)} \d s + 4x \sigma_1 \int_0^\gamma e^{-2 \xi(s)} \big( e^{4 \sigma_2 s} -1 \big) \d s \right),
\end{equation*}
and $\sigma = (\sigma_1, \sigma_2)$. Now, for $\xi \in B_c$ and $\gamma \in [0,T]$:
\begin{equation} \label{eq:C3bound1}
\begin{aligned}
    \big| \H_d(\xi)(\gamma) - \H(\xi)(\gamma) \big| &\leq \frac{1}{m_d} \tr \, \Big| S(\sigma_d, \xi, \gamma, Y_d) - S(\sigma, \xi, \gamma, Y_d) \Big| \\
    &\hspace{1.5 cm}+ \left| \frac{1}{m_d} \tr S(\sigma, \xi, \gamma, Y_d) - \int S(\sigma, \xi, \gamma, x) \d x \right|.
\end{aligned}
\end{equation}
We start by the first term. We have:
\begin{equation*}
    \frac{1}{m_d} \tr \, \Big| S(\sigma_d, \xi, \gamma, Y_d) - S(\sigma, \xi, \gamma, Y_d) \Big| \leq \sup_{x \in [0,1]} \big| S(\sigma_d, \xi, \gamma, x) - S(\sigma, \xi, \gamma, x) \big|,
\end{equation*}
since the eigenvalues of $Y_d$ are contained in $[0,1]$. We take $\sigma_1, \tilde{\sigma}_1$ and suppose that they are lower bounded by some $\sigma_0 > 0$. Now, for $a, \tilde{a} > 0$:
\begin{equation*}
\begin{aligned}
    \left| \frac{1}{4 \sigma_1} \log \big( 1 + 4\sigma_1 a \big) - \frac{1}{4 \tilde{\sigma}_1} \log \big( 1 + 4\tilde{\sigma}_1 \tilde{a} \big) \right| &\leq \left| \frac{1}{4 \sigma_1} \log \big( 1 + 4\sigma_1 a \big) - \frac{1}{4 \tilde{\sigma}_1} \log \big( 1 + 4\tilde{\sigma}_1 a \big) \right| \\
    &\hspace{1 cm} + \left| \frac{1}{4 \tilde{\sigma}_1} \log \big( 1 + 4\tilde{\sigma}_1 a \big) - \frac{1}{4 \tilde{\sigma}_1} \log \big( 1 + 4\tilde{\sigma}_1 \tilde{a} \big) \right| \\
    &\leq \frac{2a}{\sigma_0} | \sigma_1 - \tilde{\sigma}_1 | + |a - \tilde{a}|.
\end{aligned}
\end{equation*}
We used that the map $u \mapsto \dfrac{1}{4u} \log(1+4ua)$ is Lipschitz with coefficient $\dfrac{2a}{u_0}$ for $u \geq u_0$ and the fact that $u \mapsto \log(1+u)$ is $1-$Lipschitz. Applying this result with $a = \int_0^\gamma e^{-2 \xi(s)} \d s + x \int_0^\gamma e^{-2 \xi(s)} \big( e^{4 \sigma_2 s} - 1 \big) \d s$ (and the same for $\tilde{a}$ with $\tilde{\sigma}_2$):
\begin{equation*}
\begin{aligned}
    \big| S(\sigma, \xi, \gamma, x) - S(\tilde{\sigma}, \xi, \gamma, x) \big| &\leq \frac{2a}{\sigma_0} | \sigma_1 - \tilde{\sigma}_1| + x \int_0^\gamma e^{-2 \xi(s)} \big| e^{4s \sigma_2} - e^{4s \tilde{\sigma}_2} \big| \d s \\
    &\leq \frac{2a}{\sigma_0} | \sigma_1 - \tilde{\sigma}_1| + 4x |\sigma_2 - \tilde{\sigma}_2|  \int_0^\gamma s e^{-2 \xi(s)} e^{4s \sigma_0^*} \d s,
\end{aligned}
\end{equation*}
where we assumed that $\sigma_2, \tilde{\sigma}_2 \leq \sigma_0^*$. Now, using that $\gamma \in [0,T]$, $x \in [0,1]$ and $\xi \in B_c$:
\begin{equation*}
    \big| S(\sigma, \xi, \gamma, x) - S(\tilde{\sigma}, \xi, \gamma, x) \big| \leq \frac{2}{\sigma_0}\left( T e^{-2c} + \frac{e^{-2c}}{4 \sigma_0^*} e^{4 \sigma_0^* T} \right) |\sigma_1 - \tilde{\sigma}_1 | + \frac{Te^{-2c}}{\sigma_0^*} e^{4T \sigma_0^*} |\sigma_2 - \tilde{\sigma}_2 |.
\end{equation*}
Now using this result with $\sigma_d = (d/m_d, d/m_d^*) \xrightarrow[d \to \infty]{} \sigma = (\alpha^{-1}, \alpha^{*-1})$, it is clear that $d/m_d$ is bounded away from zero and $d/m_d^*$ is lower bounded (at least for $d$ large enough), which confirms the existence of $\sigma_0, \sigma_0^*$. Finally:
\begin{equation*}
    \sup_{\xi \in B_c} \sup_{\gamma \in [0,T]} \sup_{x \in [0,1]} \big| S(\sigma_d, \xi, \gamma, x) - S(\sigma, \xi, \gamma, x) \big| \xrightarrow[d \to \infty]{} 0,
\end{equation*}
which proves that the first term of equation \eqref{eq:C3bound1} goes to zero uniformly in $\xi$ and $\gamma$. We look at the second term. Using the bounds, for $\xi \in B_c$ and $\gamma \in [0,T]$:
\begin{equation*}
\begin{aligned}
    1 +4 \sigma_1 \int_0^\gamma e^{-2 \xi(s)} \d s &\leq 1 + 4 \sigma_1 e^{-2c} T \equiv A(\sigma),\\
    4 \sigma_1 \int_0^\gamma e^{-2 \xi(s)} \big( e^{4 \sigma_2 s} - 1 \big) \d s &\leq \frac{\sigma_1}{\sigma_2} e^{-2c} e^{4 \sigma_2 T} \equiv B(\sigma),
\end{aligned}
\end{equation*}
then:
\begin{equation*}
    \sup_{\xi \in B_c} \sup_{\gamma \in [0,T]} \left| \frac{1}{m_d} \tr S(\sigma, \xi, \gamma, Y_d) - \int S(\sigma, \xi, \gamma, x) \d x \right| \leq \frac{1}{4 \sigma_1} \sup_{(a,b) \in K(\sigma)} \left| \frac{1}{m_d} \tr \Big[ l_{a,b}(Y_d) \Big] - \int l_{a,b}(x) \d \mu(x) \right|,
\end{equation*}
with $K(\sigma) = [1, A(\sigma)] \times [0, B(\sigma)]$ and $l_{a,b}(x) = \log(a+bx)$. Since the family $(l_{a,b})_{a,b \in K(\sigma)}$ is compact, we can apply the result of \cref{lemma:largedev} to obtain that the second term also goes to zero with probability one. Therefore, both terms of equation \eqref{eq:C3bound1} converge uniformly to zero, which concludes the proof.

\subsection{Proof of Lemma \ref{lemma:phi5}} \label{App:proofphi5}

From equation \eqref{eq:selfphi}, we have:
\begin{equation} \label{eq:D4eq1}
    4 \gamma = \alpha \log F_\phi(\gamma) + (\alpha + \alpha^*-1)^+ \log J_\phi(\gamma) + \Theta \big(J_\phi(\gamma) - 1 \big),
\end{equation}
with:
\begin{equation*}
    F_\phi(\gamma) = 1 + \frac{4}{\alpha} \int_0^t e^{-2 \phi(s)} \d s \hspace{1 cm} J_\phi(\gamma) = F_\phi(\gamma)^{-1} \left(  1 + \frac{4}{\alpha} \int_0^t e^{-2 \phi(s)} e^{4s / \alpha^*} \d s \right). 
\end{equation*}
From the expression of $\Theta$ in \cref{PropPhi}:
\begin{equation*}
   \Theta(u) =\frac{1}{2\pi} \int_{r_-}^{r_+} \frac{\sqrt{(r_+-x)(x-r_-)}}{x(1-x)} \log(1 + ux) \d x,
\end{equation*}
it is easy to see that $\Theta$ is twice  continuously differentiable on $]-1, \infty[$, and that $\Theta'(u) \geq 0$ for all $u > -1$. We now suppose that $\phi$ is a continuous solution of equation \eqref{eq:D4eq1}. Then, since $J_\phi(\gamma) \geq 1$, we can differentiate with respect to $\gamma$:
\begin{equation*}
    4 = \alpha \frac{\dot{F}_\phi(\gamma)}{F_\phi(\gamma)} + (\alpha + \alpha^*-1)^+ \frac{\dot{J}_\phi(\gamma)}{J_\phi(\gamma)} + \dot{J}_\phi(\gamma) \Theta' \big(J_\phi(\gamma)-1\big).
\end{equation*}
From the relationships:
\begin{equation*}
    \dot{F}_\phi(\gamma) = \frac{4}{\alpha} e^{-2 \phi(\gamma)} \hspace{1 cm} \dot{J}_\phi(\gamma) = \frac{4}{\alpha} e^{-2 \phi(\gamma)} F_\phi(\gamma)^{-1} \big( e^{4 \gamma / \alpha^*} - J_\phi(\gamma) \big),
\end{equation*}
we finally obtain that $(F_\phi(\gamma), J_\phi(\gamma))$ is solution of the differential system:
\begin{equation} \label{eq:D4eq2}
\begin{pmatrix}\dot{F}(\gamma) \\  \dot{J}(\gamma) \end{pmatrix} = \frac{4}{\alpha + \Gamma(J(\gamma)) \big( e^{4 \gamma / \alpha^*} - J(\gamma) \big)} \begin{pmatrix}F(\gamma) \\ e^{4 \gamma / \alpha^*} - J(\gamma) \end{pmatrix},
\end{equation}
with initial condition $F(0) = J(0) = 1$ and:
\begin{equation*}
    \Gamma(J) = \frac{( \alpha + \alpha^* - 1)^+}{J} + \Theta'(J-1).
\end{equation*}
Thus, if we manage to show that the solution of equation \eqref{eq:D4eq2} is unique, this will imply the uniqueness of the continuous solution $\phi$. Reciprocally, suppose that we have a couple $(F(\gamma), J(\gamma))$ solution of equation \eqref{eq:D4eq2} with the initial condition, then it is easily shown that it is also solution of equation \eqref{eq:D4eq1}. Provided that $\dot{F}(\gamma) > 0$ for all $\gamma \geq 0$, we will obtain the existence of a solution:
\begin{equation*}
    \phi(\gamma) = - \frac{1}{2} \log \left( \frac{\alpha \dot{F}(\gamma)}{4} \right).
\end{equation*}
The goal is now to study equation \eqref{eq:D4eq2}. We now set:
\begin{equation*}
    \Omega(F, J, \gamma) =  \frac{4}{\alpha + \Gamma(J) \big( e^{4 \gamma / \alpha^*} - J \big)} \begin{pmatrix}F  \\ e^{4 \gamma / \alpha^*} - J  \end{pmatrix}.
\end{equation*}
$\Omega$ is well defined on the domain $D = \big\{ (F, J, \gamma) \in \R \times ]0, \infty[ \times \R^+ \ \big| \  \alpha + \Gamma(J) \big( e^{4 \gamma / \alpha^*} - J \big) \neq 0 \big\}$, which is an open subset of $\R^2 \times \R^+$. Now, since $\Theta$ is twice continuously differentiable on $]-1, \infty[$, $\Omega$ is continuously differentiable on its domain. Let $u \in D$ and $\mathcal{V} \subset D$ a compact neighbourhood of $u$. Then by continuity:
\begin{equation*}
    \sup_{(F,J, \gamma) \in \mathcal{V}} \big\| \nabla \Omega(F,J,\gamma) \big\| < \infty.
\end{equation*}
Which proves that $\Omega$ is locally Lipschitz continuous on $D$. Thus, from Picard–Lindelöf theorem, the solution $F(\gamma), J(\gamma)$ of the differential system \eqref{eq:D4eq2} with initial condition $F(0) = J(0) = 1$ are well defined and unique on a subset of the form $[0, \eta]$ for $\eta > 0$. We now let $F_M, J_M$ be maximal solutions for this Cauchy-Lipschitz problem. Following the previous reasoning, they are at least defined and unique on $[0, \eta]$. We suppose by contradiction that they are defined up to $[0,T[$ with $T > 0$. From equation \eqref{eq:D4eq2}, we get that $\alpha + \Gamma(J_M(\gamma) ( e^{4 \gamma / \alpha^*} - J_M(\gamma) ) > 0$ for $\gamma \in [0,T[$, since it is positive for $\gamma = 0$ and cannot cancel. Thus, $\dot{F}_M(\gamma)$ remains of the same sign as $F_M(\gamma)$. Since $F_M(0) = 1$, we get that $\dot{F}_M(\gamma) > 0$ and $F_M$ is non-decreasing. Now, from the equation:
\begin{equation*}
    \dot{J}_M(\gamma) = \big( e^{4 \gamma / \alpha^*} - J_M(\gamma) \big) \frac{\dot{F}_M(\gamma)}{F_M(\gamma)},
\end{equation*}
we get:
\begin{equation*}
    J_M(\gamma) = e^{4 \gamma / \alpha^*} - \frac{4}{\alpha^* F_M(\gamma)} \int_0^\gamma F_M(s) e^{4s / \alpha^*} \d s < e^{4 \gamma / \alpha^*},
\end{equation*}
for $\gamma \in [0, T[$. Thus, $\dot{J}_M(\gamma) > 0$ and $J_M$ is non decreasing. Therefore, $\lim_{\gamma \to T^-} J_M(\gamma)$ exists and is bounded by $e^{4 T / \alpha^*}$. Since $J_M(\gamma)$ remains positive on $[0,T[$, this implies that:
\begin{equation*}
    \lim_{\gamma \to T^-} \left[\alpha + \Gamma(J_M(\gamma)) \big(e^{4 \gamma / \alpha^*} - J_M(\gamma))  \right] \geq \alpha.
\end{equation*}
Thus $(F_M(\gamma), J_M(\gamma), \gamma)$ remains inside $D$ as $\gamma \to T^-$. Now, by maximality assumption, $J_M(\gamma)$ or $F_M(\gamma)$ should escape any compact as $\gamma \to T^-$. From the previous observation, this cannot be the case for $J_M(\gamma)$. Thus, we necessarily have $\lim_{\gamma \to T^{-}} F_M(\gamma) = +\infty$ (again this limit exists since $F_M$ is non-decreasing on $[0,T[$). From equation \eqref{eq:D4eq2}:
\begin{equation*}
    F_M(\gamma) = \exp \left( 4 \int_0^\gamma \frac{\d s}{\alpha + \Gamma(J_M(s)) \big( e^{4 s / \alpha^*} - J_M(s) \big)} \right) \leq e^{4 \gamma / \alpha},
\end{equation*}
and a contradiction. Thus necessarily $T = \infty$ and the system \eqref{eq:D4eq2} admits a unique solution defined on all $\R^+$. From the equivalence between equations \eqref{eq:D4eq1} and \eqref{eq:D4eq2}, this concludes the proof.

\subsection{Proof of Lemma \ref{lemma:overlap2}} \label{App:proofoverlap2}

As proven in \cref{PropPhi}, $\phi_d$ uniformly converges on $[0,T]$ towards a function $\phi$. Thus we define:
\begin{equation} \label{eq:C4defFG}
    F(\gamma) = 1 + \frac{4}{\alpha} \int_0^\gamma e^{-2 \phi(s)} \d s \hspace{1 cm } G(\gamma) = 1 + \frac{4}{\alpha} \int_0^\gamma e^{-2 \phi(s)} e^{4s / \alpha^*} \d s \hspace{1 cm} J(\gamma) = \frac{G(\gamma)}{F(\gamma)}.
\end{equation}
It is easily seen that $F_d \xrightarrow[d \to \infty]{} F$ and $G_d \xrightarrow[d \to \infty]{} G$ uniformly on $[0,T]$. Therefore, since $F_d(\gamma), F(\gamma) \geq 1$:
\begin{equation*}
    \big| J_d(\gamma) - J(\gamma) \big| \leq \big|F(\gamma)\big| \big|G_d(\gamma) - G(\gamma)\big| + \big|G(\gamma) \big| \big|F_d(\gamma) - F(\gamma) \big|.
\end{equation*}
This proves that $J_d \xrightarrow[d \to \infty]{} J$ uniformly on $[0,T]$. Now, for the numerator, we set $U(x, J) = \dfrac{x}{1 + (J-1)x}$, so that:
\begin{equation*}
\begin{aligned}
\left| \frac{1}{m_d} \tr \Big( U(Y_d, J_d(\gamma)) \Big) - \int U(x, J(\gamma)) \d \mu(x) \right| &\leq \frac{1}{m_d} \tr \Big| U(Y_d, J_d(\gamma)) - U(Y_d, J(\gamma)) \Big| \\
&\hspace{1 cm}+ \left| \frac{1}{m_d} \tr \Big( U(Y_d, J(\gamma)) \Big) - \int U(x, J(\gamma)) \d \mu(x) \right|.
\end{aligned}
\end{equation*}
Starting with the first term:
\begin{equation*}
     \frac{1}{m_d} \tr \Big| U(Y_d, J_d(\gamma)) - U(Y_d, J(\gamma)) \Big| \leq \sup_{x \in [0,1]} \Big| U(x, J_d(\gamma)) - U(x, J(\gamma)) \Big| \leq \big|J_d(\gamma) - J(\gamma) \big|.
\end{equation*}
Since it is easily shown that $J \mapsto U(x, J)$ is Lipschitz on $[1, \infty[$ with constant $x^2$ (note that we always have $J_d(\gamma), J(\gamma) \in [1, \infty[$). Thus the first term goes uniformly to zero on $[0,T]$. For the second, $J$ is continuous on $[0,T]$, thus it is bounded by some constant $A \geq 1$. Therefore:
\begin{equation*}
    \sup_{\gamma \in [0,T]} \left| \frac{1}{m_d} \tr \Big( U(Y_d, J(\gamma)) \Big) - \int U(x, J(\gamma)) \d \mu(x) \right| \leq \sup_{a \in [1, A]} \left| \frac{1}{m_d} \tr \Big( U(Y_d, a) \Big) - \int U(x, a) \d \mu(x) \right|,
\end{equation*}
which goes to zero following \cref{lemma:largedev}. Therefore, the first convergence of \cref{lemma:overlap2} is proven. The same can be done for the second. The same thing can be done for the denominator. Introducing:
\begin{equation*}
    V(x, J, e) = \left( \frac{1 + (e-1)x}{1 + (J-1)x} \right)^2,
\end{equation*}
as well as $e_d(\gamma) = e^{4 \gamma d / m_d^*}$ and $e(\gamma) = e^{4 \gamma / \alpha^*}$, we have:
\begin{equation} \label{eq:C4boundV}
\begin{aligned}
\Bigg| \frac{1}{m_d} \tr \Big( V(Y_d, J_d(\gamma), e_d(\gamma)) \Big) - &\int V(x, J(\gamma), e(\gamma)) \d \mu(x) \Bigg| \leq \frac{1}{m_d} \tr \Big| V(Y_d, J_d(\gamma), e_d(\gamma)) - V(Y_d, J(\gamma), e(\gamma)) \Big| \\
&\hspace{2cm}+ \left| \frac{1}{m_d} \tr \Big( V(Y_d, J(\gamma), e(\gamma)) \Big) - \int V(x, J(\gamma), e(\gamma)) \d \mu(x) \right|.
\end{aligned}
\end{equation}
Now:
\begin{equation*}
\frac{1}{m_d} \tr \Big| V(Y_d, J_d(\gamma), e_d(\gamma)) - V(Y_d, J(\gamma), e(\gamma)) \Big| \leq \sup_{x \in [0,1]} \Big| V(x, J_d(\gamma), e_d(\gamma)) - V(x, J(\gamma), e(\gamma)) \Big|.
\end{equation*}
Moreover, for $e, \tilde{e} \in [1, E]$ and $J, \tilde{J} \geq 1$:
\begin{equation*}
\begin{aligned}
    \Big| V(x, J, e) - V(x, \tilde{J}, \tilde{e}) \Big| &\leq \Big| V(x, J, e) - V(x, \tilde{J}, e) \Big| + \Big| V(x, \tilde{J}, e) - V(x, \tilde{J}, \tilde{e}) \Big| \\
    &\leq 2E^2 |J-\tilde{J}| + 2E | e-\tilde{e}|.
\end{aligned}
\end{equation*}
Applying this to our case, we have $e_d(\gamma) = e^{4 \gamma d / m_d^*} \leq e^{4 T d / m_d^*}$ stays bounded since $d / m_d^* \xrightarrow[d \to \infty]{} \alpha^{*-1}$. Thus, again, the first term of equation \eqref{eq:C4boundV} goes to zero uniformly on $[0,T]$. For the second, it is clear that since $J$ and $e$ are continuous on $[0,T]$, they both stay in a compact of the form $K = [1, A] \times [1, E]$. Therefore, by \cref{lemma:largedev}:
\begin{equation*}
    \sup_{\gamma \in [0,T]} \left| \frac{1}{m_d} \tr \Big( V(Y_d, J(\gamma), e(\gamma)) \Big) - \int V(x, J(\gamma), e(\gamma)) \d \mu(x) \right| \leq \sup_{(J,e) \in K} \left| \frac{1}{m_d} \tr \Big( V(Y_d, J, e) \Big) - \int V(x, J, e) \d \mu(x) \right|,
\end{equation*}
which goes to zero. As a conclusion, the second convergence of \cref{lemma:overlap2} is proven. Finally, for the overlap, write from equation \eqref{eq:C4overlap}:
\begin{equation*}
    \chi_d(\gamma) = \sqrt{\frac{m_d}{m_d^*}} e^{4 \gamma d /m_d^*} \frac{A_d(\gamma)}{\sqrt{B_d(\gamma)}},
\end{equation*}
where $A_d(\gamma) \xrightarrow[d \to \infty]{} A(\gamma)$ and $B_d(\gamma) \xrightarrow[d \to \infty]{} B(\gamma)$ are the quantity displayed in the lemma, which converge uniformly from the previous reasoning. We write:
\begin{equation*}
\begin{aligned}
    \big| \chi_d(\gamma) - \chi(\gamma) \big| &\leq \frac{A_d(\gamma)}{\sqrt{B_d(\gamma)}} \left| \sqrt{\frac{m_d}{m_d^*}} e^{4 \gamma d / m_d^*} - \sqrt{\frac{\alpha}{\alpha^*}} e^{4 \gamma / \alpha^*} \right| + \sqrt{\frac{\alpha}{\alpha^*}} e^{4 \gamma / \alpha^*} \frac{1}{\sqrt{B_d(\gamma)}} \big| A_d(\gamma) - A(\gamma) \big| \\
    &\hspace{3 cm}+ \sqrt{\frac{\alpha}{\alpha^*}} e^{4 \gamma / \alpha^*} A(\gamma) \left| \frac{1}{\sqrt{B_d(\gamma)}} - \frac{1}{\sqrt{B(\gamma)}} \right|. 
\end{aligned}
\end{equation*}
Since $B_d(\gamma)$ uniformly converges on $[0,T]$ towards $B(\gamma)$ which is strictly positive, thus it is uniformly bounded away from zero for $d$ large enough. Thus, there are constant $k_1, k_2, k_3 > 0$ such that for some $d \geq d_0$:
\begin{equation*}
\begin{aligned}
    \sup_{\gamma \in [0,T]} \big| \chi_d(\gamma) - \chi(\gamma) \big| &\leq k_1 \sup_{\gamma \in [0,T]} \left| \sqrt{\frac{m_d}{m_d^*}} e^{4 \gamma d / m_d^*} - \sqrt{\frac{\alpha}{\alpha^*}} e^{4 \gamma / \alpha^*} \right| + k_2 \sup_{\gamma \in [0,T]} \big| A_d(\gamma) - A(\gamma) \big| \\
    & \hspace{3 cm} + k_3 \sup_{\gamma \in [0,T]}  \left| B_d(\gamma) - B(\gamma) \right|, \\
\end{aligned}
\end{equation*}
since the map $x \mapsto x^{-1/2}$ is Lipschitz on the intervals of the form $[b, \infty[$. As shown before $A_d(\gamma) \longrightarrow A(\gamma)$ and $B_d(\gamma) \longrightarrow B(\gamma)$ uniformly as $d \to \infty$, so that the two last terms go to zero. Since $m_d^* / d$ and $m_d / d$ respectively converge to $\alpha^*, \alpha$, it is easily checked that the first term also goes uniformly to zero. Finally, $\chi_d \xrightarrow[d \to \infty]{} \chi$ uniformly on $[0,T]$, which concludes the proof.

\subsection{Proof of Lemma \ref{lemma:overlap3}} \label{App:proofoverlap3}

From the implicit equation \eqref{eq:selfphi} solved by $\phi$, we obtain:
\begin{equation} \label{eq:C4implicit}
    4 \gamma = \min \big( \alpha, 1-\alpha^* \big) \log F(\gamma) + \big(\alpha+\alpha^*-1\big)^+ \log G(\gamma) + \Theta \big( J(\gamma)-1 \big),
\end{equation}
where $F,G, J$ are defined in equation \eqref{eq:C4defFGJ} (and depend on $\phi$), and:
\begin{equation*}
\Theta(u) = \frac{1}{2 \pi} \int_{r_1}^{r_2} \frac{\sqrt{(r_2-x)(x-r_1)}}{x(1-x)} \log(1 + u x) \d x.
\end{equation*}
Using the definition of $\mu$ (and the fact it has unit total mass), it can be shown that there exists $\Omega(\alpha, \alpha^*)$:
\begin{equation} \label{eq:C4Theta}
    \Theta(u) \underset{u \to \infty}{=} \frac{1 - |\alpha - \alpha^*| - |\alpha+\alpha^*-1|}{2} \log u + \Omega(\alpha, \alpha^*) + O \left( \frac{1}{u} \right).
\end{equation}
The first step is to show that $F(\gamma), G(\gamma)$ and $J(\gamma)$ go to infinity as $\gamma \to \infty$. We recall that:
\begin{equation*} 
    F(\gamma) = 1 + \frac{4}{\alpha} \int_0^\gamma e^{-2 \phi(s)} \d s \hspace{1 cm} G(\gamma) = 1 + \frac{4}{\alpha} \int_0^\gamma e^{-2 \phi(s)}  e^{4s / \alpha^*} \d s \hspace{1 cm} J(\gamma) = \frac{G(\gamma)}{F(\gamma)}.
\end{equation*}
We start by $F$ and $G$. As they are both non-decreasing, either they converge or go to infinity. Obviously they cannot both converge as $\gamma \to \infty$, otherwise the right hand side of equation \eqref{eq:C4implicit} would stay finite as $\gamma \to \infty$. Since $G(\gamma) \geq F(\gamma)$, at least $G(\gamma) \xrightarrow[\gamma \to \infty]{} \infty$. Suppose that $F(\gamma) \xrightarrow[\gamma \to \infty]{} F$, then we would have from equations \eqref{eq:C4implicit} and \eqref{eq:C4Theta}:
\begin{equation*}
    \left( \big(\alpha + \alpha^*-1\big)^+ + \frac{1 - |\alpha - \alpha^*| - |\alpha+\alpha^*-1|}{2} \right) \log G(\gamma) \underset{\gamma \to \infty}{=} 4 \gamma + O(1),
\end{equation*}
which writes:
\begin{equation*}
    G(\gamma) = e^{O(1)} \exp \left( \frac{4 \gamma}{\min(\alpha, \alpha^*)} \right).
\end{equation*}
Therefore, integrating by parts, one has:
\begin{equation} \label{eq:C4relationFG}
    F(\gamma) = G(\gamma) e^{-4 \gamma / \alpha^*} + \frac{4}{\alpha^*}\int_0^\gamma G(s) e^{-4 s / \alpha^*} \d s \geq \frac{4}{\alpha^*}\int_0^\gamma G(s) e^{-4 s / \alpha^*} \d s \xrightarrow[\gamma \to \infty]{} \infty.
\end{equation}
This proves that necessarily $F(\gamma)$ also goes to infinity as $\gamma \to \infty$. We now finally show that also $J(\gamma) \longrightarrow \infty$. Differentiating: 
\begin{equation*}
\dot{J}(\gamma) = \frac{1}{G(\gamma)^2} \big( \dot{G}(\gamma) F(\gamma) - \dot{F}(\gamma) G(\gamma) \big) = \frac{\dot{G}(\gamma)}{G(\gamma)^2} \big(F(\gamma) - G(\gamma) e^{-4 \gamma / \alpha^*} \big) \geq 0,
\end{equation*}
thus $J$ is non-decreasing. Suppose that $J(\gamma)$ converges towards a finite value as $\gamma \to \infty$. Thus, equation \eqref{eq:C4implicit} gives:
\begin{equation*}
4 \gamma = \min \big( \alpha, 1-\alpha^* \big) \log F(\gamma) + \big(\alpha+\alpha^*-1\big)^+ \log G(\gamma) + O(1).
\end{equation*}
Splitting the case $\alpha + \alpha^* \leq 1$ and $\alpha + \alpha^* \geq 1$, we obtain in both case the existence of $C \in \R$ (depending on $\alpha, \alpha^*$) such that:
\begin{equation*}
    4 \gamma \underset{\gamma \to \infty}{=} \alpha \log F(\gamma) + C + o(1).
\end{equation*}
From the following equality, which can be obtained by integrating by parts the relationship $\dot{G}(\gamma) = \dot{F}(\gamma) e^{4 \gamma / \alpha^*}$:
\begin{equation*}
    J(\gamma) = e^{4\gamma / \alpha^*} - \frac{4}{\alpha^* F(\gamma)} \int_0^\gamma F(s) e^{4s / \alpha^*} \d s,
\end{equation*}
we have:
\begin{equation*}
    J(\gamma) \underset{\gamma \to \infty}{\sim} \frac{\alpha^*}{\alpha + \alpha^*} e^{4 \gamma / \alpha^*}.
\end{equation*}
Therefore, we necessarily have $J(\gamma) \longrightarrow \infty$. We now prove the asymptotics on $J$. Using equations \eqref{eq:C4implicit} and $\eqref{eq:C4Theta}$, we obtain as $\gamma \to \infty$:
\begin{equation*}
\begin{aligned}
    4 \gamma &\underset{\gamma \to \infty}{=} \min \big(\alpha, 1- \alpha^* \big) \log F(\gamma) + \big( \alpha + \alpha^* - 1 \big)^+ \log G(\gamma) \\
    & \hspace{2 cm}+ \frac{1 - |\alpha - \alpha^*| - |\alpha+\alpha^*-1|}{2} \log J(\gamma) + \Omega(\alpha, \alpha^*) + o(1) \\
    &= ( \alpha - \alpha^*)^+ \log F(\gamma) + \min ( \alpha, \alpha^* ) \log G(\gamma) + \Omega(\alpha, \alpha^*) + o(1).
\end{aligned}
\end{equation*}
Thus, setting:
\begin{equation*}
    \lambda^* = \frac{1}{\alpha^*} \hspace{1 cm} \lambda_m = \frac{1}{\min(\alpha, \alpha^*)} \hspace{1 cm} \nu = \frac{(\alpha - \alpha^*)^+}{\min(\alpha, \alpha^*)} \hspace{1 cm} \zeta = \exp \left( - \frac{\Omega(\alpha, \alpha^*)}{\min(\alpha, \alpha^*)} \right),
\end{equation*}
we obtain that:
\begin{equation} \label{eq:C4FGnu}
    G(\gamma) = \zeta F(\gamma)^{- \nu} e^{4 \lambda_m \gamma} \omega(\gamma),
\end{equation}
where $\omega(\gamma) \xrightarrow[\gamma \to \infty]{} 1$. Using the relation between $F$ and $G$ in equation \eqref{eq:C4relationFG}:
\begin{equation} \label{eq:C4selfF}
    F(\gamma) = \zeta F(\gamma)^{- \nu} e^{4 ( \lambda_m - \lambda^*) \gamma} \omega(\gamma) + 4 \lambda^* \zeta \int_0^\gamma F(s)^{- \nu} e^{4(\lambda_m - \lambda^*)s} \omega(s) \d s.
\end{equation}
We start by the case where $\alpha \leq \alpha^*$, so that $\nu = 0$. Therefore, using the two previous equations as well as \cref{lemma:asymptoticsIntegral}:
\begin{equation*}
    G(\gamma) \underset{\gamma \to \infty}{\sim} \zeta e^{4 \gamma / \alpha} \hspace{1 cm} F(\gamma) \underset{\gamma \to \infty}{\sim} \left\{ \begin{array}{cc}
    \vspace{0.15 cm}
       \dfrac{4 \zeta \gamma}{\alpha^*}  & \text{for } \alpha = \alpha^*  \\
       \zeta \dfrac{\alpha^*}{\alpha^* - \alpha} \exp \left( 4 \gamma \dfrac{\alpha^* - \alpha}{\alpha \alpha^*} \right)  & \text{for } \alpha < \alpha^*.
    \end{array} \right.
\end{equation*}
Therefore:
\begin{equation*}
    J(\gamma) \underset{\gamma \to \infty}{\sim} \left\{ \begin{array}{cc}
    \vspace{0.15cm}
        \dfrac{\alpha^*}{4 \gamma} e^{4 \gamma / \alpha^*} & \text{for } \alpha = \alpha^*  \\
        \left(1 -\dfrac{\alpha}{\alpha^*} \right) e^{4 \gamma / \alpha^*} & \text{for } \alpha < \alpha^*. 
    \end{array} \right.
\end{equation*}
For $\alpha > \alpha^*$, we have that $\lambda^* = \lambda_m$ and $\nu = \alpha / \alpha^* - 1 > 0$, thus $F(\gamma)^{-\nu} \xrightarrow[\gamma \to \infty]{} 0$. Thus, rewriting equation \eqref{eq:C4selfF}:
\begin{equation*}
    F(\gamma) = \underbrace{\zeta F(\gamma)^{-\nu} \omega(\gamma)}_{\textstyle \equiv \epsilon(\gamma)} + 4 \lambda^* \zeta \int_0^\gamma F(s)^{-\nu} \omega(s) \d s.
\end{equation*}
Since $\epsilon(\gamma) \xrightarrow[\gamma \to \infty]{} 0$, we introduce $\epsilon > 0$ and $\gamma_0$ such that $|\epsilon(\gamma)| \leq \epsilon$ as soon as $\gamma \geq \gamma_0$. With $K(\gamma) = \int_0^\gamma F(s)^{-\nu} \omega(s) \d s$, we obtain by integrating:
\begin{equation*}
    - \epsilon + \left[ \left( \frac{4 \zeta}{\alpha^*} K(\gamma_0) + \epsilon \right)^{\alpha / \alpha^*} + \frac{4 \zeta \alpha}{\alpha^{*2}} \int_0^\gamma \omega(s) \d s \right]^{\alpha^* / \alpha}\leq \frac{4 \zeta}{\alpha^*} K(\gamma) \leq \epsilon + \left[ \left( \frac{4 \zeta}{\alpha^*} K(\gamma_0) -\epsilon \right)^{\alpha / \alpha^*} + \frac{4 \zeta \alpha}{\alpha^{*2}} \int_0^\gamma \omega(s) \d s \right]^{\alpha^* / \alpha}.
\end{equation*}
Therefore:
\begin{equation*}
    F(\gamma) \underset{\gamma \to \infty}{\sim} \left( \frac{4 \zeta \alpha}{\alpha^{*2}} \gamma \right)^{\alpha^* / \alpha}.
\end{equation*}
Using equation \eqref{eq:C4FGnu}, we are able to determine the behaviour of $G$, and the one of $J$:
\begin{equation*}
    G(\gamma) \underset{\gamma \to \infty}{\sim} \zeta \left( \frac{4 \zeta \alpha}{\alpha^{*2}} \gamma \right)^{\alpha^* / \alpha - 1} e^{4 \gamma / \alpha^*} \hspace{1 cm} J(\gamma) \underset{\gamma \to \infty}{\sim} \frac{\alpha^{*2}}{4 \alpha \gamma} e^{4 \gamma / \alpha^*}.
\end{equation*}

\subsection{Proof of Lemma \ref{lemma:overlap4}} \label{App:proofoverlap4}

From the definition of $\mu$ in equation \eqref{eq:manova} and those of $A(\gamma)$, $B(\gamma)$ in \cref{lemma:overlap3}, we have:
\begin{equation} \label{eq:D7AB}
\begin{aligned}
    A(\gamma) &= \frac{1}{\alpha}\big( \alpha + \alpha^* - 1)^+ \frac{1}{J(\gamma)} + \int_{r_-}^{r_+} \frac{x}{1 + (J(\gamma)-1)x} w(x) \d x \\
    B(\gamma) &= \left( 1 - \frac{\alpha^*}{\alpha} \right)^+ + \frac{1}{\alpha}\big( \alpha + \alpha^* - 1)^+ \frac{e^{8 \gamma / \alpha^*}}{J(\gamma)^2} + \int_{r_-}^{r_+} \left( \frac{1 + (e^{4 \gamma / \alpha^*} - 1)x}{1 + (J(\gamma)-1)x} \right)^2 w(x) \d x,
\end{aligned}
\end{equation}
with:
\begin{equation*}
\begin{aligned}
w(x) &= \frac{1}{2 \pi \alpha} \frac{\sqrt{(r_+-x)(x-r_-)}}{x(1-x)} \\
r_{\pm} &= \left( \sqrt{\alpha(1-\alpha^*)} \pm \sqrt{\alpha^*(1-\alpha)} \right)^2.
\end{aligned}
\end{equation*}
The challenge is to compute the asymptotics of the integrals. For the first one, we compute:
\begin{equation} \label{eq:D7bound}
    \left| \frac{1}{J(\gamma)} \int_{r_-}^{r_+} w(x) \d x - \int_{r_-}^{r_+} \frac{x}{1 + (J(\gamma)-1)x} w(x) \d x \right| = \frac{1}{J(\gamma)^2} \int_{r_-}^{r_+} \frac{1-x}{x + (1-x)J(\gamma)^{-1}} w(x) \d x.
\end{equation}
We start by supposing that $r_- > 0$. Since $\alpha, \alpha^* > 0$, this corresponds to the case $\alpha \neq \alpha^*$. Then:
\begin{equation*}
    \int_{r_-}^{r_+} \frac{1-x}{x + (1-x)J(\gamma)^{-1}} w(x) \d x \leq \int_{r_-}^{r_+} \frac{1-x}{x} w(x) \d x < \infty,
\end{equation*}
so that in this case:
\begin{equation*}
    \int_{r_-}^{r_+} \frac{x}{1 + (J(\gamma)-1)x} w(x) \d x \underset{\gamma \to \infty}{=} \frac{1}{J(\gamma)} \int_{r_-}^{r_+} w(x) \d x + O \left( \frac{1}{J(\gamma)^2} \right).
\end{equation*}
The integral $\int_{r_-}^{r_+} w(x) \d x$ can be computed using the fact that $\mu$ has unit mass. Thus:
\begin{equation*}
    A(\gamma) \underset{\gamma \to \infty}{=} \frac{\min(\alpha, \alpha^*)}{\alpha J(\gamma)} + O \left( \frac{1}{J(\gamma)^2} \right),
\end{equation*}
whenever $r_- = 0$ this is not possible anymore. In this case, the right integral of equation \eqref{eq:D7bound} writes:
\begin{equation*}
    \frac{1}{2\pi \alpha J(\gamma)} \int_0^{r_+} \sqrt{\frac{r_+-x}{x}} \frac{\d x}{x J(\gamma) + 1-x}.
\end{equation*}
We change variables and let $x = r_+ \dfrac{t^2}{1+t^2}$. Decomposing in partial fractions:
\begin{equation*}
\begin{aligned}
    \frac{1}{2\pi} \int_0^{r_+} \sqrt{\frac{r_+-x}{x}} \frac{\d x}{x J(\gamma) + 1-x} &= \frac{r_+}{\pi} \int_0^\infty \frac{\d t}{(1+t^2) \big( 1 + (1 -r_+ + r_+ J(\gamma))t^2 \big)} \\
    &= \frac{1}{\pi} \left( - \frac{1}{J(\gamma) - 1} \int_0^\infty \frac{\d t}{1+t^2} + \frac{1 + r_+(J(\gamma)-1)}{J(\gamma)-1} \int_0^\infty \frac{\d t}{1 + (1 -r_+ + r_+ J(\gamma))t^2} \right) \\
    &= - \frac{1}{2} \frac{1}{J(\gamma)-1} + \frac{1}{2} \frac{\sqrt{1 +r_+(J(\gamma)-1)}}{J(\gamma)-1} \\
    &= O \left( J(\gamma)^{-1/2} \right),
\end{aligned}
\end{equation*}
so that we get:
\begin{equation*}
    \int_{r_-}^{r_+} \frac{x}{1 + (J(\gamma)-1)x} w(x) \d x \underset{\gamma \to \infty}{=} \frac{1}{J(\gamma)} \int_{r_-}^{r_+} w(x) \d x + O \left( \frac{1}{J(\gamma)^{3/2}} \right),
\end{equation*}
which proves the result for $A(\gamma)$ using equation \eqref{eq:D7AB}. Now for $B(\gamma)$:
\begin{equation} \label{eq:D7boundI1I2}
\begin{aligned}
    \int_{r_-}^{r_+} \left( \frac{1 + (e^{4 \gamma / \alpha^*}-1)x}{1+(J(\gamma)-1)x} \right)^2 w(x) \d x - \frac{e^{8 \gamma / \alpha^*}}{J(\gamma)^2} \int_{r_-}^{r_+} w(x) \d x &= \frac{e^{8 \gamma / \alpha^*}}{J(\gamma)^2} \Bigg[ \left( e^{-4 \gamma / \alpha^*} - J(\gamma)^{-1} \right) I_1(\gamma) \\
    &+\left( e^{-8 \gamma / \alpha^*} - J(\gamma)^{-2} \right) I_2(\gamma)\Bigg],
\end{aligned}
\end{equation}
with:
\begin{equation*}
    I_1(\gamma) = \int_{r_-}^{r_+} \frac{2x(1-x)}{\big(x + (1-x)J(\gamma)^{-1}\big)^2} w(x) \d x \hspace{1 cm} I_2(\gamma) = \int_{r_-}^{r_+} \frac{(1-x)^2}{\big(x + (1-x)J(\gamma)^{-1}\big)^2} w(x) \d x .
\end{equation*}
Clearly the last term is negligible with respect to the others. Again, for $r_- > 0$, i.e., $\alpha \neq \alpha^*$:
\begin{equation*}
    I_1(\gamma) \leq \int_{r_-}^{r_+} \frac{2(1-x)}{x} w(x) \d x < \infty \hspace{1 cm} I_2(\gamma) \leq  \int_{r_-}^{r_+} \frac{(1-x)^2}{x^2} w(x) \d x < \infty,
\end{equation*}
so that:
\begin{equation*}
    \int_{r_-}^{r_+} \left( \frac{1 + (e^{4 \gamma / \alpha^*-1)x}}{1+(J(\gamma)-1)x} \right)^2 w(x) \d x = \frac{e^{8 \gamma / \alpha^*}}{J(\gamma)^2} \left( \int_{r_-}^{r_+} w(x) \d x + O \big( e^{-4 \gamma / \alpha^*} \big) + O \big( J(\gamma)^{-1} \big) \right).
\end{equation*}
From \cref{lemma:overlap3}, we have $e^{-4 \gamma / \alpha^*} = O(J(\gamma)^{-1})$, which proves the first result. If $\alpha = \alpha^*$ thus $r_- = 0$, we apply the same trick as before. We have:
\begin{equation*}
\begin{aligned}
I_1(\gamma) &= \frac{J(\gamma)^2}{\pi \alpha} \int_0^{r_+} \sqrt{x(r_+-x)} \frac{\d x}{\big(1-x + xJ(\gamma)\big)^2} \\
&= \frac{2 r_+^2 J(\gamma)^2}{\pi \alpha} \int_0^\infty \frac{t^2\d t}{(1+t^2)\big( 1 + (1 -r_+ + r_+ J(\gamma))t^2 \big)^2} \\
&\underset{\gamma \to \infty}{\sim} \frac{r_+^2}{2 \alpha} J(\gamma)^{1/2}.
\end{aligned}
\end{equation*}
Since $2I_2(\gamma) \leq I_1(\gamma)$, the last term in equation \eqref{eq:D7boundI1I2} is negligible and from the expression of $r_+$:
\begin{equation*}
    \int_{r_-}^{r_+} \left( \frac{1 + (e^{4 \gamma / \alpha^*}-1)x}{1+(J(\gamma)-1)x} \right)^2 w(x) \d x = \frac{e^{8 \gamma / \alpha^*}}{J(\gamma)^2} \left( \int_{r_-}^{r_+} w(x) \d x + O \left( e^{-4 \gamma / \alpha^*} J(\gamma)^{1/2} \right) + O \left( J(\gamma)^{-1/2} \right)  \right),
\end{equation*}
from \cref{lemma:overlap3}, $e^{-4 \gamma / \alpha^*} J(\gamma)^{1/2} = O \left( J(\gamma)^{-1/2} \right)$, which proves the result using equation \eqref{eq:D7AB}. 

\vspace{0.5 cm}

\section{Numerical Experiments} \label{App:Numerics}

All numerical experiments were carried on a professional laptop equipped with a NVIDIA GeForce GTX 1650. The code is written in Python and uses Pytorch \citep{pytorch} to run the gradient descent algorithm on GPU (for finite dimensional simulations). As for the numerical integration of the high-dimensional equations, a simple Euler method was used to approximate integrals and differential equations. 

\paragraph{Training details. } 
\begin{itemize}
    \item For gradient descent (Figures \ref{fig1}, \ref{fig2}, \ref{fig:CVrate}), we used a stepsize $\eta = 10^{-2}$ and initialized the weights $W^0$ from a Gaussian distribution (Figures \ref{fig2}, \ref{fig:CVrate}) or orthonormally (Figure \ref{fig1}), i.e $W^0 \overset{\mathrm{distrib}}{=} U(U^TU)^{-1/2}$ with $U$ having i.i.d Gaussian entries. 
    \item To simulate $\phi(\gamma)$ (Figures \ref{fig1}, \ref{fig2}) in the high-dimensional limit, we numerically solved equation \eqref{eq:selfphi} by interpreting it as a differential equation on the vector $(F_\phi(\gamma), G_\phi(\gamma))$, and using a stepsize $\eta = 2 \times 10^{-5}$.
    \item The asymptotic behaviour of the overlap (Figure \ref{fig:Overlap}) was simulated from equation \eqref{eq:C4limitxi} using a simple discretization method and an adapted 1D grid (which helped capturing the high variations of the integration measure at the edges) to compute the integrals (with $9 \times 10^4$ integration points) in the case $\alpha \neq \alpha^*$, and an analytic formula for those integrals in the special case $\alpha = \alpha^* = 1/2$, allowing for a reduction of the numerical error. 
\end{itemize}

\end{document}